\documentclass[review]{elsarticle}

\usepackage{lineno,hyperref}
\usepackage{mathrsfs}
\usepackage{enumitem}
\usepackage{verbatim}
\usepackage{graphicx}
\usepackage{tikz}
\usepackage{bbm}
\usepackage{extpfeil}
\usepackage{amsmath}
\usepackage{amsthm}
\usepackage{amssymb,url,color, booktabs,nccmath}
\usepackage{mhequ}
\usepackage{fancyhdr}
\usepackage{graphicx}
\usepackage{appendix}
\usepackage{cases}
\newtheorem{theorem*}{Theorem}
\newtheorem{theorem}{Theorem}[section]
\newtheorem{lemma}{Lemma}[section]

\newtheorem{proposition}{Proposition}[section]

\newtheorem{remark}{Remark}[section]

\newtheorem{Hypothesis}[theorem]{Hypothesis}
\newtheorem{definition}[theorem]{Definition}
\newtheorem{condition}{Condition}[section]
\newtheorem{example}{Example}
\newtheorem{corollary}{Corollary}[section]
\numberwithin{equation}{section}
\usepackage{geometry}
\allowdisplaybreaks
\begin{document}

\begin{frontmatter}

\title{\bf Large deviations principle for stationary solutions of stochastic  differential equations with multiplicative noise}

\author[AMSS,UCAS]{Peipei Gao}
\ead{peipeigao@amss.ac.cn}
\author[PKU]{Yong Liu}
\ead{liuyong@math.pku.edu.cn}

\author[AMSS,UCAS]{Yue Sun\corref{cor1}}
\ead{sunyue183@mails.ucas.ac.cn}
\cortext[cor1]{Corresponding author at: Academy of Mathematics and Systems Science, Chinese Academy of Sciences, Beijing
100190, China.}
\author[HN,AMSS,UCAS]{Zuohuan Zheng}
\ead{zhzheng@amt.ac.cn}
\address[AMSS]{Academy of Mathematics and Systems Science, Chinese Academy of Sciences, Beijing 100190, China}
\address[UCAS]{School of Mathematical Sciences, University of Chinese Academy of Sciences, Beijing 100049, China}
\address[PKU]{LMAM, School of Mathematical Sciences,
Peking University, Beijing 100871, China}
\address[HN]{College of Mathematics and Statistics, Hainan Normal University, Haikou, Hainan 571158, China}




\begin{abstract}
We study the large deviations principle (LDP) for stationary solutions of a class of stochastic differential equations (SDE) in infinite time intervals by the weak convergence approach, and then establish
the LDP for the invariant measures of the SDE by the contraction principle. We further point out the equivalence of
 the rate function of the LDP for invariant measures induced  by the LDP for stationary solutions and  the rate function  defined by quasi-potential. This fact gives another view of the quasi-potential  introduced  by Freidlin and Wentzell.

\end{abstract}


\end{frontmatter}

\tableofcontents

\section{Introduction}

\setcounter{equation}{0}
Real world phenomena are often effected by random perturbations or under the influences of noise, random dynamical systems arise to model those problems.
The stationary solution, random periodic solution, random quasi-periodic solution (see, for example \cite{SZZ08,LZ09,ZZ17,FZ12,FZ15,FZB11,FQZ21}) are natural extensions correspond to the fixed point, periodic solution, quasi-periodic solution from deterministic dynamical
systems to random dynamical systems, which are the fundamental concepts that describe the long time behavior of random dynamical systems.
In particular, there are many phenomena with small random perturbations in nature,
and small perturbations essentially influence the long time  behaviour of the system in general. Following Freidlin and Wentzell's perspective in \cite{FW12}, we use the LDP to characterize the long time asymptotic behavior of dynamical systems as the small random perturbation converges to zero. Therefore, it is interesting to investigate the LDP for stationary solutions, random periodic solutions and random quasi-periodic solutions in infinite time intervals with small
noises.

The LDP for solutions in infinite time intervals is also an interesting problem in mathematics. The LDP for Wiener process in infinite time intervals  is shown by  Deuschel and Stroock (cf. Schilder's theorem in Section 1.3 of \cite{DS84}), the proof  depends on the properties of the Gaussian measure. Although we can get the LDP for stationary solutions of Eq. \eqref{Msode0} with additive noise by using the contraction principle, it is not feasible for multiplicative noise. The exponential tightness for stationary solutions of Eq. (\ref{Msode0}) with multiplicative noise in the infinite intervals plays an important role in the proof of LDP if we follow the approach of Deuschle and Stroock \cite{DS84}, but it is hard to obtain the exponential tightness
for the lack of certain  accurate estimates, such as the Fernique theorem for Gaussian measure (cf. Theorem 1.3.24 in \cite{DS84}). Although, the proof of exponential tightness for non-Gaussian measures in the continuous function space is an interesting problem, unfortunately we have not found a suitable way to solve it by now.

Therefore, we use another well-known method, the weak convergence approach to prove the family of stationary solutions of Eq. \eqref{Msode0} satisfies the LDP as follows.
Noticing that the large deviations
principle is equivalent to the Laplace principle (LP) in Polish space, it is sufficient to prove the LP  by the Bou\'e-Dupuis formula (see, for example, \cite{BD98,BD01,DE97,Z09}). This method is often referred as the weak convergence approach. Using the weak convergence approach and the Bou\'e-Dupuis formula, the Freidlin-Wentzell type LDP for the family of solutions of SDE or SPDE driven by Wiener process in finite intervals has been re-proved (see, for example, \cite{FW12,SP06}).
However, considering the LDP in infinite intervals, we need the Bou\'e-Dupuis formula  in infinite intervals. Since the compactness of space and the integrability of bounded functions on infinite intervals are different from those on finite intervals, we have  to extend the Bou\'e-Dupuis formula from finite intervals to infinite intervals.
As far as we know, the Bou\'e-Dupuis formula in infinite intervals has been used directly without proof in \cite{BG20} by Barashkov and Gubinelli, we have to present the proof of Bou\'e-Dupuis
formula and the weak convergence approach in infinite intervals for the completeness of  the present paper in  Appendix A.
In fact, the most difficult problem in our proof is the well-posedness with
respect to the skeleton Eq.  \eqref{skeleton equation} in infinite time intervals, as we need to prove the uniqueness for the  backward infinite horizon  integral Eq. \eqref{SK-BIHE} for the skeleton Eq. \eqref{skeleton equation} and construct the solution of integral equation \eqref{SK-BIHE}. Besides, verifying the two conditions (cf.  Condition \ref{weakcondition} in Section \ref{Preliminary knowledge for LDP} below) of the weak convergence approach in infinite time intervals is also quite technical.

The main purpose of this paper is to establish the LDP of stationary solutions for a class of stochastic differential equations in infinite time intervals.
Moreover, since the one dimensional distribution of the stationary solution always generates an invariant measure of the SDE, here we give another illuminating view of the LDP for invariant
measures by the contraction principle (cf. Section 4.2.1 in \cite{DZ07}). In fact, the invariant measure is unable to provide the
properties for the trajectories of the dynamical systems,
and there exist many systems with different dynamical behaviors but with the same invariant measure (cf. Example \ref{EX same in} below). By comparison, the stationary solution show more accurate information about dynamical systems, and the LDP for them will give more explicit characterization for the long time asymptotic behavior of random dynamical systems with small perturbations.

As far as we know, Freidlin and Wentzell \cite{FW12},  Cerrai and R\"{o}ckner \cite{CR05}, and Brze\'{z}niak and Cerrai \cite{BC17}
studied the LDP for invariant measures by taking the quasi-potential as the rate function.
Moreover, it is not difficult to prove that the rate function \eqref{I'} defined in Theorem \ref{stoi} is equivalent to the rate function defined by quasi-potential, and the rate function \eqref{I'} gives an natural explanation of quasi-potential.
Roughly speaking, Freidlin and Wentzell \cite{FW12},  Cerrai and R\"{o}ckner \cite{CR05}, and Brze\'{z}niak and Cerrai \cite{BC17}
studied the LDP for invariant measures
by using the result of the LDP of solutions in finite intervals and let the length of time intervals tends to infinity. It is different to our method that  we directly consider the asymptotic
behavior of stationary solutions for
random dynamical systems  in infinite intervals as perturbation converges to zero, and by using the contraction principle to get the LDP for invariant measures. Besides, we believe that our
methods to show the LDP for stationary solutions is also available to explore the LDP for random periodic solutions and random quasi-periodic solutions.  We give two examples (cf. Example \ref{perio} and \ref{quasi-periodic solution} below) to show the random periodic solutions of Eq. (\ref{periodic eq}) (see, for example, Feng, Liu and Zhao \cite{FZ15}) and random quasi-periodic solutions (see, for example, Hopf \cite{H56}) of equation \eqref{hopf} satisfy the LDP, and we will research the LDP for random periodic solutions and random quasi-periodic solutions for others equation in future papers.

In this paper, we consider the stochastic differential equation as follows
\begin{equation}\label{Msode0}
	\mathrm{d} X_{\varepsilon}=AX_{\varepsilon} \mathrm{d} t+F(X_{\varepsilon})\mathrm{d} t+\sqrt{\varepsilon}B(X_{\varepsilon}) \mathrm{d} W_t.
\end{equation}
All definitions of the symbols involved can be found later in Section \ref{notations}.

Liu and Zhao give the representation of the stationary solution for Burgers equation with large viscosity in \cite{LZ09}.
Under our Hypothesis \ref{Msodehy}, similar to the proof of \cite{M99} and \cite{LZ09}, we have proved that
there exists $\varepsilon_0>0$ such that for every $\varepsilon\in(0,\varepsilon_0)$, the stochastic equation \eqref{Msode0} exists a unique  stationary solution $X^{*}_{\varepsilon}(t,\omega)$ satisfies the following equation in $H$ for any $t \in \mathbb{R}$
\begin{equation}\label{MIHSIE}
	X_{\varepsilon}^{*}(t)=\int_{-\infty}^{t} S_{A}(t-s)F\left( X_{\varepsilon}^{*}\right) \mathrm{d} s+\sqrt{\varepsilon}\int_{-\infty}^{t} S_{A}(t-s) B \left( X_{\varepsilon}^{*}\right)\mathrm{d} W_{s},
\end{equation}
where $H$ be a separable Hilbert space, which is defined in Section \ref{notations}.

The first major conclusion as follows.
\begin{theorem*}(cf. Theorem \ref{stationary LDP} below)
	If the operators $A$, $F$ and $B$ satisfy Hypothesis \ref{Msodehy}, then
	the family of stationary solutions $\{X^{*}_\varepsilon:\varepsilon>0\}$ satisfies the LDP in $C(\mathbb{R};H)$ with good rate function $I$ given in (\ref{ratefunction}).
\end{theorem*}
We could get the following result for the family of invariant measures of  Eq. \eqref{Msode0} satisfies the LDP by using the contraction principle .
\begin{theorem*}(cf. Theorem \ref{stoi} below)
	If the operators $A$, $F$ and $B$ satisfy Hypothesis \ref{Msodehy}, then the family of invariant measures $\left\{\nu_{\varepsilon}\right\}_{\varepsilon>0}$ for Eq. \eqref{Msode0} satisfies the LDP in $H$, with good rate function given in \eqref{I'}.
\end{theorem*}
 Noticing that under  Hypothesis \ref{Msodehy}, we actually only deal with the case where Eq. \eqref{Msode0} has only one asymptotically stable fixed point for $\varepsilon=0$. And even the stochastic Burgers equation and $2$-dimensional Navier-Stokes on tours are all satisfies Hypothesis \ref{Msodehy} (cf. Example \ref{burgers equation}, \ref{Navier-Stokes} below), but
there are also many types of stochastic partial differential equations that do not satisfy Hypothesis \ref{Msodehy}, such as the stochastic wave equations.
Therefore, there are many more complex and interesting cases that we might follow, such as the stochastic wave equations, random dynamical systems with random periodic solutions, random quasi-periodic solutions. In particular, the cases that the deterministic dynamical systems for $\varepsilon=0$ have multiple equilibrium points (see, for example, \cite{DM15}) is also an attractive problem. Furthermore, the existence of invariant measure does not imply the existence of stationary solution, unless
we consider in an extended probability space (see, for example, Arnold \cite{LA98}). The asymptotic behavior for the stationary solutions in the extended probability space is also an interesting problem.

This paper is structured as follows.
In Section \ref{sec-2}, we present the basic notations and some definitions for the stationary solution, LDP and quasi-potential. We give Hypothesis \ref{Msodehy} which the operators $A$, $F$ and $B$ satisfy throughout the paper and the definition for stationary solution in Section \ref{notations},  the preliminary knowledge for LDP in Section  \ref{Preliminary knowledge for LDP} and the definition of quasi-potential in Section \ref{The definition of quasi-potential}.

In Section \ref{LDP for stationary solution}, we prove the family of stationary solutions of Eq. \eqref{Msode0} satisfies LDP. We use the weak convergence approach to prove the LDP for stationary solution of Eq. \eqref{Msode0}, and verify that the Burger's equation satisfies the Hypothesis \ref{Msodehy} in Example \ref{burgers equation}. Moreover, we state two examples to show the LDP for random periodic solutions and random quasi-periodic solutions in Example \ref{perio} and Example \ref{quasi-periodic solution}.

In Section \ref{LDP for invariant measure}, we prove the LDP for invariant measures of Eq. \eqref{Msode0} in Theorem \ref{stoi} and prove the rate function defined in Theorem \ref{stoi} is equivalent to the rate function \eqref{V} defined by quasi-potential in Lemma \ref{quasi rate}. Moreover, we give two examples to show that the rate function \eqref{I'} defined in Theorem \ref{stoi} of the invariant measures is consistent with quasi-potential (see, for example, \cite{BC17,CR05,FW12}), and the LDP of the stationary solutions can provide more dynamic information.

In Section \ref{Supplementary proofs of skeleton equation and stationary solution},
we consider the well-posedness of the skeleton equation in Section \ref{The well-posedness of the skeleton equation}, and prove there exists a $\varepsilon_0$ such that for every $\varepsilon\in(0,\varepsilon_0)$, Eq. \eqref{Msode0} exists a unique stationary solution in Section \ref{Stationary solution}.

In Appendix A, we prove the Bou\'e-Dupuis formula and the weak convergence approach in infinite intervals.
\section{Preliminaries}\label{sec-2}
\subsection{Notations}\label{notations}
Let $V$ be a reflexive Banach space, and $V^{*}$ be the dual space of $V$. Let $H$ be a separable Hilbert space with inner product $\langle\cdot,\cdot\rangle_{H}$, $H$ and its dual space $H^{*}$ are consistent by the Riesz isomorphism. Let $V \subset H$ continuously and densely, then
$\left(V, H, V^{*}\right)$ is called a Gelfand triple.

Let
$L(H)$ be the space of all bounded linear operators from $H$ to $H$ with the operator norm $|\cdot|_{L}$.
Let $L_2(H)$ be the space of all Hilbert-Schmidt operators from $H$ to $H$ with Hilbert-Schmidt norm $|T|_{L_2}:=\sum_{k\in\mathbb{N}}
\left\|Ta_k\right\|_{H}^2$, where $\left\{a_k\right\}_{k\in\mathbb{N}}$ is an orthogonal basis of $H$.
Let $Q$ be a trace class operator, $L_{Q}(H)$ is denoted by the space of linear operators $T$ such that $T Q^{1 / 2}$ is a Hilbert-Schmidt operator from $H$ to $H$, with the norm $|T|_{L_{Q}}:=|TQ^{1/2}|_{L_2}$.

Let $H_{0}=Q^{1 / 2} H$. Then $H_{0}$ is a Hilbert space with the inner product
$$\langle u, v\rangle_{H_0}=\langle Q^{-1 / 2} u, Q^{-1 / 2} v\rangle _H ,\quad\forall u,v \in H_{0} .$$
Clearly, the imbedding of $H_{0}$ in $H$ is Hilbert-Schmidt for $Q$ is a trace class operator.

Let $\{W(t)\}_{t\in\mathbb{R}}$ be a $Q$-Wiener process on $H$ with respect to a probability space $(\Omega, \mathcal{F}^0, \mathbb{P})$, where $\mathcal{F}^0$  is  Borel
$\sigma$-field of $\Omega$, $\mathbb{P}$ is the Wiener measure on $\Omega$. Let $\mathcal{F}$ be the $\mathbb{P}$-completion of $\mathcal{F}^0$. Define
\begin{equation*}
	\mathcal{F}_{s}^{t} \equiv \sigma\left\{W_{u}-W_{v}: s \leq u, v \leq t\right\} \vee \mathcal{N},\quad
	\mathcal{F}_{-\infty}^{t} \equiv \bigvee_{s \leq t} \mathcal{F}_{s}^{t}, \quad\forall s\leq t \in \mathbb{R},
\end{equation*}
where $\mathcal{N}$ are the null sets of $\mathcal{F}$ (see, for example, \cite{LZ09}).

We consider the stochastic differential equation
\begin{equation}\label{Msode}
	\mathrm{d} X=AX \mathrm{d} t+F(X)\mathrm{d} t+\sqrt{\varepsilon} B (X) \mathrm{d} W_t,
\end{equation}
where
$
A: V \rightarrow V^{*}$, $F:V  \rightarrow H$,  $B : V  \rightarrow L(H)$
be progressively measurable
and satisfies the following Hypothesis.
\begin{Hypothesis}\label{Msodehy}
	(i). Assume there exist constants $\lambda, C_0>0$ such that
	\begin{equation*}
		_{V^*}\langle A\textbf{u}-A\textbf{v}+F(\textbf{u})-F(\textbf{v}),\textbf{u}-\textbf{v} \rangle_V
		\leq -\lambda \left\|\textbf{u}-\textbf{v}\right\|_{V}^2 +C_0\left\|\textbf{u}-\textbf{v}\right\|_{H}^2\left\|\textbf{u}\right\|_{V}^2,
	\end{equation*}
	and
	\begin{equation*}
		_{V^*}\langle AX+F(\textbf{u}),\textbf{u} \rangle_V
		\leq -\lambda \left\|\textbf{u}\right\|_{V}^2,\quad \forall \textbf{u},\textbf{v}\in V.
	\end{equation*}
	(ii).
	Let  $S_{A}(t)$ be the associated semigroup on H corresponding to $A$, for any $N\in\mathbb{N}^+$, $t_0,t\in[-N,N]$ and $t_0<t$, there exists constant $C>0$ and $-1<\alpha<0$ such that the semi-group satisfies
	\begin{equation*}
		\Big\|\int_{t_0}^{t}S_{A}(s)(F(\textbf{u})-F(\textbf{v}))\mathrm{d}s\Big\|_{H}^2\leq C (N) \int_{t_0}^{t} s^{\alpha}\|\textbf{u}-\textbf{v}\|_{H}^2\mathrm{d}s.
	\end{equation*}
	(iii). The function $ B $ will not be zero operator  and  there exist positive  constants $\beta_0,D_0$ such that
	\begin{equation*}
		| B (\textbf{u})- B (\textbf{v})|_{L}\leq \beta_0 \left\|\textbf{u}-\textbf{v}\right\|_{H},\quad\forall \textbf{u},\textbf{v}\in V,
	\end{equation*}
	and
	\begin{equation*}
		| B (\textbf{u})|_{L}\leq D_0, \quad\forall \textbf{u}\in V.
	\end{equation*}
\end{Hypothesis}
\begin{remark}\label{condition remark}
	(i). Since $V$ imbedes to $H$, there exists a constant $C_1$ such that $C_1\|\textbf{u}\|_{H}^2\leq \|\textbf{u}\|_V^2,\forall \textbf{u}\in V.$\\
	(ii). Let $\beta=\beta_0trQ$ and $D=D_0trQ$. It follows from
  $Q$ is a trace class operator and combining Hypothesis \ref{Msodehy} (iii), we could obtain that
	
\begin{align*}
		| B (\textbf{u})- B (\textbf{v})|_{L_Q}
		=|( B (\textbf{u})- B (\textbf{v}))Q^{1/2}|_{L_2}
\leq |( B (\textbf{u})- B (\textbf{v}))|_{L} |Q^{1/2}|_{L_2}\leq\beta \left\|\textbf{u}-\textbf{v}\right\|_{H},
	\end{align*}
	and
	\begin{equation*}
		| B (\textbf{u})|_{L_Q}\leq | B (\textbf{u})Q^{1/2}|_{L_2}\leq  | B (\textbf{u})|_{L}|Q^{1/2}|_{L_2} \leq D,\quad \forall \textbf{u},\textbf{v}\in V.
	\end{equation*}
	(iii). Let $\textbf{u}(t):=S_A(t)\textbf{u}_0$ be the solution of the equation $\dot{\textbf{u}}=A\textbf{u}$ with initial value $\textbf{u}_0$, where $\dot{\textbf{u}}$  be the deririaive of $\textbf{u}$ with respect to time $t$, then we have the estimates
	\begin{equation*}
		\frac{1}{2}\frac{\mathrm{d}\|\textbf{u}\|_{H}^2}{\mathrm{d}t}=_V\langle \textbf{u},A\textbf{u}\rangle_{V^{\ast}}
		\leq-\lambda\|\textbf{u}\|_{V}^2,
	\end{equation*}
	then $\frac{1}{2}\|\textbf{u}(t)\|_{H}^2+\lambda\int_0^t\|\textbf{u}\|_{V}^2\mathrm{d}s\leq\frac{1}{2}\|\textbf{u}_0\|_{H}^2$, thus
	$$\|S_A(t)\textbf{u}_0\|_{H}^2\leq \|\textbf{u}_0\|_{H}^2,\quad\int_0^{\infty}\|S_A(s)\textbf{u}_0\|_{H}^2\mathrm{d}s\leq \frac{1}{C_1}\int_0^{\infty}\|S_A(s)\textbf{u}_0\|_{V}^2\mathrm{d}s\leq\frac{1}{2\lambda C_1}\|\textbf{u}_0\|_{H}^2.$$
	For any $\textbf{u}_0\in H,\|\textbf{u}_0\|_{H}=1$, it implies that
	\begin{equation}
		|S_A(t)|_{L}\leq 1,\quad\lim_{t\rightarrow\infty}|S_A(t)|_{L}=0.
	\end{equation}
\end{remark}
\begin{definition}\label{mild solution}
	(Mild solution). For any $\varepsilon,T>0$, $\xi\in H$, an $H$-valued predictable process $X_{\varepsilon}(t,\omega;\xi)$, $t \in[0, T]$, is called a mild solution of Eq. \eqref{Msode} with initial value $\xi$ if
$$
		\begin{aligned}
			X_{\varepsilon}(t,\omega;\xi)=&S_A(t) \xi +\int_{0}^{t} S_A(t-s)F(X_{\varepsilon}(s,\omega;\xi)) \mathrm{d} s
		+\sqrt{\varepsilon}\int_{0}^{t} S_A(t-s) B (X_{\varepsilon}(s,\omega;\xi)) \mathrm{d} W_s(\omega) \quad \mathbb{P} \text {-a.s. }
		\end{aligned}
$$
	for each $t \in[0, T]$. In particular, the appearing integrals are well defined.
\end{definition}
\begin{theorem}
	For any $\varepsilon,T>0$, $\xi\in H$, there exists a unique mild solution $X_{\varepsilon}\in C([0,T];H), a.s.$ of Eq. \eqref{Msode} in the sense of definition \ref{mild solution}.
\end{theorem}
\begin{proof}
	It is obvious that the Hypothesis \ref{Msodehy} satisfies the Hypothesis $H1$ - $H4$ in Chapter
	4 of \cite{LW15}, which  guarantees that the variational solution of Eq. \eqref{Msode}  exists and is unique. Combining the Appendix $F$ in \cite{PR07}, we get the conclusion.
\end{proof}
To define the stationary solution of random dynamical system, let $(\Omega, \mathcal{F}, \mathbb{P})$ be a complete probability space, $\theta: \mathbb{R} \times \Omega \rightarrow \Omega$ be a group of $\mathbb{P}$-preserving ergodic transformations on $(\Omega, \mathcal{F}, \mathbb{P})$.

\begin{definition}(cf. \cite{LA98})
	A crude cocycle $(U, \theta)$ on $H$ is a $\left(\mathcal{B}\left(\mathbb{R}^{+}\right) \otimes\right.$ $\mathcal{B}(H) \otimes \mathcal{F}, \mathcal{B}(H))$-measurable random field $U: \mathbb{R}^{+} \times H \times \Omega \rightarrow H$ with the following properties:\\
	(i). $U\left(t_{1}+t_{2}, \cdot, \omega\right)=U\left(t_{2}, U\left(t_{1}, \cdot, \omega\right), \theta\left(t_{1}, \omega\right)\right)  $ for fixed $t_{1}$, and all $t_2$, $\mathbb{P}$-a.s. (where the
exceptional set $\mathcal{N}_{t_1}$ can depend on $t_1$).\\
	(ii). $U(0, \xi, \omega)=\xi$ for all $\xi \in H$, $\omega \in \Omega$.
\end{definition}
\begin{definition}\label{Defition of SS}
	An $\mathcal{F}$ measurable random variable $\textbf{v}: \Omega \rightarrow H$ is said to be a stationary solution for the crude  cocycle $(U, \theta)$ if it satisfies for any $r$
	\begin{equation}\label{defition of SS}
		U(\cdot, Y(r,\omega), \theta(r, \omega))=Y(\cdot+r,\omega)=Y(\cdot,\theta(r, \omega)),\quad a.s.
	\end{equation}
\end{definition}
\begin{remark}
	We denote by $\theta:\mathbb{R}\times \Omega\rightarrow \Omega$ be the standard $\mathbb{P}$-preserving ergodic Wiener shift on $\Omega$, $\theta(t,\omega(s)):=\omega(t+s)-\omega(t)$, $t,s\in\mathbb{R}$.  Noticing that the stationary solution defined in Definition \ref{Defition of SS} means that for any $r$, Eq. \eqref{defition of SS} holds almost surely, which is slightly different from the definition in \cite{LZ09,SZZ08} needs Eq. \eqref{defition of SS} holds for all $r$ for all $\omega\in \Omega$ with respect to the perfect cocycle (cf. \cite{LA98,LZ09,SZZ08}). We prove the unique solution of the pullback integral Eq. \eqref{MIHSIE} be the stationary solution (in the sense of Definition \ref{Defition of SS}), and the distribution of the solution for \eqref{MIHSIE} generates an invariant measure.
Our results about the LDP are independent of the perfect cocycle, since the LDP describes the properties of the  distribution of stochastic processes, but not the pathwise properties.     Even the well-posedness of the pullback infinite horizon  integral equation \eqref{MIHSIE} is enough to show the LDP  and depicts the long time asymptotic behavior of random dynamical systems with small perturbations.,

\end{remark}
Liu and Zhao \cite{LZ09} have proved the stationary solution for Burgers equation with additive noise by using the proof of Mattingly \cite{M99}. However, we consider Eq. \eqref{Msode} with multiplicative noise, we prove the following theorem  for the completeness of  the present paper in Section \ref{Stationary solution}.
\begin{theorem}\label{Mstationary solution}
	There exists $\varepsilon_0>0$, for every $\varepsilon\in(0,\varepsilon_0)$, Eq. \eqref{Msode} exists a unique  stationary solution (in the sense of Definition \ref{Defition of SS}) $X^{*}_{\varepsilon}(t,\omega) (t \in \mathbb{R})$ be a $(\mathcal{B}(\mathbb{R})\otimes \mathcal{F},\mathcal{B}(H))$-measurable, $(\mathcal{F}_{-\infty}^{t})_{t\in \mathbb{R}}$-adapted process. Moreover, $X_{\varepsilon}^{*}(t,\omega)$ satisfies the following equation in $H$ for any $t \in \mathbb{R}$
	\begin{equation}\label{MIHSIE}
		X_{\varepsilon}^{*}(t)=\int_{-\infty}^{t} S_{A}(t-s)F\left( X_{\varepsilon}^{*}\right) \mathrm{d} s+\sqrt{\varepsilon}\int_{-\infty}^{t} S_{A}(t-s) B \left( X_{\varepsilon}^{*}\right)\mathrm{d} W_{s},
	\end{equation}
	and
	\begin{equation}\label{uniqueness proof condition}
		\sup_{t\in \mathbb{R}}\mathbb{E}\left\|X_{\varepsilon}^{*}(t)\right\|_{H}^{2} < \infty.
	\end{equation}
\end{theorem}
\subsection{Preliminary knowledge for LDP}\label{Preliminary knowledge for LDP}
We recall some standard definitions and results of the large deviations theory (cf. \cite{BD98,DE97,DZ07}). Let $E$ be a Polish space, $\{X^{*}_\varepsilon\}$ be a family
of $E$-valued random variables defined on a probability space $(\Omega, \mathcal{F},\mathbb{P})$.
The large deviation theory concerns the exponential decay of the probability measures
of rare events. The rate of such exponential decay is expressed by the rate
function.
\begin{definition}(Rate function).
	A function $I$ mapping $E$ to $[0,+\infty)$ is called a rate function if $I$ is lower semi-continuous. A rate function $I$ is called a good rate function if for each $M<\infty$, the level set $\{x\in E:I(x)\leq M\}$ is compact.
\end{definition}
\begin{definition}(Large deviation principle).
	The sequence $\{X^{*}_\varepsilon\}$ is said to satisfies the LDP with rate function $I$ if for each Borel subset $A$ of $E$,
	\begin{equation*}
		-\inf_{x\in A^{\circ}}I(x)\leq\liminf_{\varepsilon\longrightarrow0}\varepsilon\log \mathbb{P}(X^{*}_\varepsilon\in A)\leq\limsup_{\varepsilon\longrightarrow0}\varepsilon\log \mathbb{P}(X^{*}_\varepsilon\in A)\leq-\inf_{x\in \bar{A}}I(x),
	\end{equation*}
	where $A^{\circ}$ and $\bar{A}$ are the interior and the closure of $A$ in $E$, respectively.
\end{definition}
Since we can not prove the exponential tightness for stationary solutions of Eq. \eqref{Msode} with multiplicative noise in infinite intervals, which  extremely depends on the properties of the Gaussian measure.
We choose the weak convergence method to study the LDP of stationary solution for Eq. \eqref{Msode}. Some fundamental concepts and results
about the weak convergence method are stated as follows.
\begin{definition}(Laplace principle).
	The sequence $\{X_\varepsilon\}$ is said to satisfy the Laplace principle with a
	rate function $I$ if for each bounded continuous real-valued function $h$ defined on $E$, we have
	\begin{equation}\label{LPeq}	\lim_{\varepsilon\rightarrow0}\varepsilon\log\mathbb{E}
		\Big\{\sup\Big[-\frac{1}{\varepsilon}h(X_\varepsilon)\Big]\Big\}=-\inf_{x\in E}\{h(x)+I(x)\}.
	\end{equation}
\end{definition}

We consider the stationary solution $X^{*}_{\varepsilon}$ of Eq. \eqref{Msode} satisfies LDP in the Polish space $E:=C(\mathbb{R};H)$ with the norm
\begin{equation}\label{Cnorm}
	\|f\|_{C(\mathbb{R};H)}:=
	\|f\|_{\mathcal{C}}=
	\sum_{k=1}^{\infty}2^{-k}\left(\sup_{s\in[-k,k]}\left\|
	f(s)\right\|_{H}\wedge1\right), \quad  f\in C(\mathbb{R};H).
\end{equation}
According to Theorems 1.2.1 and 1.2.3 in Dupuis and
Ellis \cite{DE97}, if $E$ is a Polish space and $I$ is a good rate function, then the LDP and Laplace
principle are equivalent, which is the basis for the weak convergence approach.  Moreover, we need the Bou\'e-Dupuis formula in infinite intervals to prove the weak convergence approach in infinite intervals.
Although the Bou\'e-Dupuis formula in infinite intervals has been used in \cite{BG20} by Barashkov and Gubinelli, for the completeness of the article, we still give the proof of Bou\'e-Dupuis formula and the weak convergence approach in infinite intervals in Appendix A, and illustrate the results in the following.\\
Let
$$
\begin{aligned}
\mathcal{A}=&\left\{v:\quad v \text{~is the}~ H_0\text{-valued} ~ \mathcal{F}_t\text{-predictable process and~}\int_{-\infty}^{+\infty}\|v(s)\|_{H_0}^2\mathrm{d}s<\infty~a.s. \right\},
\end{aligned}
$$
and
\begin{align*}
	S_{M}:&=\left\{v\in L^2(\mathbb{R};H_0):\int_{-\infty}^{\infty}
	\|v(s)\|_{H_0}^2\mathrm{d}s\leq M\right\}.
\end{align*}
It is similar to prove the set $S_M$ endowed with the weak topology is a Polish space (cf. \cite{KF57}, Theorem III.1'). In this paper, except for special instructions, the topology of $S_M$ is always weak topology.
We also define
$$\mathcal{A}_{M}:=\Big\{v\in\mathcal{A},\quad v(\omega)\in S_M,\quad \mathbb{P}\text{-}a.s.\Big\}.$$

We provide the sufficient condition for the Laplace principle (equivalently, the LDP), which is similar to Assumption 4.3 in \cite{BD01}, in the follwoing.
\begin{condition}\label{weakcondition}
	There exists a measurable map  $\ \mathcal{G}^0:C(\mathbb{R};H_0)\longrightarrow E$ such that the following hold:\\
	(i). Let $\{v^\varepsilon:\varepsilon>0\}\subset\mathcal{A}_M$ for some $M<\infty$. If $v^{\varepsilon}$ converges  to $v$ in distribution as $S_M$-valued random elements, then $\mathcal{G}^\varepsilon(W(\cdot)+\frac{1}	{\sqrt{\varepsilon}}\int_{-\infty}^{\cdot}v^{\varepsilon}\mathrm{d}s)$ converges to $\mathcal{G}^0(\int_{-\infty}^{\cdot}v(s)\mathrm{d}s)$ in distribution as $\varepsilon\rightarrow 0$.\\
	(ii). For every $M<\infty$, the set $K_M=\{\mathcal{G}^0(\int_{-\infty}^{\cdot}v(s)\mathrm{d}s):v\in S_M\}$ is a compact subset of $E$.
\end{condition}
Similar to the proof of Budhiraja and Dupuis in \cite{BD98}, we prove the following result in  Appendix \ref{LP proof}.
\begin{theorem}\label{LP}
	If $X_\varepsilon=\mathcal{G}^\varepsilon(W(\cdot))$  satisfies Condition \ref{weakcondition}, then the family $\{X_\varepsilon:\varepsilon>0\}$ satisfies the Laplace principle in $E$ with good rate function
	\begin{equation}\label{1.15}
		I(f)=\inf_{\left\{v\in L^2(\mathbb{R};H_0):f			=\mathcal{G}^0(\int_{-\infty}^{\cdot}v(s)\mathrm{d}s)\right\}}		\left\{\frac{1}{2}\int_{-\infty}^{+\infty}\|v(s)\|_{H_0}^2\mathrm{d}s\right\}.
	\end{equation}
	where the infimum over an empty set is taken as $+\infty$.
\end{theorem}
In order to verify  Condition \ref{weakcondition} for the stationary solution $X^{*}_{\varepsilon}$ of Eq. \eqref{Msode} in space $C(\mathbb{R};H)$, we need to consider  $\mathcal{G}^0$, which is defined as the unique solution of the skeleton equation \eqref{skeleton equation}. Furthermore, we need to prove the well-posedness with
respect to Eq. \eqref{skeleton equation} in infinite time intervals.
\subsection{Definition of the quasi-potential}\label{The definition of quasi-potential}
It is well known that \cite{FW12} and \cite{CR05} have proved the family of invariant measures for stochastic equations satisfies LDP  with  the quasi-potential as rate function. We  give the definition of quasi-potential below, and compare
the rate function of the LDP for invariant measures induced by the LDP for stationary
solutions and the rate function defined by quasi-potential in Section \ref{LDP for invariant measure}.

For any $-\infty\leq t_1<t_2\leq +\infty$ and $v\in L^2(t_1,t_2;H_0)$, we denote by $u^x_{t_1}(v)$ any solution belonging to $C(t_1,t_2;H)$ of the control equation
\begin{equation}\label{ce}
	\frac{\mathrm{d}u}{\mathrm{d}t}=Au+F(u)+ B (u)v,\quad u(t_1)=x\in H.
\end{equation}
And we define the action functionals by
\begin{equation*}
	S_{t_{0}, t_{1}}(u):=
	\frac{1}{2}\inf\left\{ \int_{t_{0}}^{t_{1}}\left\| v(t)\right\|_{H_0}^{2} \mathrm{d} t;\quad u=u(v)\right\},
\end{equation*}
where $u(v)$ is the solution of Eq. \eqref{ce} in the intervals $[t_1,t_2]$ corresponding to the control $v$, and $\inf\varnothing=+\infty.$

Moreover, we denote
$$
S_{-T}:=S_{-T, 0}, \quad S_{T}:=S_{0, T}, \quad \text { for every } T>0.
$$
In particular, when $t_0=-\infty$ and $t_1=0$, we set
$$
S_{-\infty}(u):=\frac{1}{2} \inf \left\{ \int_{-\infty}^{0}\left\| v(t)\right\|_{H_0}^{2} \mathrm{d} t;\quad u=u(v)\right\}.
$$

Similar to \cite{FW12} and \cite{CR05}, we define the quasi-potential $V$ associated with Eq. \eqref{Msode}, by setting
$$
\begin{aligned}
	V(x) &:=\inf \left\{S_{T}(u): T>0, u \in C([0, T] ; H),u(0)=0,u(T)=x\right\} \\
	&=\inf \left\{S_{-T}(u): T>0, u \in C([-T, 0] ; H),u(-T)=0,u(0)=x\right\}, \quad x \in H.
\end{aligned}
$$
It follows from the Proposition 5.4 in \cite{CR05}, which is also mentioned in Chapter 4 of \cite{FW12}, quasi-potential $V$ has a good characterization as follows.
\begin{equation}\label{V}
	V(x)=\inf \Big\{S_{-\infty}(u):u \in C((-\infty, 0] ; H),u(0)=x,\lim_{t\rightarrow -\infty }\|u(t)\|_{H}=0\Big\},
\end{equation}
for all $x\in H$, $V(x)<\infty.$

We will prove that the rate function \eqref{I'} defined in Theorem \ref{stoi} is equivalent to the rate function \eqref{V} defined by quasi-potential for invariant measures.

\section{Large deviation principle for stationary solution}\label{LDP for stationary solution}
\subsection{The proof of LDP for Large deviation principle for stationary solution}
We verify that the stationary solution $X^{*}_{\varepsilon}$ of Eq. \eqref{Msode} satisfies  Condition \ref{weakcondition}, which is a sufficient condition to  $X^{*}_{\varepsilon}$ satisfies LDP in this section.

To define $\mathcal{G}^0$ in Condition \ref{weakcondition}, we consider the following control equation
\begin{equation}\label{skeleton equation}
	\frac{\mathrm{d}X(t)}{\mathrm{d}t} = AX(t)+F(X(t))+  B  (X)v.
\end{equation}
We illustrate the following result for
the skeleton equation \eqref{skeleton equation}, and
Theorem \ref{SE IF TH} will be proved in Section \ref{The well-posedness of the skeleton equation}.
\begin{theorem}\label{SE IF TH}
	For  some finite $M>0$ and $v\in S_M$, under Hypothesis \ref{Msodehy}, there exists a unique  solution $X^*_v$ of the backward infinite horizon integral equation
	\begin{equation}\label{SK-BIHE0}
		X^*_v(t,\omega) = \int_{-\infty}^{t}S_A(t-r)F(X^*_v(r,\omega))\mathrm{d}r + \int_{-\infty}^{t}S_{A}(t-r) B (X^*_v(r,\omega))v(r)\mathrm{d}r,
	\end{equation}
	and satisfies
	\begin{equation}\label{SK-uniqueness proof condition0}
		\sup_{t\in \mathbb{R}}\left\|X^*_v(t)\right\|_{H}^{2} < \infty.
	\end{equation}
\end{theorem}
Therefore, we could define the measurable map $\mathcal{G}^0:C(\mathbb{R};H_0)\longrightarrow C(\mathbb{R};H)$ by
\begin{equation*}
	\mathcal{G}^{0}\Big(\int_{-\infty}^{\cdot}v(s)\mathrm{d}s\Big):=X_v^{*}(\cdot),
\end{equation*}
where $X_v^{*}$ be the unique solution of Eq. \eqref{SK-BIHE0} with control term $v$ and satisfies \eqref{SK-uniqueness proof condition0}.

We consider the LDP for the family of stationary solutions of Eq. \eqref{Msode}
in space $C(\mathbb{R};H)$, which norm is defined by (\ref{Cnorm}).
According to Theorem \ref{Mstationary solution}, we have proved that there exists $\varepsilon_0>0$ such that for any $\varepsilon\in(0,\varepsilon_0)$, there exists a unique stationary solution $X^{*}_{\varepsilon}$ of Eq. \eqref{Msode} and satisfies the integral equation
\begin{equation}\label{backward infinite horizon stochastic intergal equation'}
	X^*_{\varepsilon}(t,\omega) = \int_{-\infty}^{t}S_{A}(t-r)F(X^*_{\varepsilon}(r,\omega))\mathrm{d}r + \sqrt{\varepsilon} \int_{-\infty}^{t}S_{A}(t-r) B \left(X^*_{\varepsilon}(r,\omega)\right)\mathrm{d}W(r).
\end{equation}
Then we could define $\mathcal{G}^{\varepsilon}:C(\mathbb{R};H_0)\rightarrow C(\mathbb{R};H)$ by $\mathcal{G}^{\varepsilon}(W(\cdot)):=X_{\varepsilon}^{*}(\cdot)$.
We define $X_{\varepsilon,v_{\varepsilon}}^*$ by $$\mathcal{G}^{\varepsilon}\left(W(\cdot)+\frac{1}{\sqrt{\varepsilon}}
\int_{-\infty}^{\cdot}v_{\varepsilon}(s)\mathrm{d}s\right),$$  with the help of the Girsanov transform, $X_{\varepsilon,v_{\varepsilon}}^*$ be the unique stationary solution of the control equation
\begin{equation}\label{control}
	\mathrm{d}X_{\varepsilon,v_{\varepsilon}}^*=A X_{\varepsilon,v_{\varepsilon}}^*(t,x) \mathrm{d}t +F(X_{\varepsilon,v_{\varepsilon}}^*(t,x))\mathrm{d}t + \sqrt{\varepsilon} B  (X_{\varepsilon,v_{\varepsilon}}^*)\mathrm{d}W(t)
	+ B (X_{\varepsilon,v_{\varepsilon}}^*)v_{\varepsilon}\mathrm{d}t.
\end{equation}

Before we prove the stationary solution $X^{*}_{\varepsilon}$ of Eq. \eqref{Msode} satisfies  Condition \ref{weakcondition}, we first give the following Lemma, which will be used in the proof of Lemma \ref{verify weak condition 1}.
\begin{lemma}\label{prior Xev}
	If $\{v^{\varepsilon}:\varepsilon>0\}\subset\mathcal{A}_M$, for some $M<\infty$, then for every fixed $k\in\mathbb{N}^+$, there exists $\varepsilon_0,\eta>0$ and a constant $C$, which depend only on $M,\varepsilon_0,k,D$ such that the solution $X_{\varepsilon,v_{\varepsilon}}^*$ of Eq. \eqref{control} satisfies
	\begin{equation*}
		\mathbb{E}\sup_{t\in(-\infty,k)}e^{2\eta t} \|X_{\varepsilon,v_{\varepsilon}}^*\|_{H}^{2}\leq C,\quad\forall \varepsilon\in(0,\varepsilon_0),
	\end{equation*}
	where $D$ is defined by Remark \ref{condition remark} (ii).
\end{lemma}
\begin{proof}
	By using the It\^{o} formula, Hypothesis \ref{Msodehy} (i) , Remark \ref{condition remark} (ii) and Young inequality, we get
	\begin{align*}
		&e^{2\eta t} \|X_{\varepsilon,v_{\varepsilon}}^*\|_{H}^{2}\\
		=&2\eta\int_{-\infty}^{t}e^{2\eta s} \|X_{\varepsilon,v_{\varepsilon}}^*\|_{H}^{2}\mathrm{d}s+2\int_{-\infty}^{t} e^{2\eta s}   {_{V^*}}\left\langle AX_{\varepsilon,v_{\varepsilon}}^*+ F(X_{\varepsilon,v_{\varepsilon}}^*), X_{\varepsilon,v_{\varepsilon}}^*(s)\right\rangle \mathrm{d} s\\
		&+2\int_{-\infty}^{t}e^{2\eta s}\left\langle B \left( X_{\varepsilon,v_{\varepsilon}}^*(s)\right) v_{\varepsilon}(s), X_{\varepsilon,v_{\varepsilon}}^*(s)\right\rangle_H \mathrm{d} s \\
		&+2\sqrt{\varepsilon} \int_{-\infty}^{t}e^{2\eta s}\langle X_{\varepsilon,v_{\varepsilon}}^*(s),  B \left( X_{\varepsilon,v_{\varepsilon}}^* (s)\right) \mathrm{d} W(s)\rangle_H+\varepsilon \int_{-\infty}^{t}e^{2\eta s}\left| B \left( X_{\varepsilon,v_{\varepsilon}}^* (s)\right)\right|_{L_Q}^{2} \mathrm{~\mathrm{d}} s\\
		\leq&2\eta\int_{-\infty}^{t}e^{2\eta s} \|X_{\varepsilon,v_{\varepsilon}}^*\|_{H}^{2}\mathrm{d}s-2\lambda\int_{-\infty}^{t}e^{2\eta s} \|X_{\varepsilon,v_{\varepsilon}}^*\|_{V}^{2}\mathrm{d}s+\delta\int_{-\infty}^{t}e^{2\eta s} \|X_{\varepsilon,v_{\varepsilon}}^*\|_{H}^{2}\mathrm{d}s\\
		&+\frac{D^2}{\delta}\int_{-\infty}^{t}e^{2\eta s}\|v_{\varepsilon}(s)\|_{H_0}^2\mathrm{d}s+2\sqrt{\varepsilon} \int_{-\infty}^{t}e^{2\eta s}\left\langle X_{\varepsilon,v_{\varepsilon}}^*(s),  B \left( X_{\varepsilon,v_{\varepsilon}}^* (s)\right) \mathrm{d} W(s)\right\rangle_H+\varepsilon \int_{-\infty}^{t}e^{2\eta s}D^2 \mathrm{~\mathrm{d}} s.
	\end{align*}
	After choosing $\eta,\delta$ small enough, it follows from Remark \ref{condition remark} (i) that
	\begin{align*}
		e^{2\eta t} \|X_{\varepsilon,v_{\varepsilon}}^*\|_{H}^{2}\leq & \frac{D^2}{\delta}\int_{-\infty}^{t}e^{2\eta s}\|v_{\varepsilon}(s)\|_{H_0}^2\mathrm{d}s+\varepsilon \int_{-\infty}^{t}e^{2\eta s}D^2 \mathrm{~\mathrm{d}} s\\
		&+2\sqrt{\varepsilon} \int_{-\infty}^{t}e^{2\eta s}\left\langle X_{\varepsilon,v_{\varepsilon}}^*(s),  B \left( X_{\varepsilon,v_{\varepsilon}}^* (s)\right) \mathrm{d} W(s)\right\rangle_H,
	\end{align*}
	thus for any $k\in\mathbb{N}^+$, we can obtain that
	\begin{align*}
		\mathbb{E}\sup_{t\in(-\infty,k)}e^{2\eta t} \|X_{\varepsilon,v_{\varepsilon}}^*\|_{H}^{2}
\leq & \frac{D^2e^{2\eta k}}{\delta}\mathbb{E}\int_{-\infty}^{+\infty}\|v_{\varepsilon}(s)\|_{H_0}^2\mathrm{d}s+\varepsilon \int_{-\infty}^{k}e^{2\eta s}D^2  \mathrm{~\mathrm{d}} s \\&+2\sqrt{\varepsilon}\mathbb{E}\sup_{t\in(-\infty,k)} \int_{-\infty}^{t}e^{2\eta s}\left\langle X_{\varepsilon,v_{\varepsilon}}^*(s),  B \left( X_{\varepsilon,v_{\varepsilon}}^*(s)\right) \mathrm{d} W(s)\right\rangle_H.
	\end{align*}
	
	It follows from the Burkholder-Davis-Gundy inequality (cf. Proposition 3.26 in \cite{KS88}), Young inequality and Remark \ref{condition remark} (ii) that
	\begin{align*}
		&2\sqrt{\varepsilon}\mathbb{E}\left\{\sup_{t\in(-\infty,k]} \int_{-\infty}^{t}e^{2\eta s}\left\langle X_{\varepsilon,v_{\varepsilon}}^*(s),  B \left( X_{\varepsilon,v_{\varepsilon}}^* (s)\right) \mathrm{d} W(s)\right\rangle_H\right\}^{\frac{1}{2}}\\
		\leq&2\sqrt{\varepsilon}\mathbb{E}\left\{\int_{-\infty}^{k}e^{4\eta s}\|X_{\varepsilon,v_{\varepsilon}}^*\|_{H}^2| B \left( X_{\varepsilon,v_{\varepsilon}}^*\right)|_{L_Q}^2\mathrm{d}s\right\}^{\frac{1}{2}}\\
		\leq&2\sqrt{\varepsilon}\mathbb{E}\left\{\sup_{t\in(-\infty,k]}e^{\eta t}\|X_{\varepsilon,v_{\varepsilon}}^*\|_{H}\left\{\int_{-\infty}^{k}e^{2\eta s}| B \left( X_{\varepsilon,v_{\varepsilon}}^*\right)|_{L_Q}^2\mathrm{d}s\right\}^{\frac{1}{2}}\right\}\\
		\leq&\sqrt{\varepsilon}\mathbb{E}\left\{\sup_{t\in(-\infty,k]}
		e^{2\eta t}\|X_{\varepsilon,v_{\varepsilon}}^*\|_{H}^2+D^2\int_{-\infty}^{k}e^{2\eta s}\mathrm{d}s\right\},\\
	\end{align*}
	thus for every $\varepsilon<\frac{1}{4}$, we could have
	\begin{equation}\label{pq1}
		\begin{aligned}
			\mathbb{E}\sup_{t\in(-\infty,k)}e^{2\eta t} \|X_{\varepsilon,v_{\varepsilon}}^*\|_{H}^{2}\leq & \frac{2D^2e^{2\eta k}}{\delta}\mathbb{E}\int_{-\infty}^{+\infty}\|v_{\varepsilon}(s)\|_{H_0}^2\mathrm{d}s+2\varepsilon \int_{-\infty}^{k}e^{2\eta s}D^2  \mathrm{~\mathrm{d}} s \\&+2\sqrt{\varepsilon}D^2\int_{-\infty}^{k}e^{2\eta s}\mathrm{d}s.
		\end{aligned}
	\end{equation}
	For every $\varepsilon>0$, it follows from (\ref{pq1}) and $v_{\varepsilon}\in \mathcal{A}_M$, there exists $\varepsilon_0\in(0,\frac{1}{4})$  such that
	\begin{equation*}
		\mathbb{E}\sup_{t\in(-\infty,k)}e^{2\eta t} \|X_{\varepsilon,v_{\varepsilon}}^*\|_{H}^{2}\leq C(M,\varepsilon_0,D,k),\quad\forall \varepsilon\in(0,\varepsilon_0).
	\end{equation*}
	
\end{proof}
The following lemmas
will verify the stationary solution $X^{*}_{\varepsilon}$ of Eq. \eqref{Msode} satisfies  Condition \ref{weakcondition} in space $C(\mathbb{R};H)$, which is a sufficient condition to $X^{*}_{\varepsilon}$ satisfies the LDP in space $C(\mathbb{R};H)$.

\begin{lemma}\label{verify weak condition 1}
	Let $\{v^{\varepsilon}:\varepsilon>0\}\subset\mathcal{A}_M$, for some $M<\infty$. If $v^{\varepsilon} $ converges to $v$ as $S_M$-valued random element in distribution, then
	$
	X_{\varepsilon,v_{\varepsilon}}^* \rightarrow X_v^{*}
	$
	in distribution as $\varepsilon\rightarrow 0$.
\end{lemma}
\begin{proof}
	Since $X_{\varepsilon,v_{\varepsilon}}^*$ converges to $X_v^*$ in probability  could deduce  $X_{\varepsilon,v_{\varepsilon}}^*$ converges to $X^*_v$ in distribution, it is sufficient to prove $X_{\varepsilon,v_{\varepsilon}}^*$ converges to $X^*_v$ in probability.
	For convenience, let $w_{\varepsilon}:=X_{\varepsilon,v_{\varepsilon}}^*-X_v^*$. Since
	for any $\delta>0$, there exists $N>0$ such that  $\sum_{k=N+1}^{\infty}\frac{1}{2^k}<\frac{\delta}{2}$, then it follows from the definition of the norm (\ref{Cnorm}) in space $C(\mathbb{R};H)$,
	\begin{align*}
		\mathbb{P}(\|w_{\varepsilon}\|_{\mathcal{C}}>\delta)\leq& \mathbb{P}\left(\sum_{k=1}^{N}\frac{\sup_{t\in[-k,k]}\| w_{\varepsilon}(t)\|_{H}\wedge1}{2^k}>\frac{\delta}{2}\right)\\
		\leq&\sum_{k=1}^{N}\mathbb{P}\left(\frac{\sup_{t\in[-k,k]}\| w_{\varepsilon}(t)\|_{H}\wedge1}{2^k}>\frac{\delta}{2N}\right).
	\end{align*}
	Thus it is sufficient to prove that $\sup_{t\in[k,k]}\|w_{\varepsilon}(t)\|_{H}\rightarrow 0$ as $\varepsilon\rightarrow 0$ in probability for every $k\in\mathbb{N}^+$.
	
	By using the It\^{o} formula, Hypothesis \ref{Msodehy} (i) and Remark \ref{condition remark} (ii), we can get that
	\begin{align*}
		\frac{1}{2}e^{2\eta t} \|w_{\varepsilon}\|_{H}^{2}
		=&\eta\int_{-\infty}^{t}e^{2\eta s} \|w_{\varepsilon}\|_{H}^{2}\mathrm{d}s+\int_{-\infty}^{t} e^{2\eta s}   {_{V^*}}\left\langle Aw_{\varepsilon}+ F(X_{\varepsilon,v_{\varepsilon}}^*)-F(X_v^*), w_{\varepsilon}(s)\right\rangle \mathrm{d} s\\
		&+\int_{-\infty}^{t}e^{2\eta s}\left\langle B \left( X_{\varepsilon,v_{\varepsilon}}^*(s)\right) v_{\varepsilon}(s)- B \left( X_v^*(s)\right)v(s), w_{\varepsilon}(s)\right\rangle_H \mathrm{d} s \\
		&+\sqrt{\varepsilon} \int_{-\infty}^{t}e^{2\eta s}\left\langle w_{\varepsilon}(s),  B \left( X_{\varepsilon,v_{\varepsilon}}^* (s)\right) \mathrm{d} W(s)\right\rangle_H+\frac{\varepsilon}{2} \int_{-\infty}^{t}e^{2\eta s}\left| B \left( X_{\varepsilon,v_{\varepsilon}}^* (s)\right) Q^{1 / 2}\right|_{L_2}^{2} \mathrm{~\mathrm{d}} s \\
		\leq & \eta\int_{-\infty}^{t}e^{2\eta s} \|w_{\varepsilon}\|_{H}^{2}\mathrm{d}s-\lambda \int_{-\infty}^{t}e^{2\eta s}\left\|w_{\varepsilon}(s)\right\|_{V}^2
		+C_0\int_{-\infty}^{t}e^{2\eta s}
		\left\|w_{\varepsilon}(s)\right\|_{H}^2\left\|X_v^*(s)\right\|_{V}^2 \mathrm{d} s \\
		&+\int_{-\infty}^{t}e^{2\eta s}\left| B \left( X_{\varepsilon,v_{\varepsilon}}^* (s)\right)- B \left( X_v^*(s)\right) Q^{1 / 2}\right|_{L_2}\left\|v_{\varepsilon}(s)\right\|_{H_0}\left\|w_{\varepsilon}(s)\right\|_H \mathrm{d} s \\
		&+\int_{-\infty}^{t}e^{2\eta s}\left\| B  (X_v^*(s))(v_{\varepsilon}(s)-v(s))\right\|_{H} \left\| w_{\varepsilon}(s)\right\|_{H} \mathrm{d} s \\
		&+\sqrt{\varepsilon} \int_{-\infty}^{t}e^{2\eta s}\left\langle w_{\varepsilon}(s),  B \left( X_{\varepsilon,v_{\varepsilon}}^* (s)\right) \mathrm{d} W(s)\right\rangle_H+\frac{\varepsilon}{2} \int_{-\infty}^{t}e^{2\eta s}D^2 \mathrm{~\mathrm{d}} s .\\
	\end{align*}
	It follows from the Young inequality, Remark \ref{condition remark} (i) and (ii), after choosing $\eta$ small enough, we can obtain that
	\begin{align*}
		&\frac{1}{2}e^{2\eta t} \|w_{\varepsilon}\|_{H}^{2}+\frac{3\lambda}{4} \int_{-\infty}^{t}e^{2\eta s}\left\|w_{\varepsilon}(s)\right\|_{V}^{2} \mathrm{~\mathrm{d}} s \\
		\leq &  C_0\int_{-\infty}^{t}e^{2\eta s}\left\|w_{\varepsilon}(s)\right\|_{H}^2\left\|X_v^*(s)\right\|_{V}^2 \mathrm{~\mathrm{d}} s+\beta \int_{-\infty}^{t}e^{2\eta s}\left\|w_{\varepsilon}(s)\right\|_{H}\left\|v_{\varepsilon}(s)\right\|_{H_0}\left\|w_{\varepsilon}(s)\right\|_{H} \mathrm{d} s \\
		&+\int_{-\infty}^{t}e^{2\eta s}\left\| B  (X_v^*(s))(v_{\varepsilon}(s)-v(s))\right\|_{H} \left\| w_{\varepsilon}(s)\right\|_{H} \mathrm{d} s \\
		&+\sqrt{\varepsilon} \int_{-\infty}^{t}e^{2\eta s}\left\langle w_{\varepsilon}(s),  B \left( X_{\varepsilon,v_{\varepsilon}}^* (s)\right) \mathrm{d} W(s)\right\rangle_H+\frac{\varepsilon}{2} \int_{-\infty}^{t}e^{2\eta s}D^2 \mathrm{~\mathrm{d}} s \\
		\leq&C_0
		\int_{-\infty}^{t}e^{2\eta s}\left\|w_{\varepsilon}(s)\right\|_{H}^2\left\|X_v^*(s)\right\|_{V}^2 \mathrm{~\mathrm{d}} s+\frac{\lambda}{4} \int_{-\infty}^{t}e^{2\eta s}\left\|w_{\varepsilon}(s)\right\|_{V}^2 \mathrm{d} s+\frac{2\beta^2}{\lambda C_1^2} \int_{-\infty}^{t}e^{2\eta s}\left|v_{\varepsilon}(s)\right|_{H_0}^2\|w_{\varepsilon}\|_{H}^2 \mathrm{d} s  \\
		&+\frac{\lambda}{4} \int_{-\infty}^{t}e^{2\eta s}\left\|w_{\varepsilon}(s)\right\|_{V}^2 \mathrm{d} s +\frac{2}{\lambda C_1^2}
		\int_{-\infty}^{t}e^{2\eta s}\left\| B  (X_v^*(s))(v_{\varepsilon}(s)-v(s))\right\|_{H}^2 \mathrm{d} s \\
		&+\sqrt{\varepsilon} \int_{-\infty}^{t}e^{2\eta s}\left\langle w_{\varepsilon}(s),  B \left( X_{\varepsilon,v_{\varepsilon}}^* (s)\right) \mathrm{d} W(s)\right\rangle_H+\frac{\varepsilon}{2} \int_{-\infty}^{t}e^{2\eta s}D^2 \mathrm{~\mathrm{d}} s .\\
	\end{align*}
	
	Define
	\begin{align*}
		\tau_{N,\varepsilon}^k:=k\wedge\inf\Big\{&t:\int_{-\infty}^{t}\left\|X_v^*(s)\right\|_{V}^2\mathrm{d}s>N
		\quad\text{\rm or}\quad\sup_{s\in(-\infty,t]}\|X_v^*(s)\|_H^2>N\\
		&\text{\rm or}\quad\sup_{s\in(-\infty,t]}e^{2\eta s}\|X_{\varepsilon,v_{\varepsilon}}^*(s)\|_H^2>N
		\Big\},
	\end{align*} $\tau_{N,\varepsilon}^k$ is a continuous random process for $\|X_v^*\|_{H}$ and $\|X_{\varepsilon,v_{\varepsilon}}^*\|_{H}$ are continuous (cf. Theorem \ref{Mstationary} or Theorem 2.2 in \cite{LZ09}).
	
	Then for any fixed $k\in\mathbb{N}^+$, $T\in[-k,k]$, there exists a constant $C$, which only depend on $\lambda, C_0, \beta, D, C_1$ such that
	\begin{align*}
		&\sup_{t\in(-\infty,T\wedge\tau_{N,\varepsilon}^k]}e^{2\eta t} \|w_{\varepsilon}(t)\|_{H}^{2}+ \int_{-\infty}^{T\wedge\tau_{N,\varepsilon}^k}e^{2\eta s}\left\|w_{\varepsilon}(s)\right\|_{V}^{2} \mathrm{~\mathrm{d}} s \\
		\leq&C\left\{\int_{-\infty}^{T\wedge\tau_{N,\varepsilon}^k}e^{2\eta s}\left\|w_{\varepsilon}(s)\right\|_{H}^2
		\left\|X_v^*(s)\right\|_{V}^2 \mathrm{~\mathrm{d}} s+
		\int_{-\infty}^{T\wedge\tau_{N,\varepsilon}^k}e^{2\eta s}\|v_{\varepsilon}\|_{H_0}^2\left\| w_{\varepsilon}(s)\right\|_{H}^2 \mathrm{d} s\right.\\
		&+ \int_{-\infty}^{T\wedge\tau_{N,\varepsilon}^k}e^{2\eta s}\left\| B  (X_v^*(s))(v_{\varepsilon}(s)-v(s))\right\|_{H}^2 \mathrm{d} s+\varepsilon \int_{-\infty}^{T\wedge\tau_{N,\varepsilon}^k}e^{2\eta s}D^2 \mathrm{~\mathrm{d}} s
		\\
		&\left.+\sqrt{\varepsilon}\sup_{t\in(-\infty,T\wedge\tau_{N,\varepsilon}^k]} \int_{-\infty}^{t}e^{2\eta s}\left\langle w_{\varepsilon}(s), B \left( X_{\varepsilon,v_{\varepsilon}}^* (s)\right) \mathrm{d} W(s)\right\rangle_H\right\}.\\
	\end{align*}
	
	Let $M_t:=\sup_{s\in(-\infty,t]}e^{2\eta s} \|w_{\varepsilon}(s)\|_{H}^{2}$, then we have
	\begin{align*}
		M_{T\wedge\tau_{N,\varepsilon}^k}\leq &C\left\{\int_{-\infty}^{T\wedge\tau_{N,\varepsilon}^k}M_s
		\left\|X_v^*(s)\right\|_{V}^2 \mathrm{~\mathrm{d}} s+ \int_{-\infty}^{T\wedge\tau_{N,\varepsilon}^k}M_s\left|v_{\varepsilon}(s)\right|_{H_0}^2 \mathrm{d} s\right.\\
		& +
		\int_{-\infty}^{T\wedge\tau_{N,\varepsilon}^k}\left\| B\left(X_v^*(s)\right)(v_{\varepsilon}(s)-v(s))\right\|_{H}^2 \mathrm{d} s +\varepsilon \int_{-\infty}^{T\wedge\tau_{N,\varepsilon}^k}e^{2\eta s}D^2 \mathrm{~\mathrm{d}} s\\
		&\left.+\sqrt{\varepsilon}\sup_{t\in(-\infty,T\wedge\tau_{N,\varepsilon}^k]} \int_{-\infty}^{t}e^{2\eta s}\left\langle w_{\varepsilon}(s), B \left( X_{\varepsilon,v_{\varepsilon}}^* (s)\right) \mathrm{d} W(s)\right\rangle_H \right\}.\\
	\end{align*}
	By using the Burkholder-Davis-Gundy inequality, Young inequality and Remark \ref{condition remark} (ii), there exists a constant $C$ such that
	\begin{align}\label{w1m}
		&\mathbb{E}\left\{\sup_{t\in(-\infty,T\wedge\tau_{N,\varepsilon}^k]} \int_{-\infty}^{t}e^{2\eta s}\left\langle w_{\varepsilon}(s), B \left( X_{\varepsilon,v_{\varepsilon}}^* (s)\right) \mathrm{d} W(s)\right\rangle_H\right\}^{\frac{1}{2}}\nonumber\\
		\leq&C\mathbb{E}\left\{\int_{-\infty}^{T\wedge\tau_{N,\varepsilon}^k}e^{4\eta s}\|w_{\varepsilon}\|_{H}^2| B \left( X_{\varepsilon,v_{\varepsilon}}^*\right)|_{L_Q}^2\mathrm{d}s\right\}^{\frac{1}{2}}\nonumber\\
		\leq&C\mathbb{E}\left\{\sup_{t\in(-\infty,T\wedge\tau_{N,\varepsilon}^k]}e^{\eta t}\|w_{\varepsilon}\|_{H}\left\{\int_{-\infty}^{T\wedge\tau_{N,\varepsilon}^k}e^{2\eta s}| B \left( X_{\varepsilon,v_{\varepsilon}}^*\right)|_{L_Q}^2\mathrm{d}s\right\}
		^{\frac{1}{2}}\right\}\nonumber\\
		\leq&\frac{C}{2}\mathbb{E}\left\{\sup_{t\in(-\infty,T\wedge\tau_{N,\varepsilon}^k]}
		e^{2\eta t}\|w_{\varepsilon}\|_{H}^2+D^2\int_{-\infty}^{T\wedge\tau_{N,\varepsilon}^k}e^{2\eta s}\mathrm{d}s\right\}<\infty.\nonumber\\
	\end{align}
	Then using the Gronwall inequality, we could obtain that
	\begin{equation}\label{w1g}
		\begin{aligned}
			M_{T\wedge\tau_{N,\varepsilon}^k}&\leq C e^{C(N+M)}\left\{\sqrt{\varepsilon}\sup_{t\in(-\infty,T\wedge\tau_{N,\varepsilon}^k]} \int_{-\infty}^{t}e^{2\eta s}\left\langle w_{\varepsilon}(s), B \left( X_{\varepsilon,v_{\varepsilon}}^* (s)\right) \mathrm{d} W(s)\right\rangle_H\right.\\
			&\left.+\varepsilon \int_{-\infty}^{T\wedge\tau_{N,\varepsilon}^k}e^{2\eta s}D^2 \mathrm{~\mathrm{d}} s
			+\int_{-\infty}^{T\wedge\tau_{N,\varepsilon}^k}\left\| B  (X_v^*(s))(v_{\varepsilon}(s)-v(s))\right\|_{H}^2 \mathrm{d} s\right\}.
		\end{aligned}
	\end{equation}
	
	It follows from $ B  (\cdot)Q^{1/2}$
	is a Hilbert-Schmidt operator on $H$, Remark \ref{condition remark} (ii), and $v_{\varepsilon}$ converges to $v$ as $S_M$-value random variable in distribution,
	\begin{equation}\label{w1v}
		\lim_{\varepsilon\rightarrow 0}\mathbb{E}\int_{-\infty}^{+\infty}\left\| B  (X_v^*(s))(v_{\varepsilon}(s)-v(s))\right\|_{H}^2 \mathrm{d} s\leq \lim_{\varepsilon\rightarrow 0}D^2\mathbb{E}\int_{-\infty}^{+\infty}\left\|(v_{\varepsilon}(s)-v(s))\right\|_{H_0}^2 \mathrm{d} s =0.
	\end{equation}
	It follows from  Lemma \ref{Piror of skeleton equation}, Lemma \ref{prior Xev}, and the Chebyshev inequality, for any fixed $k$, it is easy to obtain that there exists $\varepsilon_0$, which is defined by Lemma \ref{prior Xev}, and a constant $C$ such that
	$$\mathbb{P}\{\tau_{N,\varepsilon}^k= k\}\geq1-\frac{C}{N},\quad\forall \varepsilon\in(0,\varepsilon_0),$$
	it implies that $\lim\limits_{N\rightarrow \infty}\tau_{N,\varepsilon}^k=k,~ a.e.~ \forall \varepsilon\in(0,\varepsilon_0)$. Thus combining inequalitys (\ref{w1m}), (\ref{w1g})  and (\ref{w1v}), let $N\rightarrow \infty$ and $\varepsilon\rightarrow 0$, we get that
	\begin{equation}\label{pp1}
		\sup_{t\in(-\infty,k]}e^{2\eta t}\|w_{\varepsilon}(t)\|_{H}^2\rightarrow 0
	\end{equation}
	in probability. Moreover, for any fixed $k$, it follows from
	\begin{equation*}
		e^{-\eta k}\sup_{t\in[-k,k]}\|w_{\varepsilon}(t)\|_{H}^2\leq \sup_{t\in[-k,k]}e^{2\eta t}\|w_{\varepsilon}(t)\|_{H}^2\leq \sup_{t\in(-\infty,k]}e^{2\eta t}\|w_{\varepsilon}(t)\|_{H}^2,
	\end{equation*}
	and (\ref{pp1}) to obtain that $\sup_{t\in[-k,k]}\|w_{\varepsilon}(t)\|_{H}^2\rightarrow 0$ in probability as $\varepsilon\rightarrow 0$.
\end{proof}
\begin{lemma}\label{verify weak condition 2}
	For  any fixed finite positive number $M$, the set $K_M=\Big\{\mathcal{G}^{0}(\int_{-\infty}^{\cdot}v(s)\mathrm{d}s);v\in S_M\Big\}$ is a compact subset of $C((-\infty,0];H)$.
\end{lemma}
\begin{proof}
	Let $\big\{\mathcal{G}^{0}(\int_{-\infty}^{\cdot}v_n
	(s)\mathrm{d}s)\big\}_{n=1}^{\infty}$ be a sequence in $K_M$, for convenience let $X^n:=\mathcal{G}^{0}\big(\int_{-\infty}^{\cdot}v_n(s)\mathrm{d}s\big)$ where $X^n$ corresponds to the solution of Eq. \eqref{skeleton equation} with $v_n\in S_M$ in place of $v$. By weak compactness of $S_M$, there exists a subsequence of $\{v_n\}$ which converges to a limit $v$ weakly in $L^2(\mathbb{R};H_0)$. The subsequence is indexed by $n$ for ease of notation. Let  $X:=\mathcal{G}^{0}\big(\int_{-\infty}^{\cdot}v(s)\mathrm{d}s\big)$ be the solution of Eq. \eqref{skeleton equation} responding to $v$ and $w^n=X^n-X$. Then we obtain
	\begin{equation*}
		\frac{\mathrm{d}w^n}{\mathrm{d}t}=A w_n+F(X^n)-F(X)+ B (X^n)v_n- B (X)v,
	\end{equation*}
	it follows from Hypothesis \ref{Msodehy} (i) and Remark \ref{condition remark} (i), (ii) that
	$$
	\begin{aligned}
		\frac{1}{2}\left\|w^{n}(t)\right\|_{H}^{2}
		=& \int_{-\infty}^{t}{_{V^*}}\langle A w_n+F(X^n)-F(X),w^n\rangle_V\\
		&+ \int_{-\infty}^{t}\Big\{\langle\left( B \left( X^n(s)\right)- B ( X(s)\right) v_{n}(s), w^n(s)\rangle_H\\
		&+\langle B ( X(s))\left(v_{n}(s)-v(s)\right), w^n(s)\rangle_H\Big\} \mathrm{d} s\\
		\leq&-\lambda\int_{-\infty}^{t}\left\|w^{n}(s)\right\|_{V}^{2}\mathrm{d} s+C_0\int_{-\infty}^{t}\left\|w^n(s)\right\|_{H}^2\|X(s)\|_{V}^2 \mathrm{d} s \\
		&+\frac{\beta}{C_1}\int_{-\infty}^{t} \left\|w^n(s)\right\|_{V}\left\|w^n(s)\right\|_{H}\left\|v_{n}(s)\right\|_{H_0} \mathrm{~\mathrm{d}} s\\
		&+\int_{-\infty}^{t}\left\| B (X(s))(v_{n}(s)-v(s))\right\|_{H}\left\|w^n(s)\right\|_{H} \mathrm{d} s.
	\end{aligned}
	$$
	For any $k\in\mathbb{N}^+$ and $T\in[-k,k]$,
	by using the Young inequality and Remark \ref{condition remark} (i), we have
	$$
	\begin{aligned}
		\frac{1}{2}\sup_{t\in(-\infty,T]}\left\|w^{n}(t)\right\|_{H}^{2} + \frac{\lambda}{2} \int_{-\infty}^{T}\left\|w^n(s)\right\|_{V}^{2} \mathrm{~\mathrm{d}} s \leq& C \int_{-\infty}^{T}\left\|w^n(s)\right\|_{H}^{2}\left(\|X(s)\|_{V}^2+\left\|v_{n}(s)\right\|_{H_0}^2\right) \mathrm{~\mathrm{d}} s \\
		&+\int_{-\infty}^{+\infty}\left\| B (X(s))(v_{n}(s)-v(s))\right\|_{H}^2 \mathrm{d} s,
	\end{aligned}
	$$
	where $C$ is a constant only depend on $\lambda,C_0,C_1,\beta$.
	
	It follows from the Gronwall inequality that
	\begin{equation}\label{4}
		\begin{aligned}
			&\frac{1}{2}\sup_{t\in(-\infty,k]}\left\|w^{n}(t)\right\|_{H}^{2}\\
			\leq & C\int_{-\infty}^{+\infty}\left\| B (X(s))(v_{n}(s)-v(s))\right\|_{H}^2 \mathrm{d} s\cdot
			\exp\left\{\int_{-\infty}^{k}\left(\|X(s)\|_{V}^2
			+\left\|v_{n}(s)\right\|_{H_0}^2\right) \mathrm{~\mathrm{d}} s \right\}.
		\end{aligned}
	\end{equation}
	It follows from $v_n\in S_M$, and Lemma \ref{Piror of skeleton equation} that
	\begin{equation}\label{5}
		\exp\left\{\int_{-\infty}^{k}\left(\|X(s)\|_{V}^2
		+\left\|v_{n}(s)\right\|_{H_0}^2\right) \mathrm{~\mathrm{d}} s \right\}<\infty.
	\end{equation}
	Since $v_n,v\in S_M$, there exits simple function sequences $\tilde{v}_{n_k}$ and $\tilde{v}_k$ strong convergence to $v_n$ and $v$ respectively, we can choose a subsequence still record it as $n_k$, such that
	\begin{equation}\label{1}
		\lim_{n\rightarrow\infty}\lim_{k\rightarrow\infty}\int_{-\infty}^{+\infty}\|\tilde{v}_{n_k}(s)-v_n(s)\|_{H_0}^2\mathrm{d}s=0,
	\end{equation}
	since $v_n-v$ weak converges to $0$ in $S_M$,  $\tilde{v}_{n_k}-\tilde{v}_k$ also weak converges to $0$. By using Remark \ref{condition remark} (ii) and Hypothesis \ref{Msodehy} (iii), we can obtain that
	\begin{align}\label{2}
		&\int_{-\infty}^{+\infty}\left\| B (X(s))(v_{n}(s)-v(s))\right\|_{H}^2 \mathrm{d}s\nonumber\\
		\leq&\int_{-\infty}^{+\infty}\left\| B (X(s))(v_{n}(s)-\tilde{v}_{n_k}(s))\right\|_{H}^2 \mathrm{d} s +\int_{-\infty}^{+\infty}\left\| B (X(s))(\tilde{v}_{n_k}(s)-\tilde{v}_{k}(s))\right\|_{H}^2 \mathrm{d} s\\ &+\int_{-\infty}^{+\infty}\left\| B (X(s))(\tilde{v}_{k}(s)-v(s))\right\|_{H}^2 \mathrm{d} s\nonumber\\
		\leq& D^2\int_{-\infty}^{+\infty}\left\|v_{n}(s)-\tilde{v}_{n_k}(s)\right\|_{H_0}^2
		\mathrm{d} s
		+D^2\int_{-\infty}^{+\infty}\left\|\tilde{v}_{k}(s)-v(s)\right\|_{H_0}^2 \mathrm{d} s\nonumber\\
		+&D_0^2\int_{-\infty}^{+\infty}\left\|Q^{1/2}Q^{-1/2}(\tilde{v}_{n_k}(s)-\tilde{v}_{k}(s))\right\|_{H}^2 \mathrm{d} s\nonumber.
	\end{align}
	It follows from  $Q^{1/2}$ is
	a Hilbert-Schmidt operator on $H$ and $\tilde{v}_{n_k}-\tilde{v}_k$ weak converges to $0$ in $S_M$
	that there exists a subsequence, still set $n_k$, such that
	\begin{equation}\label{3}
		\lim_{k\rightarrow\infty}\int_{-\infty}^{+\infty}\left\|Q^{1/2}Q^{-1/2}(\tilde{v}_{n_k}(s)-\tilde{v}_{k}(s))\right\|_{H}^2 \mathrm{d} s=0.
	\end{equation}
	
	It follows from (\ref{1}), (\ref{2}), (\ref{3}) and $\tilde{v}_k$ strong converges to  $v$ in $S_M$ that
	\begin{equation}\label{6}
		\lim_{n\rightarrow\infty}\int_{-\infty}^{+\infty}\left\| B (X(s))(v_{n}(s)-v(s))\right\|_{H}^2 \mathrm{d}s=0.
	\end{equation}
	Combining (\ref{4}), (\ref{5}) and (\ref{6}), we could get that for any $k\in\mathbb{N}^+$,
	\begin{equation*}
		\lim_{n\rightarrow \infty}\sup_{t\in(-\infty,k]}\left\|w^{n}(t)\right\|_{H}^{2}=0.
	\end{equation*}
	By using the dominated convergence theorem, we obtain that
	\begin{align*}
		\lim_{n\rightarrow \infty}\left\|X^n-X\right\|_{\mathcal{C}}
		&=\lim_{n\rightarrow \infty}\sum_{k=1}^{\infty}\frac{\sup_{t\in[-k,k]}
			\|X^n-X\|_{H}\wedge 1}{2^k}\leq \lim_{n\rightarrow \infty}\sum_{k=1}^{\infty}
		\frac{\sup_{t\in(-\infty,k]}\|X^n-X\|_{H}\wedge 1}{2^k}\\
		&=\sum_{k=1}^{\infty}\frac{\lim_{n\rightarrow \infty}\sup_{t\in(-\infty,k]}\|X^n-X\|_{H}\wedge 1}{2^k}=0,
	\end{align*}
	it implies that $K_M$ is compact.
\end{proof}
Lemma \ref{verify weak condition 1} and \ref{verify weak condition 2} have proved that the stationary solution $X^{*}_\varepsilon$ of Eq. \eqref{Msode} satisfies Condition \ref{weakcondition} in Polish Space $C(\mathbb{R};H)$. By using Theorem \ref{LP} and the equivalence of LDP and LP in Polish space, the following theorem holds.
\begin{theorem}\label{stationary LDP}
	Under the Hypothesis \ref{Msodehy},
	the family $\{X^{*}_\varepsilon:\varepsilon>0\}$ satisfies the LDP in $C(\mathbb{R};H)$ with rate function
	\begin{equation}\label{ratefunction}
		I(f)=\inf_{\left\{v\in L^2(\mathbb{R};H_0):f
			=\mathcal{G}^0(\int_{-\infty}^{\cdot}v(s)\mathrm{d}s)\right\}}
		\left\{\frac{1}{2}\int_{-\infty}^{+\infty}\|v(s)\|_{H_0}^2\mathrm{d}s\right\},
	\end{equation}
	where the infimum over an empty set is taken as $+\infty$.
\end{theorem}
\subsection{Some Examples}
Next, we will give an example that satisfies the Hypothesis \ref{Msodehy}.
\begin{example}\label{burgers equation}
	We consider the stochastic Burgers equation with multiplicative noise as follows,
	\begin{equation}\label{MSBE1}
		\left\{
		\begin{array}{l}
			\mathrm{d}\textbf{u}(t,x) =  \Delta \textbf{u}(t,x)\mathrm{d}t + \frac{1}{2}\partial _{x}[\textbf{u}(t,x)]^2\mathrm{d}t + \sqrt{\varepsilon} B  (\textbf{u})\mathrm{d}W(t),t\geqslant s, x\in(0,1),\\
			\textbf{u}(s,x) = \textbf{u}_{s}(x),\\
			\textbf{u}(t,1)=\textbf{u}(t,0)=0.
		\end{array}
		\right.
	\end{equation}
	To correspond to the notation of Hypothesis \ref{Msodehy}, let $A\textbf{u}+F(\textbf{u}):=\Delta \textbf{u}(t,x) + \frac{1}{2}\partial _{x}[\textbf{u}(t,x)]^2$.\\
	Let
	\begin{align*}
		&H:=\left\{f:[0,1]\rightarrow \mathbb{R}^1:f(0)=f(1)=0,\quad\int_{0}^{1}f^2(x)\mathrm{d}x<\infty\right\},\\
		&V:=\left\{f\in H: \int_{0}^{1}(\partial_xf(x))^2\mathrm{d}x<\infty\right\}.
	\end{align*}
	
	We will verify $A\textbf{u}$ and $F(\textbf{u})$ satisfies Hypothesis \ref{Msodehy} (i) and (ii) below.\\
	By using integration by parts, we can obtain that
	\begin{align*}
		\int_{0}^{1}\textbf{u}^2(t,x)\partial _{x}(\textbf{u}(t,x))\mathrm{d}x &= 0, \\	\int_{0}^{1}\textbf{u}(t,x)\textbf{w}(t,x)\partial _{x}(\textbf{w}(t,x))\mathrm{d}x &= -\frac{1}{2}\int_{0}^{1}\textbf{w}^2(t,x)\partial _{x}\textbf{u}(t,x)\mathrm{d}x,
	\end{align*}
	for any $\textbf{u},\textbf{w}\in V$, then
	\begin{equation}\label{burgers1}
		\begin{aligned}
			_{V^*}\langle F(\textbf{u}(t,x))- F(\textbf{w}(t,x)),\textbf{u}-\textbf{w}\rangle_V
			=&\langle \textbf{u}\partial_x\textbf{u} - \textbf{w}\partial_x\textbf{w},\textbf{u}-\textbf{w}\rangle_{H}\\
			=& \langle \textbf{u}\partial_x(\textbf{u}-\textbf{w})+(\textbf{u}-\textbf{w})\partial_x\textbf{u}-(\textbf{u}-\textbf{w})\partial_x(\textbf{u}-\textbf{w}),\textbf{u}-\textbf{w}\rangle_{H}\\
			=&\frac{1}{2}\int_{0}^{1}\textbf{w}^2(t,x)\partial _{x}\textbf{u}(t,x)\mathrm{d}x.
		\end{aligned}
	\end{equation}
	
	According to Lemma A.1 in \cite{LZ09}, we could have the following property.
	\begin{equation}\label{gamma}
		\left|\int_{0}^{1} \textbf{v}(x) \partial_x\textbf{w}(x) \textbf{u}(x) \mathrm{d} x\right|^{2} \leq \gamma^{2}\|\textbf{v}\|_{H}\|\textbf{v}\|_{V}
		\|\textbf{w}\|_{V}^{2}\|\textbf{u}\|_{H}\|\textbf{u}\|_{V},
		\quad\forall \textbf{u}, \textbf{v}, \textbf{w} \in V,
	\end{equation}
	where $\gamma$ is the minimal constant such that Sobolev's inequality,
	$$
	\max _{x \in[0,1]}|\textbf{u}(x)| \leq \gamma\|\textbf{u}\|_{V}, \quad \forall u \in V
	$$
	holds.
	
	For linear operator $A$, we have known that
	\begin{equation}\label{buregrsA}
		_{V^*}\langle A\textbf{u},\textbf{u}\rangle_{V}=_{V^*}\langle \Delta \textbf{u},\textbf{u}\rangle_{V}\leq -\|\textbf{u}\|_{V}^2,\quad\forall \textbf{u}\in V.
	\end{equation}
	It follows from \eqref{burgers1}, \eqref{gamma}, \eqref{buregrsA} and the Young inequality that
	\begin{align*}
		{_{V^*}}\langle A\textbf{u}-A\textbf{w}+F(\textbf{v})-F(\textbf{w}),\textbf{v}-\textbf{w}\rangle_V
		=&{_{V^*}}\langle\Delta (\textbf{u}-\textbf{w}),\textbf{u}-\textbf{w}\rangle_V+\frac{1}{2}\int_{0}^{1}\textbf{w}^2(t,x)\partial _{x}\textbf{u}(t,x)\mathrm{d}x\\
		\leq&-\|\textbf{u}-\textbf{w}\|_{V}^2+\frac{\gamma}{2}\|\textbf{u}-\textbf{w}\|_{H}\|\textbf{u}-\textbf{w}\|_{V}\|\textbf{u}\|_{V}\\
		\leq&-\frac{1}{2}\|\textbf{u}-\textbf{w}\|_{V}^2+\frac{\gamma^2}{8}\|\textbf{u}-\textbf{w}\|_{H}^2\|\textbf{u}\|_{V}^2,
	\end{align*}
	it implies that Eq. \eqref{MSBE1} satisfies Hypothesis \ref{Msodehy} (i).
	
	Let $S_{A}$ be the heat semigroup, through the  similar calculation  to Lemma 3.3 in \cite{LZ09}, for any $N\in\mathbb{N}^+$, $t_0<t\in[-N,N]$, we can get
	\begin{equation*}
		\left\|\int_{t_0}^{t}S_{A}(s)(F(\textbf{v})-F(\textbf w))\mathrm{d}s\right\|_{H}^2\leq C (N) \int_{t_0}^{t} s^{-\frac{3}{4}}\|\textbf{v}-\textbf{w}\|_{H}^2\mathrm{d}s,
	\end{equation*}
	it deduces that Eq. \eqref{MSBE1} satisfies Hypothesis \ref{Msodehy} (ii).
	
	Therefore, assume that the function $ B $ satisfies Hypothesis \ref{Msodehy} (iii). It follows from Lemma \ref{Mstationary solution} that there exists $\varepsilon_0>0$ such that for any $\varepsilon\in(0,\varepsilon_0)$,
	the stationary solution $\textbf{u}^*_{\varepsilon}$ of Burgers equation exists and is unique, which satisfies the following equation in $H$ for any $t \in \mathbb{R}$,
	\begin{equation*}
		\textbf{u}^*_{\varepsilon}(t,\omega) = \frac{1}{2}\int_{-\infty}^{t}S_{A}(t-r)\partial _{x}[\textbf{u}^*_{\varepsilon}(r,\omega)]^2\mathrm{d}r +\sqrt{\varepsilon} \int_{-\infty}^{t}S_{A}(t-r) B (\textbf{u}^*_{\varepsilon}(r,\omega))\mathrm{d}W(r),
	\end{equation*}
	and
	\begin{equation*}
		\sup_{t\in \mathbb{R}}\mathbb{E}\left\|\textbf{u}^*_{\varepsilon}(t)\right\|_{H}^{2} < \infty.
	\end{equation*}
	
	The skeleton equation of Burgers equation is
	\begin{equation}\label{SKMSBE1}
		\left\{
		\begin{array}{l}
			\frac{\mathrm{d}\textbf{u}(t,x)}{\mathrm{d}t }=  \Delta \textbf{u}(t,x) + \frac{1}{2}\partial _{x}[\textbf{u}(t,x)]^2 + B  (\textbf{u})v(t),\quad t\geqslant  s, x\in(0,1),\\
			\textbf{u}(s,x) = \textbf{u}_{s}(x),\\
			\textbf{u}(t,1)=\textbf{u}(t,0)=0.
		\end{array}
		\right.
	\end{equation}
	It follows from Lemma \ref{SK-existence} and Lemma \ref{SK-uniqueness}, for any $v\in L^2(\mathbb{R};H_0)$,  there exits a unique solution $\textbf{u}^{*}_{v}$ of Eq. \eqref{SKMSBE1}, which satisfies the following equation in $H$ for any $t \in \mathbb{R}$,
	\begin{equation*}
		\textbf{u}^{*}_{v}(t,\omega) = \frac{1}{2} \int_{-\infty}^{t}S_{A}(t-r)\partial_x[\textbf{u}^{*}_{v}(r,\omega)]^2\mathrm{d}r + \int_{-\infty}^{t}S_{A}(t-r) B (\textbf{u}^{*}_{v}(r,\omega))v(r)\mathrm{d}r.
	\end{equation*}
	We define $\mathcal{G}^0:C(\mathbb{R};H_0)\longrightarrow C(\mathbb{R};H)$ by
	\begin{equation*}
		\mathcal{G}^{0}\left(\int_{-\infty}^{\cdot}v(s)\mathrm{d}s\right):=\textbf{u}^{*}_{v}(\cdot).
	\end{equation*}
	
	By using Theorem \ref{stationary LDP},
	the family  $\{\textbf{u}^{*}_\varepsilon:\varepsilon>0\}$ satisfies the LDP in $C(\mathbb{R};H)$ with rate function
	\begin{equation}\label{ratefunction}
		I(f)=\inf_{\left\{v\in L^2(\mathbb{R};H_0):f
			=\mathcal{G}^0(\int_{-\infty}^{\cdot}v(s)\mathrm{d}s)\right\}}
		\left\{\frac{1}{2}\int_{-\infty}^{+\infty}\|v(s)\|_{H_0}^2\mathrm{d}s\right\},
	\end{equation}
	where the infimum over an empty set is taken as $+\infty$.
\end{example}
The following two examples show that we can also study the LDP for random periodic solutions and random quasi-periodic solutions. We give the definition of random periodic solution (see, for example, \cite{FZB11}).
\begin{definition}
	A random periodic solution of period $T$ of the random dynamical system $\Phi: \mathbb{R}^{+} \times$ $\Omega \times \mathbb{X} \rightarrow \mathbb{X}$ is an $\mathcal{F}$-measurable map $Y: \mathbb{R} \times \Omega \rightarrow \mathbb{X}$ such that for almost all $\omega \in \Omega$,
	\begin{equation}\label{rpd}
		\Phi(t, \theta(s,\omega) ) Y(s, \omega)=Y(t+s, \omega), Y(s+T, \omega)=Y(s, \theta(T,\omega)),
	\end{equation}
	for any $t \in \mathbb{R}^{+}, s \in \mathbb{R}$. It is called a random periodic solution with the minimal period $T$ if $T>0$ is the smallest number such that \eqref{rpd} holds.
\end{definition}
The following lemma is the result of the LDP for $n$-dimensional Brownian motion.\\
For convenience, let $E_n$ be the space $C(\mathbb{R};\mathbb{R}^n)$ with norm
\begin{equation}\label{norm1}
	\|B\|_{\mathcal{E}_n}:=\sup_{t\in\mathbb{R}}
	\frac{\|B_t\|_{\mathbb{R}^n}}{1+|t|},
\end{equation}
where $\|x\|_{\mathbb{R}^n}^2:=\sum_{i=1}^{n}|x_i|^2, x=(x_1, x_2, \cdots, x_n)\in\mathbb{R}^n$.
\begin{lemma}(cf. Schilder's theorem in section 1.3 of \cite{DS84})
	$n$-dimensional Brownian motion $B$ satisfies the LDP in space $E_1$ with rate function
	\begin{equation}\label{J}
		J(\phi)=\frac{1}{2}\int_{-\infty}^{+\infty}\|\dot{\phi}(s)\|_{\mathbb{R}^n}^2\mathrm{d}s, \quad  \phi\in C(\mathbb{R};\mathbb{R}^n),
	\end{equation}
	where $\dot{\phi}$  be the derivative of $\phi$ with respect to time $t$.
\end{lemma}

\begin{example}\label{perio}
	Consider the one-dimension equation shown by Feng, Liu and Zhao in \cite{FZ15},
	\begin{equation}\label{periodic eq}
		\mathrm{d}x=-5x\mathrm{d}t+(\sin(x) + 0.3\sin(2\pi t))\mathrm{d}t+\sqrt{\varepsilon} \mathrm{d}B_t,
	\end{equation}
	where $B_t$ is one-dimension Brownian motion.\\
 According to \cite{FZ15}, for every $\varepsilon>0$, the random periodic solution is
	\begin{equation*}
		\tilde{x}_{\varepsilon}(t)=\int_{-\infty}^{t}e^{-5(t-s)}(\sin(\tilde{x}_{\varepsilon}) + 0.3\sin(2\pi t))\mathrm{d}s+\sqrt{\varepsilon}\int_{-\infty}^{t}e^{-5(t-s)} \mathrm{d}B_s.
	\end{equation*}
	
	Let $G_1$ be the space $C(\mathbb{R};\mathbb{R})$ with another norm
	\begin{equation}\label{norm2}
		\|g\|_{\mathcal{C}_1}=
		\sum_{k=1}^{\infty}2^{-k}\left(\sup_{s\in[-k,k]}\left\|
		g(s)\right\|_{\mathbb{R}}\wedge1\right),\quad  g\in C(\mathbb{R};\mathbb{R}).
	\end{equation}
	Let  $\tilde{x}:=\tilde{x}_{\varepsilon=1}$, and $f:E_1\rightarrow G_1$ defined by $$f(\omega)(t):=\int_{-\infty}^{t}e^{-5(t-s)}
	(\sin(\tilde{x}) + 0.3\sin(2\pi t))\mathrm{d}s+\int_{-\infty}^{t}e^{-5(t-s)} \mathrm{d}\omega(s).$$
	For any  $\omega_1,\omega_2\in E_1$, then after simple calculation, we obtain that
	\begin{align*}
		&|f(\omega_1)(t)-f(\omega_2)(t)|\\
		\leq&\left\{\left|\int_{-\infty}^{t}e^{-5(t-r)}
		[\sin(\tilde{x}(r,\omega_1))-\sin(\tilde{x}(r,\omega_2))]\mathrm{d}r\right|\right.\\
		&+\left.\left|\int_{-\infty}^{t}e^{-5(t-r)}\mathrm{d}\omega_1(r)-\int_{-\infty}^{t}e^{-5(t-r)}\mathrm{d}\omega_2(r)\right|\right\}\\
		\leq& \left\{\left|\int_{-\infty}^{t}e^{-5(t-r)}\left\{
		[\tilde{x}(r,\omega_1)-\tilde{x}(r,\omega_2)]\right\}\mathrm{d}r\right|\right.\\
		&+\left.\left|\int_{-\infty}^{t}e^{-5(t-r)}\mathrm{d}\omega_1(r)-\int_{-\infty}^{t}e^{-5(t-r)}\mathrm{d}\omega_2(r)\right|\right\},\\
	\end{align*}
	by using the Grownall inequality, there exists a constant $C$ such that
	\begin{equation}\label{p gn}
		|F(\omega_1)(t)-F(\omega_2)(t)|
		\leq C \left|\int_{-\infty}^{t}e^{-5(t-r)}\mathrm{d}\omega_1(r)-\int_{-\infty}^{t}e^{-5(t-r)}\mathrm{d}\omega_2(r)\right|.
	\end{equation}
	
	It follows from the integration by parts that
	\begin{align*}
		&\left|\int_{-\infty}^{t}e^{-5(t-r)}\mathrm{d}\omega_1(r)-\int_{-\infty}^{t}e^{-5(t-r)}\mathrm{d}\omega_2(r)\right|\\
		=&\left|(\omega_1(t)-\omega_2(t))-5\int_{-\infty}^{t} e^{-5(t-r)}\omega_1(r)\mathrm{d}r
		+5\int_{-\infty}^{t} e^{-5(t-r)}\omega_2(r)\mathrm{d}r\right|\\
		\leq&|\omega_1(t)-\omega_2(t)|
		+5
		\int_{-\infty}^{t} e^{5(t-r)}|\omega_1(r)-\omega_2(r)|\mathrm{d}r,
	\end{align*}
	for any $N\in\mathbb{N}^+$, it implies that there exists a constant only depends on $N$ such that
	\begin{equation}\label{oc}
		\sup_{t\in[-N,N]}\left|\int_{-\infty}^{t}e^{-5(t-r)}d\omega_1(r)
		-\int_{-\infty}^{t}e^{-5(t-r)}d\omega_2(r)\right|^2\leq C(N)\|\omega_1-\omega_2\|_{\mathcal{E}_1}.
	\end{equation}
	For any $\eta>0$, we choose $N$ such that $\sum_{k=N}^{+\infty}2^{-k}\leq \frac{\eta}{2}$,  it follows from \eqref{p gn} and \eqref{oc} that
	\begin{align*}
		\|f(\omega_1)-f(\omega_2)\|_{\mathcal{C}_1}\leq&\sum_{k=1}^{N}2^{-k}\left(\sup _{s \in[-k, k]}|f(\omega_1)(s)-f(\omega_2)(s)| \wedge 1\right)+\sum_{k=N}^{+\infty}2^{-k}\\
		\leq &\sup _{s \in[-N, N]}|f(\omega_1)(s)-f(\omega_2)(s)| \wedge 1+\frac{\eta}{2}\\
		\leq & C(N)\|\omega_1-\omega_2\|_{\mathcal{C}_1}\wedge 1+\frac{\eta}{2},
	\end{align*}
	then there exists $\delta\leq \frac{\eta}{2C(N)}$ such that for any $\|\omega_1-\omega_2\|_{\mathcal{E}_1}\leq \delta$, it deduces that $\|f(\omega_1)-f(\omega_2)\|_{\mathcal{C}_1}\leq \eta$. It imply that $f$ is continuous, by using the contraction principle, we obtain that the  random periodic solutions of Eq. \eqref{periodic eq}
	satisfies the LDP in space $G_1$ with rate function
	\begin{equation*}
		I(u)=\frac{1}{2} \int_{-\infty}^{+\infty}\left| \dot{u}(t)+5u(t)-\sin(u(t)) - 0.3\sin(2\pi t)\right|^{2} d t,
	\end{equation*}
	where $\dot{u}$  be the derivative of $u$ with respect to time $t$.
	
	The following graphs with respect to the numerical of random periodic solutions of Eq. \eqref{periodic eq} with $\varepsilon=0,0.01,1$.\\
	\includegraphics[width=5cm,height=5cm]{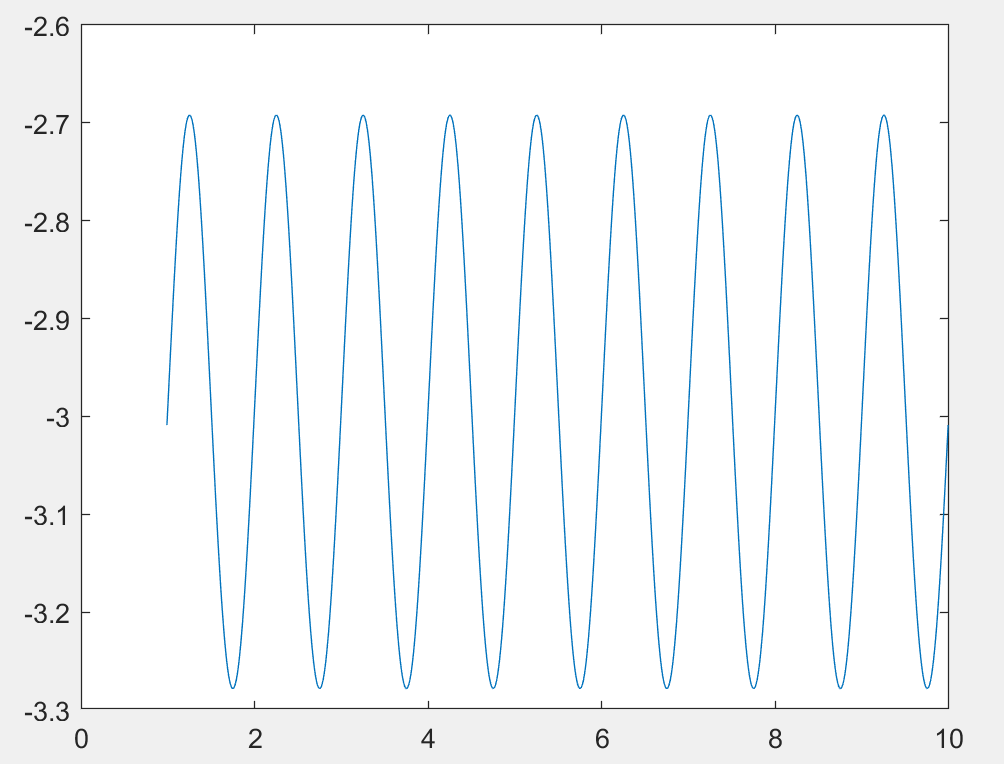}
	\includegraphics[width=5cm,height=5cm]{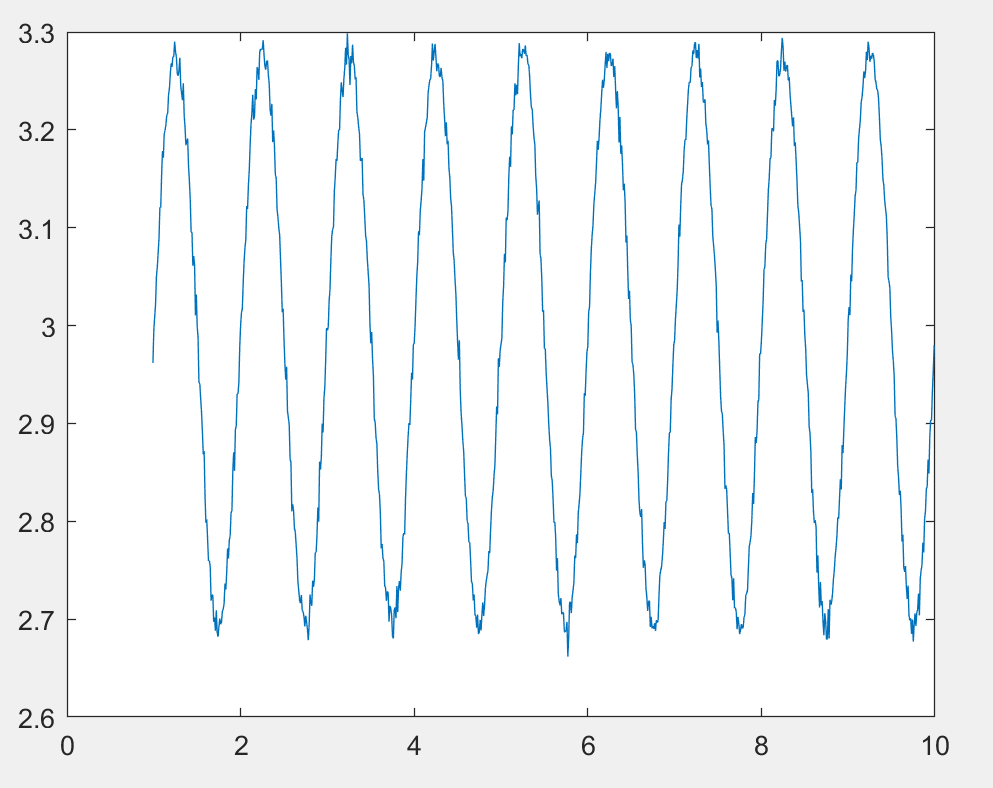}
	\includegraphics[width=5cm,height=5cm]{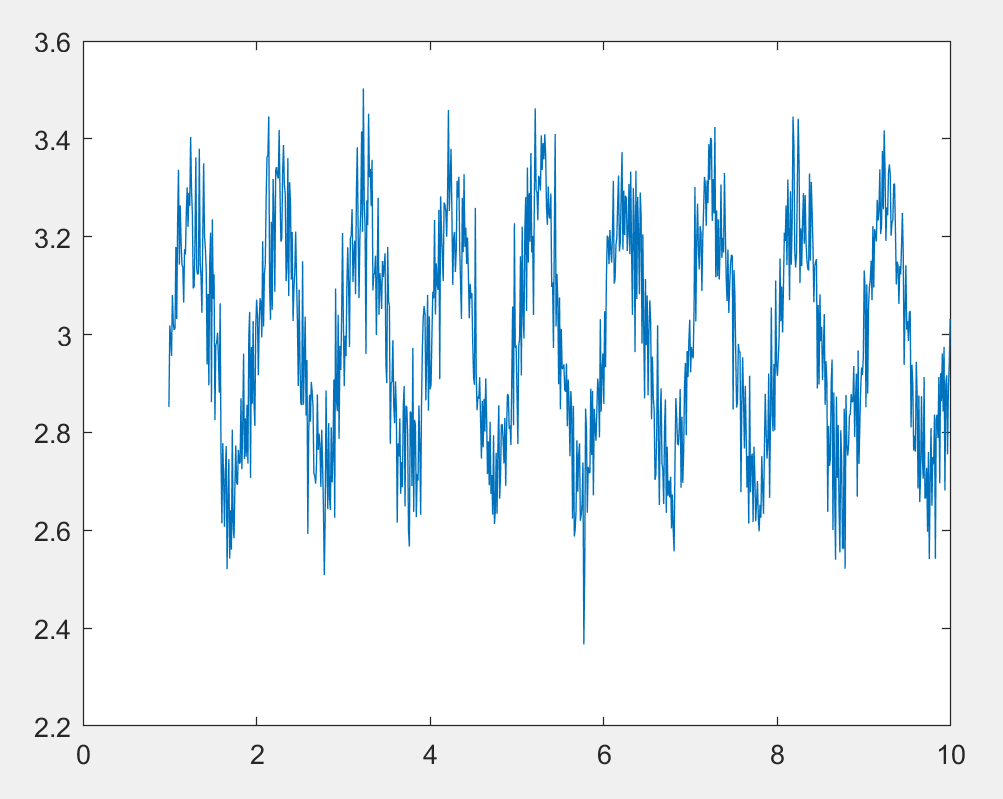}\\
\end{example}
The following example is to show that we can also study the LDP for random quasi-periodic solutions. We give the definition of random quasi-periodic solution (see, for example, \cite{FQZ21}).
\begin{definition}
	Let $F({\omega},z)$ be a measurable map from $\Omega \times T^{d}$ to $H$, and $\phi(t, \omega)$ be a random dynamical systems from $H$ to $H$. Let $\alpha=\left(\alpha_{1}, \cdots, \alpha_{d}\right)$ be a d-dimension vector, which is rationally independent. Then, we say $\phi(t, \omega)$ has a random quasi periodic solution $F({\omega},\alpha \cdot t)$, if they satisfy\\
	(i). (shift invariant of orbit) $\phi(t, \omega) F(\omega,z)=F(\theta({t} ,\omega),z)$.\\
	(ii). (quasi periodic property) $\phi(t, \omega) F(\omega,z)=F(\theta({t} ,\omega),\alpha \cdot t+z)\quad a . s .$
\end{definition}

\begin{example}\label{quasi-periodic solution}
	We consider a random Hopf's
	bifurcation  model for turbulence.
	The original  deterministic model given by Hopf in \cite{H56} is
	\begin{equation}\label{hopf}
		\frac{\partial \rho(x, t)}{\partial t}=I(\rho)+L(\rho)+\mu \frac{\partial^{2} \rho(x, t)}{\partial x^{2}},
	\end{equation}
	where $\rho=u+\mathrm{i} v$, and
	$$
	\begin{aligned}
		&L(\rho)=\frac{1}{2 \pi} \int_{0}^{2 \pi} \rho(x+y) \bar{F}(y) \mathrm{d} y ,\\
		&I(\rho)=-\frac{1}{4 \pi^{2}} \int_{0}^{2 \pi} \int_{0}^{2 \pi} \rho(y) \rho\left(y^{\prime}\right) \bar{\rho}\left(y+y^{\prime}-x\right) \mathrm{d} y \mathrm{d} y^{\prime}.
	\end{aligned}
	$$
	We denote $\overline{u+\mathrm{i} v}=u-\mathrm{i} v$. Here, $F$ is regarded as the external force acting on the ``velocity field" $\rho$.  By the Fourier transform, we know that
	$$
	\rho(x, t)=\sum_{n=-\infty}^{\infty} \rho_{n}(t){\rm e}_n(x), \text { and } F(x)=\sum_{n=-\infty}^{\infty} F_{n}e_n(x),
	$$
	where $\{e_n(x)=e^{{\rm i}nx},x\in S^1,n\in\mathbb{Z}\}$ is the orthonormal basis of $L^2(S^1)$, and $L^2(S^1)$ with norm
	\begin{equation*}
		\|h\|_{L^2(S^1)}:=\sum_{n=-\infty}^{\infty}| h_{n}|,\quad  h=\sum_{n=-\infty}^{\infty}h_n {\rm e}_n(x)\in L^2(S^1).
	\end{equation*}
	
	So, we get
	$$
	\frac{\mathrm{d} \rho_{n}(t)}{d t}=-\rho_{n}^{2} \overline{\rho_{n}}+\left(\overline{F_{n}}-\mu n^{2}\right) \rho_{n}, \quad F_{n}=a_{n}-\mathrm{i} b_{n}.
	$$
	Thus, we consider the following equation
	$$
	\frac{\mathrm{d} \rho(t)}{d t}=-\rho^{2} \bar{\rho}+(\bar{F}-\nu) \rho, \quad F=a-\mathrm{i} b ,
	$$
	let $\rho=r e^{i\theta}$, it is equivalent to
	\begin{equation}\label{hopf1}
		\left\{\begin{array}{l}
			\frac{\mathrm{d} r (t)}{\mathrm{d} t}=\left(a-\nu-r ^{2}\right) r  ,\\
			\frac{\mathrm{d} \theta(t)}{\mathrm{d} t}=b,
		\end{array}\right.
	\end{equation}
	where
	\begin{align*}
		\nu=\mu n^{2}, \quad a=a_{n}=\operatorname{Re} F_{n}, \quad b=b_{n}=-\operatorname{Im} F_{n}, \quad
		r =r _{n}, \quad \theta=\theta_{n}, \quad \rho_{n}=r _{n} e^{{\rm i} \theta_{n}}.
	\end{align*}
	
	Now, we consider that  $\rho$ is perturbed by random external force $F$. For simplicity and illustrating our ideas, we assume that  $F=\sum_{n} (a_n+{\rm i}b_n)\mathrm{e}_n + W$, where  $W_{t}(x)=\sum_{n}c_{n} B_{n}(t) \mathrm{e}_{n}(x)$ and $\{B_{n}\}$ is a sequence independent Brownian motion, then  Hopf's model \eqref{hopf} or \eqref{hopf1} with random external force is equivalent to the following stochastic equation
\begin{numcases}{}
	\frac{\mathrm{d} r (t)}{\mathrm{d} t}=\left(a-\nu+\frac{1}{2}c^2-r ^{2}\right) r  +cr \mathrm{d}B(t),\label{hopf1}\\
			\mathrm{d} \theta(t)=b\mathrm{d} t\label{hopf2},
\end{numcases}
	where
	\begin{align*}
		\nu=\mu n^{2}, \quad a=a_{n}, \quad b=b_{n}, \quad c=c_n, \quad
		r =r _{n}, \quad \theta=\theta_{n}, \quad \rho_{n}=r _{n} e^{{\rm i} \theta_{n}} ,\quad B(t)=B_n(t).
	\end{align*}
	For some more complex random Hopf's models, please see \cite{Bai19}.
	
	The Eq. \eqref{hopf1}
	has the non-trivial stationary solution
	$$
	r ^{*}(\omega)=\left(\int_{-\infty}^{0} e^{2(a-\nu) s+2 c B(s)(\omega)} \mathrm{d} s\right)^{-\frac{1}{2}}, \quad  \text { if } a>\nu,
	$$
	and unique trivial stationary solution
	$$
	r ^{*}(\omega)=0, \quad  \text { if } a\leq\nu.
	$$
	
	Moreover, assume $F_{n}=0, n \leq 0$, and $\frac{a_{n}}{n^{2}} \downarrow 0(n>0)$, $b_{1}, b_{2}, \cdots, b_{n}, \cdots$  is rationally
	independent. Let $\mu_{n}=\frac{a_{n}}{n^{2}}(n>0)$.
	Set $\theta_{k}=b_{k} t+\theta^0_{k}$, if $\mu_{m}>\mu \geq \mu_{m+1}$ for some fixed $ m\in\mathbb{N}^+$, then the solution  $\rho(x, t)$ of Eq. \eqref{hopf} converges to a random quasi-periodic solution $\rho\left(\theta_j, 1\leq j\leq m; \mu,\omega\right)=\sum_{j=1}^mr_j^*(t,\omega)e^{{\rm i}(\theta^{0}_j+b_j t)}{\rm e}_j(x)$  with the angle variables $\theta$ form a manifold of the type of a $m$-d torus $\mathbb{T}^m$.
	Specially, we can choose  $c_j=\sqrt{\varepsilon},1\leq j\leq m$, $c_{m+1}=c_{m+2}=\cdots=0$, $a_j-j^2\mu=\frac{3}{2}-\frac{1}{2}c^2,1\leq j\leq m$,
	let
	$$
	r _{\varepsilon,j}^{*}(t,\omega)=\frac{1}{\left[2 \int_{-\infty}^{t} e^{(3-\varepsilon) s+2\sqrt{\varepsilon} B_j(s)(\omega)} \mathrm{d} s\right]^{\frac{1}{2}}},
	$$
where $B_j(s)(\omega)=\omega_j(s)$, $\omega=(\omega_1,\omega_1,\cdots,\omega_m)$, $j=1,2,\cdots,m$.
	Then for every $\varepsilon>0$, the solution  $\rho_{\varepsilon}(x, t)$ of Eq. \eqref{hopf} converges to a random quasi-periodic solution $\rho_{\varepsilon}:\Omega\times\mathbb{T}^m\rightarrow L^2{(S^1)}$
	\begin{equation*}
		\rho_{\varepsilon}\left(\theta_{j}, 1\leq j\leq m; \mu,\omega\right):=\sum_{j=1}^m r _{\varepsilon,j}^*(t,\omega)e^{{\rm i}(\theta^{0}_j+b_j t)}{\rm e}_j(x),
	\end{equation*}
	where  $\|\rho_{\varepsilon}\left(\theta_{j},1\leq j\leq m; \mu,\omega\right)\|_{L^2{(S^1)}}=\sum_{j=1}^m|r _{\varepsilon,j}^*(t,\omega)|$.
	
	Let $G_2$ be the space $C(\mathbb{R};L^2(S^1))$ with another norm
	\begin{equation}\label{norm3}
		\|h\|_{\mathcal{C}_2}=
		\sum_{k=1}^{\infty}2^{-k}\left(\sup_{s\in[-k,k]}\left\|
		h(s)\right\|_{L^2(S^1)}\wedge1\right),\quad  h\in C(\mathbb{R};L^2(S^1)).
	\end{equation}
	Let  $r ^*_i:=r ^*_{\varepsilon=1,i},i=1,2,...,m$, and $g:E_m\rightarrow G_2$ defined by
	$$g(\omega)(t):=\sum_{j=1}^m r _{j}^*(t,\omega)e^{{\rm i}(\theta^{0}_j+b_j t)}{\rm e}_j(x),$$
	where $\omega=(\omega_1,\omega_2\cdots,\omega_m)\in E_m$.
	We prove the function $g$ is continuous in the following.
	
	For any  $\omega^1=(\omega_1^1,\omega_2^1,\cdots,\omega_m^1),\omega^2=(\omega_1^2,\omega_2^2,\cdots,\omega_m^2)\in E_m$, then after simple calculation, we obtain that
	\begin{equation}\label{Hq1}
		\begin{aligned}
			\|g(\omega^1)(t)-g(\omega^2)(t)\|_{L^2{(S^1)}}=\sum_{j=1}^m
			\left|r _j^*(t,\omega_j^1)-r _j^*(t,\omega_j^2) \right|.
		\end{aligned}
	\end{equation}
	For any $\eta>0$, we choose $N$ big enough such that $\sum_{k=N}^{+\infty}2^{-k}\leq \frac{\eta}{2}$,  it follows from Eqs. \eqref{norm3} and \eqref{Hq1} that
	\begin{align*}
		\|g(\omega_1)-g(\omega_2)\|_{\mathcal{C}_2}\leq&\sum_{k=1}^{N}2^{-k}\left(\sup _{s \in[-k, k]}\|g(\omega_1)(s)-g(\omega_2)(s)\|_{L^2{(S^1)}} \wedge 1\right)+\sum_{k=N}^{+\infty}2^{-k}\\
\leq&\sup _{s \in[-N, N]}\|g(\omega_1)(s)-g(\omega_2)(s)\|_{L^2{(S^1)}} \wedge 1+\sum_{k=N}^{+\infty}2^{-k}\\
		\leq &\sum_{j=1}^m\sup _{s \in[-N, N]}\left(\left|r _j^*(t,\omega_j^1)-r _j^*(t,\omega_j^2) \right|\right) \wedge 1+\frac{\eta}{2}.
	\end{align*}
Define $(\mathbb{R}^{+})^m$ be the product space of $\mathbb{R}^{+}$ times $m$,
	let $f:E_m\rightarrow (\mathbb{R}^{+})^m$ be
	$$
\begin{aligned} f(\omega)(t)=&(f_1(\omega)(t),f_2(\omega)(t),\cdots,f_m(\omega)(t))\\
: =& \left( \int_{-\infty}^{t} e^{2 s+2 B_1(s)(\omega)} \mathrm{d} s,\int_{-\infty}^{t} e^{2 s+2 B_2(s)(\omega)} \mathrm{d} s,\cdots,\int_{-\infty}^{t} e^{2 s+2 B_m(s)(\omega)} \mathrm{d} s\right),
\end{aligned}
	$$
it is sufficient to prove that $f|_{[-N,N]}$ is a  continuous function from $E_m$ to space $C([-N,N];(\mathbb{R}^{+})^m)$ with norm $\|f\|_{C([-N,N];\mathbb{R}^+\times \mathbb{R}^+)}:=\sum_{j=1}^{m}\sup_{t\in[-N,N]}|f_j(t)|$, for any $N\in\mathbb{N}^+$.

	For any $N\in\mathbb{N}^+$, $\delta>0$, $\omega^1\in E_m$ satisfies $\|\omega^1\|_{\mathcal{E}_m}\leq K$ and $\|\omega^1-\omega^2\|_{\mathcal{E}_m}\leq \delta$, we have
	\begin{equation}\label{Hf}
		\begin{aligned}
			&|f_i(\omega^1)(t)-f_i(\omega^2)(t)|=\left|\int_{-\infty}^{t} e^{2 s+2 \omega_i^1(s)} \mathrm{d} s-\int_{-\infty}^{t} e^{2 s+2 \omega_i^2(s)} \mathrm{d} s\right|\\
			=&\left|\int_{-\infty}^{t\wedge 0} e^{2 s+2 \omega_i^1(s)} \mathrm{d} s-\int_{-\infty}^{t\wedge 0} e^{2 s+2 \omega_i^2(s)} \mathrm{d} s\right|+\left|\int_{0}^{t\vee 0} e^{2 s+2 \omega_i^1(s)} \mathrm{d} s-\int_{0}^{t\vee 0} e^{2 s+2 \omega_i^2(s)} \mathrm{d} s\right|\\
			:=&I_i^1+I_i^2,\quad i=1,2,\cdots,m.
		\end{aligned}
	\end{equation}
	We first estimate $I_i^1,i=1,2,\cdots,m$, choose $\delta$ small enough such that $e^{K-1+2\delta}-e^{K-1+\delta}<1$ then
	\begin{equation}\label{HI1}
		\begin{aligned}
			I_i^1=&\left|\int_{-\infty}^{t\wedge 0} e^{2 s}\left[\exp\left\{\frac{ \omega_i^1(s)}{(1-s)}\right\}-\exp\left\{\frac{ \omega_i^2(s)}{(1-s)}\right\}\right]^{2(1-s)} \mathrm{d} s\right|\\
			\leq&\left|\int_{-\infty}^{t\wedge 0} e^{2 s}\left[\exp\left\{\frac{ \omega_i^2(s)}{(1-s)}+\delta\right\}-\exp\left\{\frac{ \omega_i^2(s)}{(1-s)}\right\}\right]^{2(1-s)} \mathrm{d} s\right|\\
			\leq&\left|\int_{-\infty}^{t\wedge 0} e^{2 s}\left[e^{K+2\delta}-e^{K+\delta}\right]^{2(1-s)} \mathrm{d} s\right|\\
			=&\left[e^{K+2\delta}-e^{K+\delta}\right]^2\left|\int_{-\infty}^{t\wedge 0} \left[e^{K-1+2\delta}-e^{K-1+\delta}\right]^{-2s} \mathrm{d} s\right|\\
			=&\frac{1}{2}\left(e^{K+2\delta}-e^{K+\delta}\right)^2\left(e^{K-1+2\delta}-e^{K-1+\delta}\right)^{-2(t\wedge 0)}\left[\ln \left(e^{K-1+2\delta}-e^{K-1+\delta}\right)\right]^{-1}\\
			\leq&\frac{1}{2}\left(e^{K+2\delta}-e^{K+\delta}\right)^2\left[\ln \left(e^{K-1+2\delta}-e^{K-1+\delta}\right)\right]^{-1}.
		\end{aligned}
	\end{equation}
	Similar estimate for $I_i^2,i=1,2,\cdots,m$, choose $\delta$ small enough such that $e^{K+1+2\delta}-e^{K+1+\delta}<1$, we get that
	\begin{equation}\label{HI2}
		\begin{aligned}
			I_i^2=&\left|\int_{0}^{t\vee 0} e^{2 s}\left[\exp\left\{\frac{ \omega_i^1(s)}{(1+s)}\right\}-\exp\left\{\frac{ \omega_i^2(s)}{(1+s)}\right\}\right]^{2(1+s)} \mathrm{d} s\right|\\
			\leq&\left|\int_{0}^{t\vee 0} e^{2 s}\left[\exp\left\{\frac{ \omega_i^2(s)}{(1+s)}+\delta\right\}-\exp\left\{\frac{ \omega_i^2(s)}{(1+s)}\right\}\right]^{2(1+s)} \mathrm{d} s\right|\\
			\leq&\left|\int_{0}^{t\vee 0} e^{2 s}\left[e^{K+2\delta}-e^{K+\delta}\right]^{2(1+s)} \mathrm{d} s\right|\\
			=&\left[e^{K+2\delta}-e^{K+\delta}\right]^2\left|\int_{0}^{t\vee 0} \left[e^{K+1+2\delta}-e^{K+1+\delta}\right]^{2s} \mathrm{d} s\right|\\
			=&\frac{1}{2}\left(e^{K+2\delta}-e^{K+\delta}\right)^2\left[\left(e^{K+1+2\delta}-e^{K+1+\delta}\right)^{2(t\wedge 0)}-1\right]\left[\ln \left(e^{K+1+2\delta}-e^{K+1+\delta}\right)\right]^{-1}\\
			\leq&\frac{1}{2}\left(e^{K+2\delta}-e^{K+\delta}\right)^2\left[\ln \left(e^{K+1+2\delta}-e^{K+1+\delta}\right)\right]^{-1}.
		\end{aligned}
	\end{equation}
	It follows from  \eqref{Hf}, \eqref{HI1} and \eqref{HI2} that $f|_{[-N,N]}$ is a  continuous function from $E_m$ to space $C([-N,N];(\mathbb{R}^+)^m)$, thus $g$ is a continuous function from $E_m$ to $G_2$.
	
	By using the contraction principle, we obtain that the random quasi-periodic solutions of Eq. \eqref{hopf}
	satisfies the LDP in space $G_2$ with rate function
	\begin{equation*}
		\begin{aligned}
			I(\rho)=\inf\left\{\frac{1}{2}\int_{-\infty}^{+\infty}\|\dot{\phi}(s)\|_{\mathbb{R}^m}^2\mathrm{d}s:\quad \phi\in E_m, \quad\rho=g(\phi)\right\}.
		\end{aligned}
	\end{equation*}
\end{example}

\section{Invariant measure, rate function and quasi-potential }\label{LDP for invariant measure}
\subsection{LDP for invariant measure}
In this section, we will use the contraction principle to prove that the LDP for the family of stationary solutions $\{X^{*}_{\varepsilon}\}_{\varepsilon>0}$ deduce the LDP for the family of invariant measures $\{\nu_{\varepsilon}\}_{\varepsilon>0}$ for \eqref{Msode}.

It follows from Lemma \ref{Mstationary solution} that there exists $\varepsilon_0>0$ such that
for $\varepsilon\in(0,\varepsilon_0)$,  \eqref{Msode} exists a unique  stationary solution $X^{*}_{\varepsilon}$. Let $\nu_{\varepsilon}$ be the distribution of the stationary solution $X^{*}_{\varepsilon}(\cdot,\omega)$ at time $0$, i.e.
\begin{equation*}
	\nu_{\varepsilon}(A):=\mathbb{P}
	\left\{\omega;X^{*}_{\varepsilon}(0,\omega)\in A\right\},\quad\forall A\in \mathcal{B}(H),
\end{equation*}
then $\nu_{\varepsilon}$ is an invariant measure of Eq. \eqref{Msode} in $H$.
Moreover, Lemma \ref{long-time} easily implies the invariant measure of Eq. \eqref{Msode} is unique in $H$, for any $\varepsilon\in(0,\varepsilon_0)$. Thus $\nu_{\varepsilon}$ is the unique invariant measure of Eq. \eqref{Msode} in $H$.

For any $\varepsilon\in(0,\varepsilon_0)$, $X^{*}_{\varepsilon}(0,\omega)$ is an $\mathcal{F}_{-\infty}^0$-measurable random variable, therefore, in this section, we only consider the LDP of stationary solutions for Eq. \eqref{Msode} in space $C((-\infty,0];H)$ with the norm
\begin{equation}\label{Cnorm2}
	\|\cdot\|_{C((-\infty,0];H)}=
	\sum_{k=1}^{\infty}2^{-k}\left(\sup_{s\in[-k,0]}\left\|
	\cdot\right\|_{H}\wedge1\right).
\end{equation}
Similar to the proof in Section \ref{LDP for stationary solution}, we could have the following theorem.
\begin{theorem}\label{stationary LDP2}
	Under the Hypothesis \ref{Msodehy},
	the family $\{X^{*}_\varepsilon:\varepsilon>0\}$ satisfies the LDP in $C((-\infty,0];H)$ with good rate function
	\begin{equation}\label{ratefunction2}
		\tilde{I}(f)=\inf_{\left\{v\in L^2((-\infty,0];H_0):f
			=\mathcal{G}^0(\int_{-\infty}^{\cdot}v(s)\mathrm{d}s)\right\}}
		\left\{\frac{1}{2}\int_{-\infty}^{0}\|v(s)\|_{H_0}^2\mathrm{d}s\right\},
	\end{equation}
	where the infimum over an empty set is taken as $+\infty$.
\end{theorem}
We could get the following Theorem by using the contraction principle (cf. Theorem 4.2.1 in \cite{DZ07}).
\begin{theorem}\label{stoi}
	Under the Hypothesis \ref{Msodehy}, the family of invariant measures $\left\{\nu_{\varepsilon}\right\}_{\varepsilon>0}$ for Eq. \eqref{Msode} satisfies the LDP in $H$, with good rate function $I'$
	\begin{equation}\label{I'}
		I'(x)=\inf
		\left\{\frac{1}{2}\int_{-\infty}^{0}\|v(s)\|_{H_0}^2\mathrm{d}s:u
		=\mathcal{G}^0\left(\int_{-\infty}^{\cdot}v(s)\mathrm{d}s\right),u(0)=x\right\}.
	\end{equation}
\end{theorem}
\begin{proof}
	It follows from Theorem \ref{stationary LDP2} that the stationary solution family   $\left\{X^{*}_{\varepsilon}(\cdot,\omega)\right\}_{\varepsilon>0}$ for Eq. \eqref{Msode} satisfies the LDP in $C((-\infty,0];H)$ with rate function $\tilde{I}$.
	
	Let $G: C((-\infty,0];H)\longrightarrow H$ by $G(f)=f(0)$,
	it is obvious that $G$ is continuous, then it follows from the contraction principle that the invariant measure family $\left\{\nu_{\varepsilon}\right\}_{\varepsilon>0}$ of Eq. \eqref{Msode} satisfies the LDP in $H$, with rate function
	\begin{align*}
		I'(x)=&\inf_{\{x=G(u)\}}\{\tilde{I}(u)\}
		=
		\inf
		\left\{\frac{1}{2}\int_{-\infty}^{0}\|v(s)\|_{H_0}^2\mathrm{d}s:u
		=\mathcal{G}^0\left(\int_{-\infty}^{\cdot}v(s)\mathrm{d}s\right),u(0)=x\right\}.
	\end{align*}
\end{proof}
\subsection{Rate function and quasi-potential}
It is well known that \cite{FW12} and \cite{CR05} have proved the family of invariant measures for stochastic equations satisfies LDP  with  the quasi-potential as rate function. The definition of quasi-potential is given in Section \ref{The definition of quasi-potential}, we will prove the rate function $I'$ defined in \eqref{I'} and quasi-potential are equivalent below.

\begin{lemma}\label{quasi rate}
	The rate function \eqref{I'} defined in Theorem \ref{stoi} are equivalent to the rate function \eqref{V} defined by quasi-potential of the LDP for invariant measures of Eq. \eqref{Msode}.
\end{lemma}
\begin{proof}
	It follows from the definition of $\mathcal{G}^0$ and \eqref{V} that
	\begin{equation}\label{I eq V}
		\begin{aligned}
			I'(x) =&\inf
			\left\{\frac{1}{2}\int_{-\infty}^{0}\|v(s)\|_{H_0}^2\mathrm{d}s:u
			=\mathcal{G}^0\left(\int_{-\infty}^{\cdot}v(s)\mathrm{d}s\right),u(0)=x\right\}\\
			=&\inf\left\{\frac{1}{2}\int_{-\infty}^{0}\left\| v(t)\right\|_{H_0}^{2} \mathrm{d} t;u\in C((-\infty,0];H), u=u(v),u(0)=x,
			\lim_{t\rightarrow\infty}
			\|u(t)\|_{H}=0\right\}\\
			=&\inf\left\{S_{-\infty}(u);u\in C((-\infty,0];H),u(0)=x,
			\lim_{t\rightarrow-\infty}
			\|u(t)\|_{H}=0\right\}\\
			=&V(x),\quad\forall x\in H,\  V(x)<\infty.
		\end{aligned}
	\end{equation}
\end{proof}

\begin{remark}
	Theorem \ref{stoi} states that the LDP for the family of stationary solutions can deduce the LDP for the family of invariant measures of Eq. \eqref{Msode}.
	And for the rate function $I'$ consistent with  quasi-potential $V$, $I'$ defined in Theorem \ref{stoi} gives another explain of quasi-potential.
\end{remark}
The next example illustrate that there exists two systems with different dynamical behaviors but with the same invariant measure. It implies that the LDP for the stationary solutions gives more dynamical information than the LDP of invariant measures.
\begin{example}\label{EX same in}
	Define two $2\times 2$ metrics $A_1$ and $A_2$ by
	$$
	A_1=\begin{bmatrix}
		-\lambda& 0 \\
		0 & -\lambda
	\end{bmatrix},\quad
	A_2=\begin{bmatrix}
		-\lambda& -\beta \\
		\beta & -\lambda
	\end{bmatrix}.
	$$
	We consider two stochastic equations
	\begin{equation}\label{A1}
		\mathrm{d}X=A_1X\mathrm{d}t +\sqrt{\varepsilon}\mathrm{d}B_t,
	\end{equation}
	and
	\begin{equation}\label{A2}
		\mathrm{d}X=A_2X\mathrm{d}t +\sqrt{\varepsilon}\mathrm{d}B_t.
	\end{equation}
	It follows from Theorem \ref{stoi} that the family of invariant measures of Eqs. \eqref{A1} and \eqref{A2} satisfies LDP with good rate function as follows
	\begin{equation*}
		V_i(X)=\inf\left\{\frac{1}{2} \int_{-\infty}^{0}\left\| \dot{Y}(t)-A_iY(t)\right\|_{\mathbb{R}^2}^{2} \mathrm{d} t:Y(0)=X,\lim_{t\rightarrow -\infty}\|Y(t)\|_{\mathbb{R}^2}=0\right\},\quad i=1,2,
	\end{equation*}
	where $X=(x_1,x_2)$ with norm $\|X\|_{\mathbb{R}^2}^2:=x_1^2+x_2^2$, and $\dot{Y}$  be the derivative of $Y$ with respect to time $t$.\\
By using  Theorem 3.1 in \cite{FW12}, we can get that
	\begin{equation*}
		V_i(X)=\lambda (x_1^2+x_2^2),\quad i=1,2.
	\end{equation*}
	It implies that the rate functions for invariant measures of  two Eqs. \eqref{A1} and \eqref{A2} are equivalent. Furthermore, let $Z_{\varepsilon}:=\int_{\mathbb{R}^2}
	e^{-\frac{\lambda(x_1^2+x_2^2)}{\varepsilon}}\mathrm{d}X$, the two Eqs. \eqref{A1} and \eqref{A2} have the same invariant measure $\nu_{\varepsilon}(\cdot)=\frac{1}{Z_{\varepsilon}}\int_{\cdot}e^{-\frac{\lambda(x_1^2+x_2^2)}{\varepsilon}}\mathrm{d}X$, for every $\varepsilon$. However, we have known that the determinate equations $\mathrm{d}X=A_1X\mathrm{d}t$ and $\mathrm{d}X=A_2X\mathrm{d}t$ with respect to Eqs. \eqref{A1} and \eqref{A2} have different asymptotic behavior.
	
	Define the stationary solutions of Eqs. \eqref{A1} and \eqref{A2} by
	\begin{equation*}
		X^{*}_{i,\varepsilon}(\cdot)=\sqrt{\varepsilon}
		\int_{-\infty}^{\cdot}e^{A_i(\cdot-s)}\mathrm{d}B_s,\quad i=1,2.
	\end{equation*}
	It follows from Theorem \ref{LP} that the family of the stationary solutions $\{X^{*}_{i,\varepsilon}\}_{\varepsilon>0}$ of Eqs. \eqref{A1} and \eqref{A2} satisfies LDP in space $C(\mathbb{R};\mathbb{R}^2)$ with different rate function
	\begin{equation*}
		I_i(Y)=\frac{1}{2} \int_{-\infty}^{+\infty}\left\| \dot{Y}(t)-A_iY(t)\right\|_{\mathbb{R}^2}^{2} \mathrm{d} t,Y\in C(\mathbb{R};\mathbb{R}^2),\quad i=1,2.
	\end{equation*}
	It implies that it is meaningful to research the LDP for stationary solutions, which  gives more dynamical information. The numerical simulation in the following graphs will give a more intuitive explanation.
	We choose $\lambda=0.3,\beta =2$, the following three graphs are the numerical approximate of solutions for Eq. \eqref{A1} correspond to the cases $\varepsilon=0,\varepsilon=0.01,\varepsilon=1$.\\
	\includegraphics[width=5cm,height=5cm]{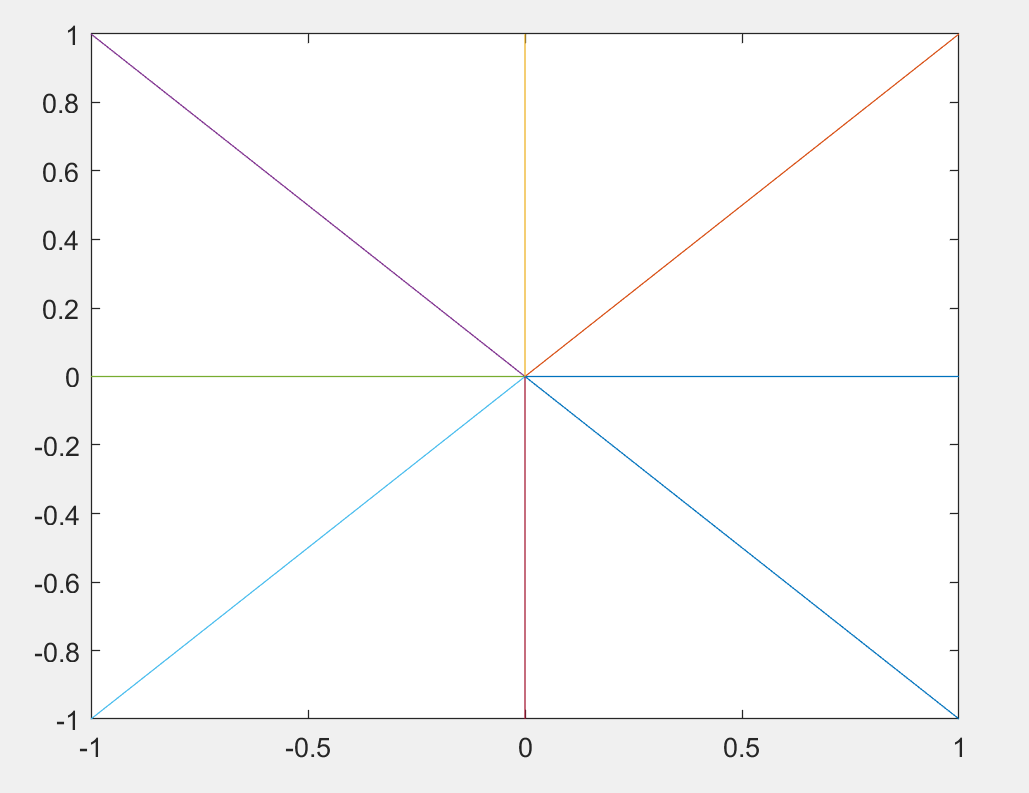}
	\includegraphics[width=5cm,height=5cm]{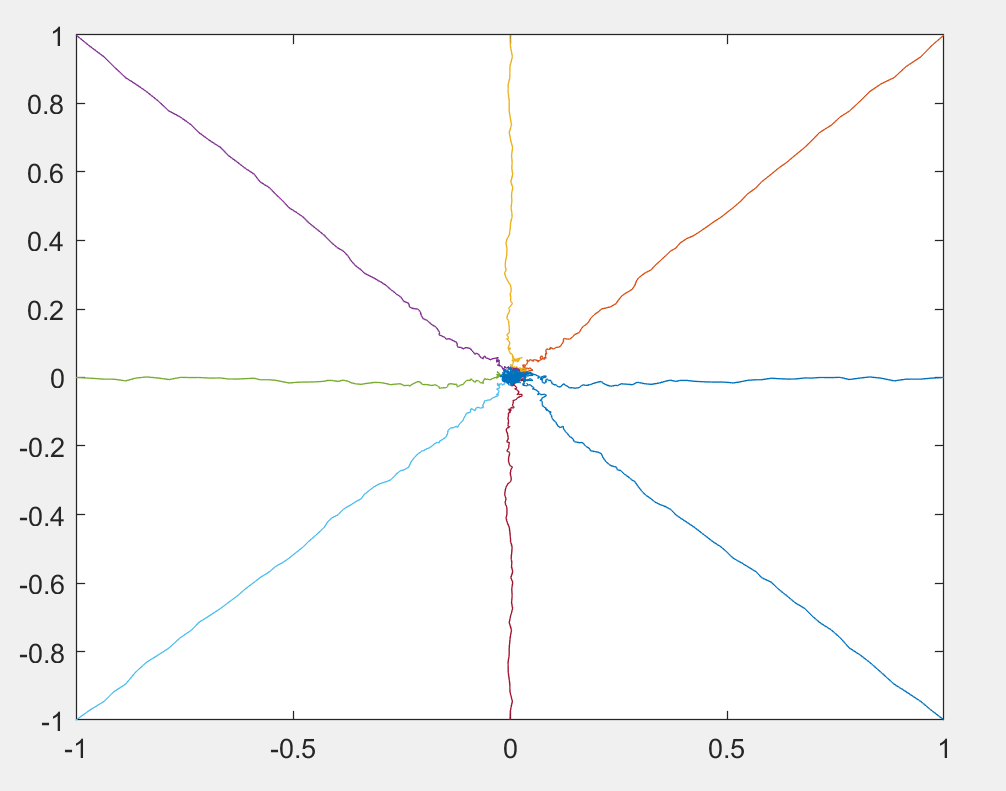}
	\includegraphics[width=5cm,height=5cm]{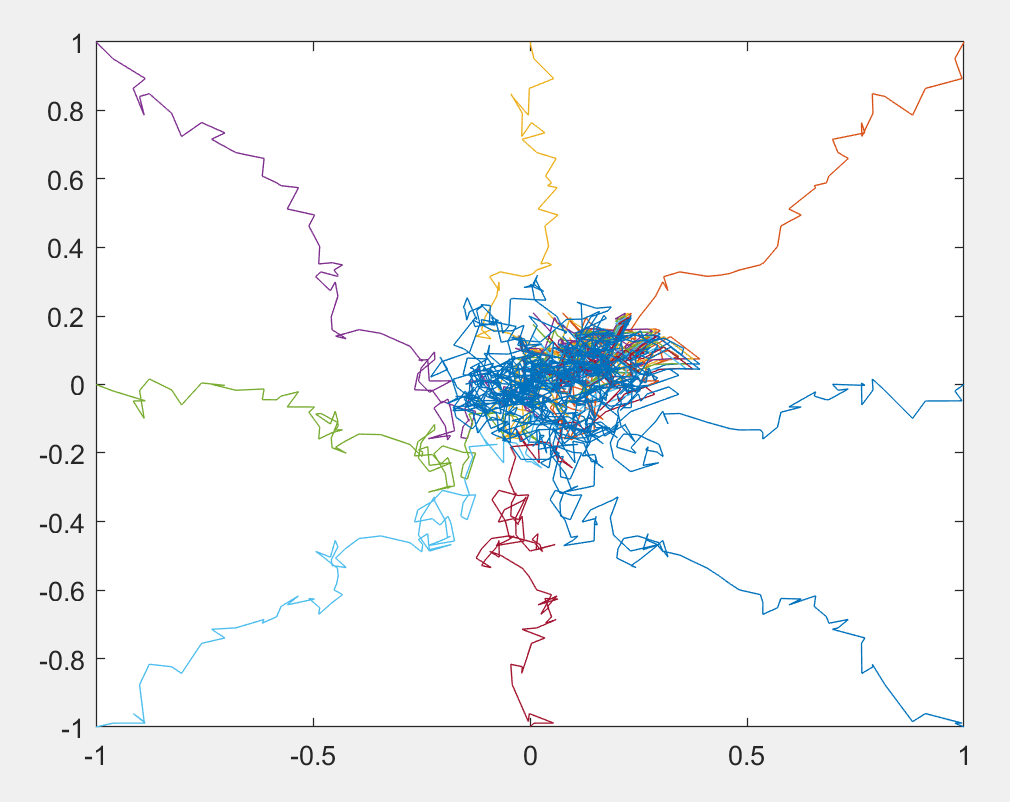}\\
	The following three graphs are the numerical approximate of solutions for Eq. \eqref{A1} correspond to the cases $\varepsilon=0,\varepsilon=0.01,\varepsilon=1$.\\
	\includegraphics[width=5cm,height=5cm]{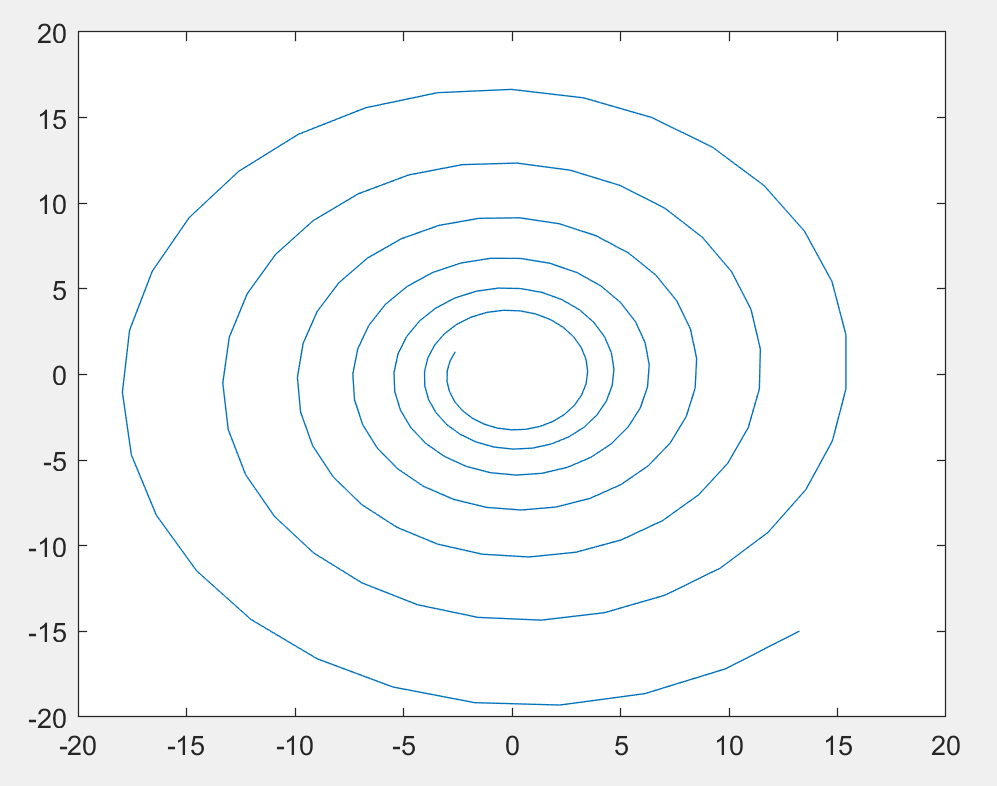}
	\includegraphics[width=5cm,height=5cm]{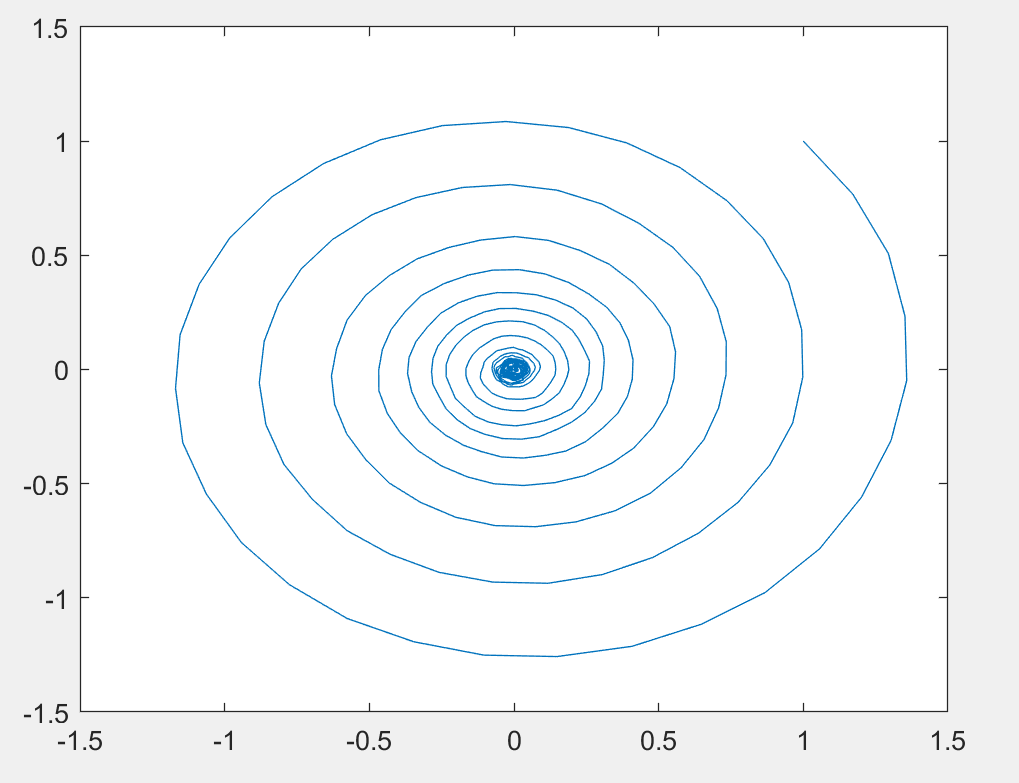}
	\includegraphics[width=5cm,height=5cm]{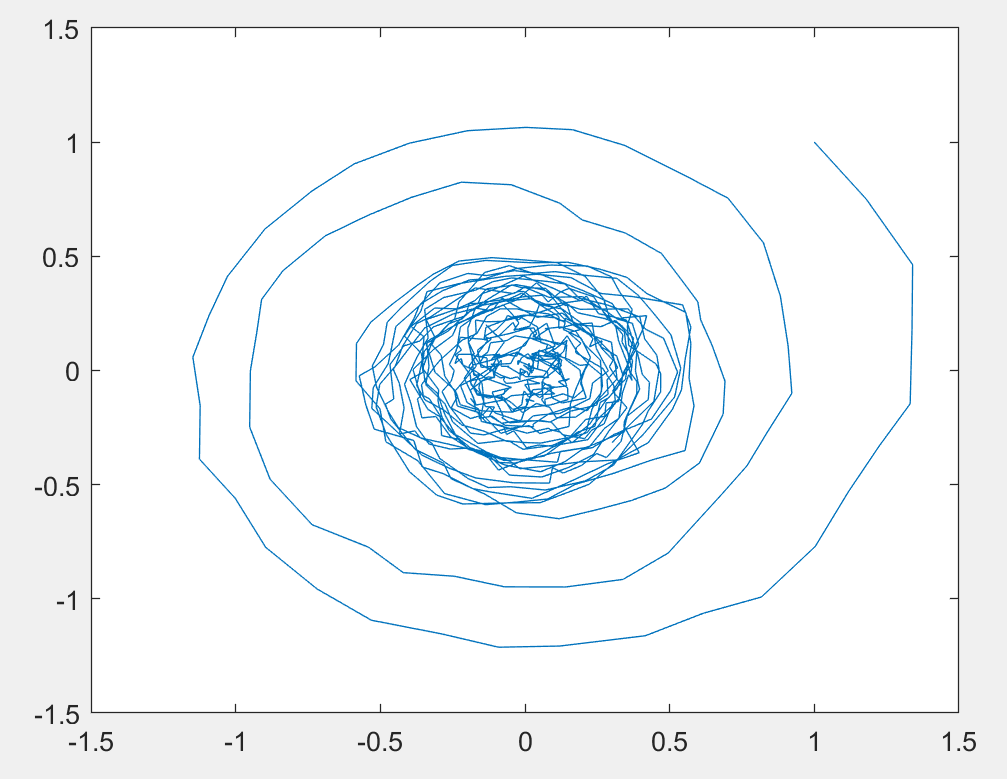}
\end{example}
Brze\'{z}niak and  Cerrai have researched the LDP for the invariant measures of the $2$-dimensional stochastic Navier-Stokes
equations on a torus in \cite{BC17}. We will consider the LDP for the solutions  of  the pullback integral equations of the $2$-dimensional stochastic Navier-Stokes
equations in the following.
\begin{example}\label{Navier-Stokes}
For convenience, we write $2$-dimensional stochastic Navier-Stokes equations perturbed by a small additive noise in a functional form shown by Brze\'{z}niak and  Cerrai in \cite{BC17}, as
\begin{equation}\label{NS}
	\mathrm{d} \textbf{u}(t)+\mathrm{A} \textbf{u}(t) \mathrm{d} t+\mathrm{B}(\textbf{u}(t), \textbf{u}(t)) \mathrm{d} t=\sqrt{\varepsilon} \mathrm{d} w_t, \quad \textbf{u}(0)=\textbf{u}_{0},
\end{equation}
for $0<\varepsilon<<1$ on a two-dimensional torus $\mathbb{T}^{2}$.

Let us recall that $A$ is the Stokes operator,  roughly speaking, equal to the Laplace operator composed with the Leray-Helmholtz projection $P$, the convection $\mathrm{B}(\textbf{u}, \textbf{u})$ is equal to $P(\textbf{u} \nabla \textbf{u})$, and $w(t)$ is a $Q$ Wiener process.\\
The space $\mathrm{H}$ is defined by
$$
\mathrm{H}=\left\{\textbf{u} \in L^{2}\left(\mathbb{T}^{2}\right),\quad \int_{\mathbb{T}^{2}} \textbf{u}(x) \mathrm{d} x=0\right\}.
$$
We also define the space $\mathrm{V}$ by setting
$$
\mathrm{V}=\left\{\textbf{u} \in \mathrm{H}: D_{j} \textbf{u} \in L^{2}\left(\mathbb{T}^{2}, \mathbb{R}^{2}\right)\right\},
$$
where $D_{j}, j=1,2$, are the 1st order weak derivatives of the torus.

The proof is similar to the proof of Mattingly in \cite{M99} and Appendix \ref{Stationary solution}, that there exists $\varepsilon_0>0$ such that for any $\varepsilon\in(0,\varepsilon_0)$, the solution of the following pullback integral equation exists and is unique.
\begin{equation}\label{NS PIE}
	\textbf{u}_{\varepsilon}^{*}(r)=\int_{-\infty}^{r} S_{A}(t-s)B\left( \textbf{u}_{\varepsilon}^{*}(s),\textbf{u}_{\varepsilon}^{*}(s)\right) \mathrm{d} s+\sqrt{\varepsilon}\int_{-\infty}^{r} S_{A}(t-s) \mathrm{d} w_{s}.
\end{equation}
Similar to Example  \ref{burgers equation}, it follows from Lemma \ref{SK-existence} and Lemma \ref{SK-uniqueness}, for any $v\in L^2(\mathbb{R};H_0)$, that there exits a unique solution $\textbf{u}^{*}_{v}$ of skeleton Eq. \eqref{SKMSNS1}, which satisfies the following equation in $H$ for any $t \in \mathbb{R}$,
\begin{equation}\label{SKMSNS1}
	\textbf{u}^{*}_{v}(t,\omega) = \int_{-\infty}^{t}S_{A}(t-r)B\left( \textbf{u}^{*}_v(s),\textbf{u}^{*}_v(s)\right) \mathrm{d}r + \int_{-\infty}^{t}S_{A}(t-r)v(r)\mathrm{d}r.
\end{equation}
We define $\mathcal{G}^0:C(\mathbb{R};H_0)\longrightarrow C(\mathbb{R};H)$ by
\begin{equation*}
	\mathcal{G}^{0}\left(\int_{-\infty}^{\cdot}v(s)\mathrm{d}s\right):
	=\textbf{u}^{*}_{v}(\cdot).
\end{equation*}

By using Theorem \ref{stationary LDP},
the family  $\{\textbf{u}^{*}_\varepsilon:\varepsilon>0\}$ satisfies the LDP in $C(\mathbb{R};H)$ with rate function
\begin{equation*}
	I(f)=\inf_{\left\{v\in L^2(\mathbb{R};H_0):f
		=\mathcal{G}^0(\int_{-\infty}^{\cdot}v(s)\mathrm{d}s)\right\}}
	\left\{\frac{1}{2}\int_{-\infty}^{+\infty}\|v(s)\|_{H_0}^2\mathrm{d}s\right\},
\end{equation*}
where the infimum over an empty set is taken as $+\infty$.

For any $\varepsilon\in(0,\varepsilon_0)$, although we have not proved the solutions of stochastic Navier-Stokes
Eq. \eqref{NS} form a $C^{k}$ perfect cocycle,  the distribution of the solution for Eq. \eqref{NS PIE} is also the unique invariant measure of Eq. \eqref{NS}.
Moreover, it follows from Theorem \ref{stoi} that the family of invariant measures of Eq. \eqref{NS} satisfies LDP with rate function
\begin{equation}\label{NSI'}
	\begin{aligned}
		I'(x)=\inf
		\left\{\frac{1}{2}\int_{-\infty}^{0}\|v(s)\|_{H_0}^2\mathrm{d}s:\textbf{u}
		=\mathcal{G}^0\left(\int_{-\infty}^{\cdot}v(s)\mathrm{d}s\right),\textbf{u}(0)=x\right\}.
	\end{aligned}
\end{equation}

It follows from the definition of $\mathcal{G}^0$ that $\textbf{u}^{*}_{v}(\cdot)$ is a very weak solution (cf. Definition 3.4 in \cite{BC17}) of Eq. \eqref{NS}. And similar to the Definition 3.6 of \cite{BC17}, we define $\mathcal{H}(\textbf{u})$ by
\begin{equation}\label{hu}
	[\mathcal{H}(\textbf{u})](t):=\dot{\textbf{u}}^{*}_{v}(t)+\mathrm{A} \textbf{u}^{*}_{v}(t)+\mathrm{B}(\textbf{u}^{*}_{v}(t), \textbf{u}^{*}_{v}(t)), \quad t \in\mathbb{R},
\end{equation}
where $\dot{\textbf{u}}^{*}_{v}(t)$  be the derivative of $\textbf{u}^{*}_{v}$ with respect to time $t$.\\
And
\begin{equation}\label{NS S}
	S_{-\infty}(\textbf{u}):=\frac{1}{2}\int_{-\infty}^{0}\|\mathcal{H}(\textbf{u})(s)\|_{H_0}^2\mathrm{d}s.
\end{equation}

Combining Eqs. \eqref{NSI'}, \eqref{hu} and \eqref{NS S}, we get that
\begin{equation*}
	I'(x)=\inf
	\left\{\frac{1}{2}\int_{-\infty}^{0}\|\mathcal{H}(\textbf{u})(s)\|_{H_0}^2
	\mathrm{d}s:\textbf{u}(0)=x\right\}=\inf
	\Big\{S_{-\infty}(\textbf{u}):\textbf{u}(0)=x\Big\},
\end{equation*}
which is equal to the rate function defined by quasi-potential in (4.13) of \cite{BC17}.
\end{example}
\section{Proofs of the well-posedness of skeleton equation in infinite intervals and stationary solution}\label{Supplementary proofs of skeleton equation and stationary solution}
This section is divided into two subsections, in the first subsection we prove Theorem \ref{SE IF TH} about the well-posedness of the skeleton Eq. \eqref{skeleton equation}. In the second subsection we prove Theorem \ref{Mstationary solution} about the existence and uniqueness of stationary solution $X^{*}_{\varepsilon}$ for  Eq. \eqref{Msode}.
\subsection{The well-posedness of
	the skeleton equation}\label{The well-posedness of
	the skeleton equation}
Similar to Theorem 4.4 of Sritharan and Sundar \cite{SP06}, we could have the following lemma.
\begin{lemma}
	For any $T>0$, $X(0)\in H$ and $v\in S_M$, for some $M<\infty$, under Hypothesis \ref{Msodehy}, there exists a unique mild solution $X$  of Eq. \eqref{skeleton equation} with initial value $X(0)$ such that
	\begin{equation*}
		X(t) =S_A(t)X(0)+ \int_{0}^{t}S_A(t-r)F(X(r))\mathrm{d}r + \int_{0}^{t}S_{A}(t-r) B (X(r))v(r)\mathrm{d}r, \quad\forall t\in [0,T]
	\end{equation*}
	in space $C([0,T];H)$.
\end{lemma}
Before proving Theorem \ref{SE IF TH}, we illustrate some lemmas below.
\begin{lemma}\label{Piror of skeleton equation}
	(Priori estimate of skeleton equation) For any $t_0\in\mathbb{R}$,  $X(t_0)\in H$, $v\in S_M$, for some $M<\infty$, under the Hypothesis \ref{Msodehy}, there exists a constant $C$, which only depend on $X(t_0)$, $M$, $\lambda$, $C_1$, $D$ such that the unique mild solution $X$  of Eq. \eqref{skeleton equation} with initial value $X(0)$ at time $t_0$ satisfies
	\begin{equation*}
		\|X(t)\|_{H}^2+\lambda\int_{t_0}^{t}\|X\|_{V}^2\mathrm{d}s\leq C,\quad\forall t\geq t_0,
	\end{equation*}
	where $\lambda$ is defined in Hypothesis \ref{Msode} (i), $C_1$ and $D$ are defined in Remark \ref{condition remark} (i) and (ii) respectively.
\end{lemma}
\begin{proof}
	It follows from the Hypothesis \ref{Msodehy} (i), Remark \ref{condition remark} (ii) and Young inequality,
	\begin{align*}
		&\frac{\mathrm{d}\|X\|_{H}^2}{\mathrm{d}t}
		=2\Big\langle X,\frac{\mathrm{d}X}{\mathrm{d}t}\Big\rangle_{H}
		=2_{V^*}\langle AX+F(X),X \rangle_V+2\langle  B (X)v,X\rangle_H\\
		\leq&-2\lambda\|X\|_{V}^2+2| B (X)|_{L_Q}\|v\|_{H_0}\|X\|_{H}
		\leq-2\lambda\|X\|_{V}^2+2D\|v\|_{H_0}\|X\|_{H}\\
		\leq& -\lambda\|X\|_{V}^2+(\delta-\lambda C_1)\|X\|_{H}^2
		+\frac{D^2}{\delta}\|v\|_{H_0}^2,
	\end{align*}
	by taking integral and since $v\in S_M$, it implies that
	\begin{equation*}
		\begin{aligned}
			\|X(t)\|_{H}^2+\lambda\int_{t_0}^{t}\|X\|_{V}^2\mathrm{d}s
			\leq&\|X(t_0)\|_{H}^2  +(\delta-\lambda C_1)\int_{t_0}^{t}\|X\|_{H}^2\mathrm{d}s
			+\frac{D^2}{\delta}\int_{-\infty}^{+\infty}\|v\|_{H_0}^2\mathrm{d}s\\
			\leq& \|X(t_0)\|_{H}^2+ \frac{D^2M}{\delta}+(\delta-\lambda C_1)\int_{t_0}^{t}\|X\|_{H}^2\mathrm{d}s.
		\end{aligned}
	\end{equation*}
	By using the Gronwall inequality, we then obtain
	\begin{equation*}
		\begin{aligned}
			\|X(t)\|_{H}^2+\lambda\int_{t_0}^{t}\|X\|_{V}^2\mathrm{d}s\leq &\Big(\left\|X(t_0)\right\|_{H}^2
			+\frac{D^2M}{\delta}\Big)\exp\{(\delta-\lambda C_1)(t-t_0)\}.
		\end{aligned}
	\end{equation*}
	After choosing $\delta$ small enough such that
	$\delta<\lambda C_1$, then
	\begin{equation*}
		\|X(t)\|_{H}^2+\int_{t_0}^{t}\|X\|_{V}^2\mathrm{d}s\leq \left\|X(t_0)\right\|_{H}^2+\frac{D^2M}{\delta}.
	\end{equation*}
\end{proof}
The following lemmas are aim to study the asymptotic stability of dynamical systems, which is the basis of the definition  $\mathcal{G}^0$.
For any  $M>0$, $t_0\in\mathbb{R}$, $t\geq t_0$, $X_0,\hat{X}_0 \in H$ and $v\in S_M$. Let $X(t,t_0;\omega)X_0$ and $\hat{X}(t,t_0;\omega)\hat{X}_0$ denote the solutions of Eq. \eqref{skeleton equation} starting at different initial value $X_0$ and $\hat{X}_0 $ at time $t_0$,  respectively. For convenience, let $\rho(t,t_0;X_0,\hat{X}_0) = X(t,t_0;\omega)X_0-\hat{X}(t,t_0;\omega)\hat{X}_0$.
\begin{lemma}\label{ltDBE}
	Let $\Gamma(l,t_0;X) = \lambda C_1-\eta -C_0(\frac{1}{l}\int_{t_0}^{t_0+l}\left\|X(s)\right \|_{V}^2\mathrm{d}s)$, then
	\begin{equation*}
		\left \|\rho(t,t_0;X_0,\hat{X}_0)\right\|
		_{H}^2 \leq e^{-2\Gamma(t-t_0,t_0;X)(t-t_0)} \Big(\Big\|X_0 -\hat{X}_0 \Big\|_{H}^2+\frac{2D^2M}{\eta}\Big),
	\end{equation*}
	where $\lambda$ and $C_0$ are defined in Hypothesis \ref{Msode} (i), $C_1$ and $D$ are defined in Remark \ref{condition remark} (i) and (ii) respectively.
\end{lemma}
\begin{proof}
	It follows from the Hypothesis \ref{Msodehy} (i), Remark \ref{condition remark} (i), (ii) and Young inequality,
	\begin{align*}
		\frac{1}{2}\frac{\mathrm{d}}{\mathrm{d}t}\left\|\rho(t)\right\|_{H}^2&=_V\langle \rho, AX-A\hat{X}+F(X)-F(\hat{X})\rangle_{V^*}+\langle B (X)v- B (\hat{X})v,\rho\rangle_H\nonumber\\
		&\leq - \lambda\left\|\rho\right\|_{V}^2 +C_0\left\|\rho\right\|_{H}^2\left\|X\right\|_{V}^2 +2D\|v\|_{H_0}\|\rho\|_{H}\nonumber\\
		&\leq -\lambda C_1\left\|\rho\right\|_{H}^2 + C_0\left\|\rho\right\|_{H}^2\left\|X\right\|_{V}^2 +\eta\|\rho\|_{H}^2+\frac{D^2}{\eta}\|v\|_{H_0}^2 \nonumber \\
		&\leq -\Big[\lambda C_1 - C_0\left\|X\right\|_{V}^2-\eta\Big]\left\|\rho\right\|_{H}^2+\frac{D^2}{\eta}\|v\|_{H_0}^2, \nonumber\\        		
	\end{align*}
	after taking integral, it follows that
	\begin{equation*}
		\left\|\rho(t)\right\|_{H}^2\leq \left\|\rho(0)\right\|_{H}^2+\frac{2D^2}{\eta}\int_{-\infty}^{+\infty}\|v\|_{H_0}^2\mathrm{d}s
		-2\int_{t_0}^{t}\Big[\lambda C_1 - C_0\left\|X\right\|_{V}^2-\eta\Big]\left\|\rho\right\|_{H}^2\mathrm{d}s,
	\end{equation*}
	by the Gronwall inequality, we can get the desired result.
\end{proof}
\begin{corollary}\label{asdbe}
	For any  $\lambda_0\in(0,\lambda C_1)$, there exists a constant $N_0$ such that for any $t-t_0>N_0$,
	\begin{equation*}\label{ec}
		\left\|X(t,t_0;X_0)-\hat{X}(t,t_0;\hat{X}_0)\right\|_{H}^2    =\left \|\rho(t,t_0;X_0,\hat{X}_0)\right\|
		_{H}^2\leq e^{-2\lambda_0(t-t_0)}\Big( \left\|X_0 -\hat{X}_0 \right\|_{H}^2+\frac{2D^2M}{\eta}\Big),
	\end{equation*}
	which imply that
	\begin{equation*}\label{ltc}
		\lim_{t\rightarrow +\infty}\left\|X(t,t_0;X_0)-\hat{X}(t,t_0;\hat{X}_0)\right\|_{H}    =0.
	\end{equation*}
\end{corollary}
\begin{proof}
	Combining Lemma \ref{Piror of skeleton equation} and Lemma \ref{ltDBE}, we could choose $\eta$ small enough and $N_0$ big enough such that $\Gamma(t-t_0,t_0;X)>\lambda_0, \forall t-t_0>N_0$, it implies the conclusion.
\end{proof}
We will give the following two lemmas to prove there exist a unique solution of the backward infinite horizon integral equation for Eq. \eqref{skeleton equation}. And then we define $\mathcal{G}^0$ as the unique solution of the backward infinite horizon  integral equation for Eq. \eqref{skeleton equation}.
\begin{lemma}\label{SK-uniqueness}
	For any $v\in L^2(\mathbb{R};H_0)$, assume that  for any $N\in \mathbb{Z}^{+},Y(t,\omega)|_{t\in[-N,N]}\in C([-N,N];H).$ Moreover, if $Y(t,\omega)$ satisfies the following equation in $H$ for any $t \in \mathbb{R}$,
	\begin{equation}\label{SK-BIHE}
		Y(t,\omega) = \int_{-\infty}^{t}S_A(t-r)F(Y(r,\omega))\mathrm{d}r + \int_{-\infty}^{t}S_{A}(t-r) B (Y(r,\omega))v(r)\mathrm{d}r,
	\end{equation}
	and
	\begin{equation}\label{SK-uniqueness proof condition}
		\sup_{t\in \mathbb{R}}\left\|Y(t)\right\|_{H}^{2} < \infty,
	\end{equation}
	then $Y(t)$ is unique.
\end{lemma}
\begin{proof}
	For any $t' < t$, it follows from Eq. \eqref{SK-BIHE} that
	\begin{equation*}
		Y(t,\omega)
		= S_{A}(t-t')Y(t',\omega) +\int_{t'}^{t}S_{A}(t-r)F(Y(r,\omega))\mathrm{d}r + \int_{t'}^{t}S_{A}(t-r) B (Y(r,\omega))v(r)\mathrm{d}r.
	\end{equation*}
	Therefore, for any $t>t'\in\mathbb{R}$, $Y(\cdot,\omega)$ is the mild solution of Eq. \eqref{skeleton equation} with initial value $Y(t',\omega)$.
	We will show the uniqueness of Eq. \eqref{SK-BIHE}. Assume $Y(\cdot,\omega)$ and $Z(\cdot,\omega)$ are two solutions of Eq. \eqref{SK-BIHE}, then for any $n\in \mathbb{Z}^{+},-n<t$,
	\begin{eqnarray}
		Y(t,\omega) = S_{A}(t+n)Y(-n,\omega) +\int_{-n}^{t}S_{A}(t-r)F(Y(r,\omega))\mathrm{d}r + \int_{-n}^{t}S_{A}(t-r) B (Y(r,\omega))v(r)\mathrm{d}r,\nonumber\\
		Z(t,\omega) = S_{A}(t+n)Z(-n,\omega) +\int_{-n}^{t}S_{A}(t-r)F(Z(r,\omega))\mathrm{d}r + \int_{-n}^{t}S_{A}(t-r) B (Z(r,\omega))v(r)\mathrm{d}r.\nonumber\\
		\nonumber
	\end{eqnarray}
	It implies that $Y(\cdot,\omega)$ and $Z(\cdot,\omega)$ are mild solutions of Eq. \eqref{skeleton equation} with initial value $Y(-n,\omega)$ and $Z(-n,\omega)$ at time $-n$, then by using Corollary \ref{asdbe} and assumption $\sup_{t\in \mathbb{R}}\left\|Y(t)\right\|_{H}^{2} < \infty$ and $\sup_{t\in \mathbb{R}}\left\|Z(t)\right\|_{H}^{2} $ $<\infty$, let $n\rightarrow \infty$, we can have
	\begin{equation*}
		\left\|Y(l,\omega) -Z(l,\omega)\right\|_{H}^2 \rightarrow 0, \quad \forall l>0.
	\end{equation*}
	The uniqueness have been proved.
\end{proof}
Finally,  we will construct the solution of  \eqref{SK-BIHE}. For any $n\in\mathbb{Z}^{+}$ and  fixed $v\in L^2(\mathbb{R};H_0)$, for convenience we set $ X_v(t,s;X_s)$ be the mild solution of  \eqref{skeleton equation} with initial value $X_s$ at initial time $s$, let
\begin{equation*}
	X_v^n(t) =
	\begin{cases}
		X_v(t,-n;0) ,& t > -n, \\
		0,&t\leq -n.\\	
	\end{cases}
\end{equation*}
\begin{lemma} \label{SK-existence}
	Assume that $v\in L^2(\mathbb{R};H_0)$, for any $N\in \mathbb{Z}^{+}$, $X_v^n(\cdot)\rightarrow X_v^{*}(\cdot)$ in $C([-N,N];H)$ as $n\rightarrow \infty$. Moreover, $X_v^{*}$ satisfies the backward infinite horizon integral Eq. \eqref{SK-BIHE} and \eqref{SK-uniqueness proof condition}.
\end{lemma}
\begin{proof}
	It follows from Corollary \ref{asdbe} that $X_v^{n}$ is a Cauchy sequence in $C([-N,N];H)$. Since the space $C([-N,N];H)$ is complete, there exists $X_v^{*}$ such that $\lim\limits_{n\rightarrow\infty}X_v^{n} = X_v^{*}$ in $C([-N,N];H)$. For $N$ is arbitrary, $X_v^{*}(\cdot)$ is defined for all time, and from  Lemma \ref{Piror of skeleton equation}, we have $\sup_{n}\sup_{t\in\mathbb{R}}\left\|X_v^{n}(t)\right\|_{H}^{2} < \infty$, this implies that
	\begin{equation}\label{X** bounded estimate}			\sup_{t\in\mathbb{R}}\left\|X_v^{*}(t)\right\|_{H}^{2} < \infty.
	\end{equation}
	Finally  we will show $X_v^{*}$ satisfies Eq. \eqref{SK-BIHE}.
	
	\textbf{Step 1}: For any $t\in\mathbb{R}$ and $t_0<t$, we will show that $X_v^{*}$ satisfies
	\begin{eqnarray}\label{constructed mild solution}
		X_v^{*}(t) = S_{A}(t-t_0)X_v^{*}(t_0) +\int_{t_0}^{t}S_{A}(t-r)F(X_v^{*}(r))^2\mathrm{d}r + \int_{t_0}^{t}S_{A}(t-r) B (X_v^{*}(r))v(r)\mathrm{d}r.\nonumber\\
	\end{eqnarray}
For any $t_0<t\in \mathbb{R}$, we can find $N\in\mathbb{N}^+$ such that $t_0,t\in[-N,N]$. Fixed $N$, it follows from Hypothesis \ref{Msodehy} (ii), we obtain that
\begin{eqnarray*}
	&&\left\| \int_{t_0}^{t}S_{A}(t-r)F(X_v^{*}(r))\mathrm{d}r - \int_{t_0}^{t}S_{A}(t-r)F(X_v^{n}(r))\mathrm{d}r \right\|^2_{H} \nonumber\\
	&\leq& C(N)\int_{t_0}^{t}(t-r)^{\alpha}\left \|X_v^{*}(r)-X_v^{n}(r) \right \|_H^2\mathrm{d}r \nonumber \\
	&\leq&C(N)\sup_{r\in[-N,N]}\left \|X_v^{*}(r)-X_v^{n}(r) \right \|_H^2\int_{t_0}^{t}(t-r)^{\alpha}\mathrm{d}r,
\end{eqnarray*}
Since $X_v^{n}(\cdot)\rightarrow X_v^{*}(\cdot)$ in $C([-N,N];H)$, it deduces that
\begin{equation}
	\left\| \int_{t_0}^{t}S_{A}(t-r)F(X_v^{*}(r))\mathrm{d}r - \int_{t_0}^{t}S_{A}(t-r)F(X_v^{n}(r))\mathrm{d}r \right \|^2_H\rightarrow 0. \nonumber
\end{equation}
It follows from  Remark \ref{condition remark} (ii) and (iii),
\begin{equation*}
	\begin{aligned}
		&\left\| \int_{t_0}^{t}S_{A}(t-r) B (X_v^{*}(r))v(r)\mathrm{d}r - \int_{t_0}^{t}S_{A}(t-r) B (X_v^{n}(r))v(r)\mathrm{d}r \right\|_{H}^2\\
		\leq& C \int_{t_0}^{t}\left\| B \left(X_v^{*}(r)\right)v(r)- B (X_v^{n}(r))v(r)\right\|_{H}^2\mathrm{d}s\\
		\leq&C\beta^2\int_{t_0}^{t}\|X_v^{*}(r)-X_v^{n}(r)\|_{H}^2\|v\|_{H_0}^2\mathrm{d}s\\
		\leq&C\beta^2\sup_{t\in[-N,N]}\|X_v^{*}(r)-X_v^{n}(r)\|_{H}^2\int_{t_0}^{t}\|v\|_{H_0}^2\mathrm{d}s.\\
	\end{aligned}
\end{equation*}
Since $X_v^{n}(\cdot)\rightarrow X_v^{*}(\cdot)$ in $C([-N,N];H)$, for any $v\in L^2(\mathbb{R};H_0)$, it implies that
\begin{equation*}
	\lim_{n\rightarrow\infty} \int_{t_{0}}^{t} S_{A}(t-r) B \left(X_v^{n}\right)v(r)\mathrm{d}r= \int_{t_{0}}^{t}S_{A}(t-r)  B \left( X_v^{*}(r)\right) v(r)\mathrm{d}r,\quad\text{in} ~ C([-N,N];H).
\end{equation*}
At the same time it follows form Remark \ref{condition remark} (iii) that $X_v^{n}(t)$, $S_{A}(t-t_0)X_v^{n}(t_0)$ converge strongly to $X_v^{*}$ and $S_{A}(t-t_0)X_v^{*}(t_0)$ in $H$
respectively, hence Eq. (\ref{constructed mild solution}) holds.

\textbf{Step 2}: We next prove that $X_v^{*}$ satisfies Eq. \eqref{SK-BIHE}. From Eq. (\ref{constructed mild solution}), it is easy to know that for any $0<m<n$,
\begin{equation*}		\int_{-n}^{-m}S_{A}(-r)[F(X_v^{*}(r))+ B (X_v^{*}(r))v(r)]\mathrm{d}r = -S_{A}(n)X_v^{*}(-n)+S_{A}(m)X_v^{*}(-m),
\end{equation*}
moreover,
$$\|S_{A}(n)X_v^{*}(-n)\|_{H} \leq |S_{A}(n)|_{L}\left\|X_v^{*}(-n)\right\|_H,\quad\|S_{A}(m)X_v^{*}(-m)\|_{H} \leq |S_{A}(m)|_{L}\left\|X_v^{*}(-m)\right\|_H^2.$$
Thus combining
(\ref{X** bounded estimate}) and  Remark \ref{condition remark} (iii), we obtain that  $\|S_{A}(n)X_v^{*}(-n)\|_{H}$ and $\|S_{A}(m)X_v^{*}(-m)\|_{H}$ converge to 0 as $m,n \rightarrow\infty$. Therefore,
\begin{equation*}
	\int_{-n}^{t}S_{A}(-r)[F(X_v^{*}(r))+ B (X_v^{*}(r))v(r)]\mathrm{d}r
\end{equation*}
is a Cauchy sequence in $H$ with respect to $n$ for any $t\in\mathbb{R}$. Let $n\rightarrow \infty$, we can obtain that
\begin{eqnarray*} \int_{-n}^{t}S_{A}(-r)[F(X_v^{*}(r))+ B (X_v^{*}(r))v(r)]\mathrm{d}r\rightarrow \int_{-\infty}^{t}S_{A}(-r)[F(X_v^{*}(r))+ B (X_v^{*}(r))v(r)]\mathrm{d}r. \nonumber\\
\end{eqnarray*}
Moreover, by using Remark \ref{condition remark} (iii),  $\left\|S_{A}(t+n)X_v^{*}(-n) \right\|_H^2\leq |S_{A}(t+n)|_{L}\left\|X_v^{*}(-n)\right\|_H^2\rightarrow 0,~\text{as}~ n\rightarrow\infty.$ Thus it follows from Eq. (\ref{constructed mild solution}) that $X_v^{*}$ satisfies Eq. \eqref{SK-BIHE}.
\end{proof}
The proof of Theorem \ref{SE IF TH} in the following.
\begin{proof}
Combining Lemma \ref{SK-existence} and Lemma \ref{SK-uniqueness}, we have known that the solution of the backward infinite horizon integral Eq. \eqref{SK-BIHE0} for skeleton Eq. \eqref{skeleton equation} exists and is unique, and satisfies \eqref{SK-uniqueness proof condition0}.
\end{proof}

\subsection{Stationary solution}\label{Stationary solution}
Similar to the proofs of Mattingly \cite{M99} and Liu and Zhao \cite{LZ09}, we prove Theorem \ref{Mstationary solution} in this subsection. This subsection under Hypothesis \ref{Msodehy},
and we first give the energy estimate for  Eq. \eqref{Msode} in the following lemma.
\begin{lemma}\label{energy estimate} \textbf{(Energy Estimate)}
For convenience, we fixed $t_0$ and denote $X_{\varepsilon}(t) = X_{\varepsilon}(t,t_0;X_0)$ be the solution of Eq. \eqref{Msode} with initial value $X_0$ at time $t_0$, then we have
\begin{equation}\label{energy estimate'}
	\begin{aligned}
		\mathbb{E}\left\| X_{\varepsilon}(t)\right \|_{H}^{2p} \leq& \mathbb{E}\left\| X_0\right \|_{H}^{2p} -2 \lambda p \mathbb{E}\int_{t_0}^{t}\left\| X_{\varepsilon}(s)\right \|_{H}^{2(p-1)}\left\| X_{\varepsilon}(s)\right \|_{V}^{2}\mathrm{d}s \\ \nonumber
		&+ \mathbb{E}\int_{t_0}^{t}2\varepsilon p(p-1)\left\| X_{\varepsilon}(s)\right \|_{H}^{2(p-2)}\left ( \|X_{\varepsilon}(s)\|_H^2| B (X_{\varepsilon}(s))|_{L_Q}^2\right ) \mathrm{d}s \\ \nonumber &+\mathbb{E} \int_{t_0}^{t}\varepsilon p\left\| X_{\varepsilon}(s)\right \|_{H}^{2(p-1)}| B (X_{\varepsilon}(s))|_{L_Q}^2 \mathrm{d}s,\quad\forall \varepsilon>0.
	\end{aligned}
\end{equation}
\end{lemma}
\begin{proof}
For every $\varepsilon>0$, it follows from the It\^{o} formula that
\begin{equation*}
	\mathrm{d}\|X_{\varepsilon}(t)\|_H^2=2  {_{V^{*}}}\langle AX_{\varepsilon},X_{\varepsilon}\rangle_V\mathrm{d}t+2\langle F(X_{\varepsilon}),X_{\varepsilon}\rangle_H\mathrm{d}t+2\sqrt{\varepsilon}\langle X_{\varepsilon}, B (X_{\varepsilon})\mathrm{d}W\rangle_H+\varepsilon| B (X_{\varepsilon})|_{L_Q}^2 \mathrm{d}t.
\end{equation*}
For $p\geq 1$, it follows from the It\^{o} formula and Hypothesis \ref{Msodehy} (i) that
\begin{eqnarray}
	\mathrm{d}\left\| X_{\varepsilon}(t)\right \|_{H}^{2p} &=& p\left\| X_{\varepsilon}(t)\right \|_{H}^{2(p-1)}\mathrm{d}\left\| X_{\varepsilon}(t)\right \|_{H}^{2} + \frac{1}{2}p(p-1)\left\| X_{\varepsilon}(t)\right \|_{H}^{2(p-2)}{\left \langle \mathrm{d}\left\| X_{\varepsilon}(t)\right \|_{H}^{2} \right \rangle}_{t} \nonumber \\
	&\leq& p\left\| X_{\varepsilon}(t)\right \|_{H}^{2(p-1)}\left[-2\lambda\|X_{\varepsilon}\|_{V}+2\sqrt{\varepsilon}\langle X_{\varepsilon}, B (X_{\varepsilon})\mathrm{d}W\rangle_H\right] \nonumber \\
	&& +2\varepsilon p(p-1)\left\| X_{\varepsilon}(t)\right \|_{H}^{2(p-2)}\left ( \|X_{\varepsilon}\|_H^2| B (X_{\varepsilon})|_{L_Q}^2\right ) \mathrm{d}t \nonumber\\
	&&+\varepsilon p\left\| X_{\varepsilon}(t)\right \|_{H}^{2(p-1)}| B (X_{\varepsilon})|_{L_Q}^2 \mathrm{d}t.\nonumber
\end{eqnarray}

For convenience set  $M_{\varepsilon}(t):=M_{\varepsilon}(t,t_0;X_0)=2\sqrt{\varepsilon}p\int_{t_0}^{t}\left\| X_{\varepsilon}\right \|_{H}^{2(p-1)}\langle X_{\varepsilon}, B (X_{\varepsilon})\mathrm{d}W\rangle_H$, then for $M_{\varepsilon}(t)$ is a local martingale there exists a sequence of stopping time $\left\{T_n\right\}$, with $\left\{T_n\right\} \rightarrow +\infty$ as $n\rightarrow +\infty$, such that $M_{\varepsilon}({t\wedge T_n})$ is a  martingale, so the Optional Stopping Time Theorem implies $\mathbb{E}M_{\varepsilon}({t\wedge T_n}) = 0$, we denote $f_n(t)$ by
\begin{eqnarray}
	f_n(t)&=&\left\| X_0\right \|_{H}^{2p} + \int_{t_0}^{t}2\varepsilon p(p-1)\left\| X_{\varepsilon}\right \|_{H}^{2(p-2)}\left ( \|X_{\varepsilon}\|_H^2| B (X_{\varepsilon})|_{L_Q}^2\right ) \mathrm{d}s \nonumber\\
	&&+ \int_{t_0}^{t}p\left\| X_{\varepsilon}(t)\right \|_{H}^{2(p-1)}| B (X_{\varepsilon})|_{L_Q}^2 \mathrm{d}s\nonumber + M_{\varepsilon}({t\wedge T_n}).\nonumber
\end{eqnarray}
When $t\leq T_n$, $f_n(t) \geq \left\| X_{\varepsilon}(t)\right \|_{H}^{2p} +\lambda \int_{t_0}^{t}2p\left\| X_{\varepsilon}(s)\right \|_{H}^{2(p-1)}\left\| X_{\varepsilon}(s)\right \|_{V}^{2}\mathrm{d}s>0$. And when $t> T_n$, $f_n(t) - f(T_n) = \int_{T_n}^{t}2\varepsilon p(p-1)\left\| X_{\varepsilon}\right \|_{H}^{2(p-2)}\left ( \|X_{\varepsilon}\|_H^2| B (X_{\varepsilon})|_{L_Q}^2\right ) \mathrm{d}s + \int_{T_n}^{t}p\left\| X_{\varepsilon}\right \|_{H}^{2(p-1)}| B (X_{\varepsilon})|_{L_Q}^2 \mathrm{d}s>0$. So, $f_n(t)$ is non-negative for all $t$. Thus we can use Fatou's Lemma to get
\begin{eqnarray}
	\mathbb{E}\left\| X_{\varepsilon}(t)\right \|_{H}^{2p} + \mathbb{E}\int_{t_0}^{t}2\lambda p\left\| X_{\varepsilon}(s)\right \|_{H}^{2(p-1)}\left\| X_{\varepsilon}(s)\right \|_{V}^{2}\mathrm{d}s
	= \mathbb{E}\lim_{n\rightarrow +\infty}f_n \leq \lim_{n\rightarrow +\infty}\mathbb{E}f_n, \nonumber
\end{eqnarray}
which imply the conclusion.
\end{proof}
Set $\mathbb{E}_{p,{\varepsilon}}(t):= \mathbb{E}\left\| X_{\varepsilon}(t)\right \|_{H}^{2p}$,  we prove the following lemma by using Lemma \ref{energy estimate}.
\begin{lemma}\label{bounded estimate}
Assume the initial condition satisfies $\mathbb{E}\left\|X(t_0)\right\|_{H}^{2p}<\infty$ for some fixed $p\geq 1$, then for any $\varepsilon_0>0$, there exists a constant $C$ such that
\begin{equation*}
	\mathbb{E}\left\|X_{\varepsilon}(t,t_0;X(t_0))\right\|_{H}^{2p} \leq C,\quad\forall t\geq t_0,\  \varepsilon\in(0,\varepsilon_0),
\end{equation*}
where C depend on $\mathbb{E}\left\|X(t_0)\right\|_{H}^{2j},1\leq j \leq p,~j\in \mathbb{Z}$ and $p,\varepsilon_0,\lambda,C_1,D$, where $\lambda$ is defined in Hypothesis \ref{Msode} (i), $C_1$ and $D$ are defined in Remark \ref{condition remark} (i) and (ii) respectively.
\end{lemma}        		
\begin{proof}
For any $t\geq t_0$, it follows from the Energy  Estimate Lemma \ref{energy estimate} and Remark \ref{condition remark} (i),(ii) that
\begin{equation*}
	\mathrm{d}\mathbb{E}_{p,{\varepsilon}}(t) \leq -2 p\lambda C_1  \mathbb{E}_{p,{\varepsilon}}(t)  +\left(2\varepsilon p(p-1)D^2+\varepsilon  p D^2\right)\mathbb{E}_{p-1,\varepsilon}(t),
\end{equation*}
and
\begin{align*}
	\mathrm{d}\left( e^{2 p\lambda C_1 t}\mathbb{E}_{p,{\varepsilon}}(t)\right) = &2 p\lambda C_1 e^{2 p\lambda C_1t}\mathbb{E}_{p,{\varepsilon}}(t)\mathrm{d}t +  e^{2 p\lambda C_1 t}\mathrm{d}\mathbb{E}_{p,{\varepsilon}}(t).
\end{align*}

Thus
\begin{equation*}
	\mathbb{E}_{p,{\varepsilon}}(t) \leq e^{-2 p\lambda C_1  (t-t_0)}\mathbb{E}_{p,{\varepsilon}}(t_0) 
	+\left(2\varepsilon p(p-1)D^2+\varepsilon  p D^2\right)\int_{t_0}^{t}e^{-2 p\lambda C_1 (t-s)}\mathbb{E}_{p-1,\varepsilon}(s)\mathrm{d}s,
\end{equation*}
 then  $\mathbb{E}_{1,\varepsilon}(t) \leq \mathbb{E}_{1,\varepsilon}(t_0)+\varepsilon   D^2 \int_{t_0}^{t}e^{-2 \lambda  C_1  (t-s)}\mathrm{d}s\leq\mathbb{E}_1(t_0)+\frac{\varepsilon   D^2}{2 \lambda  C_1 }$, set $\mathbb{E}_1^{\max}(t_0) = \mathbb{E}_1(t_0)+ \frac{\varepsilon_0 D^2 }{2 \lambda C_1 }$, then by induction we get that
$$\mathbb{E}_{p,{\varepsilon}}(t) \leq \mathbb{E}_{p,{\varepsilon}}(t_0) + C_p^{'}\mathbb{E}_{p-1}^{\max}(t_0) = \mathbb{E}_{p,{\varepsilon}}^{\max}(t_0),$$
where $C_p^{'}$ is constant only depends on $p,\varepsilon_0,\lambda,C_1,D$, which imply that $\mathbb{E}_{p,{\varepsilon}}^{\max}$ is just the linear combination of the moments of  order less that or equal to $p$ of $\mathbb{E}_{p,{\varepsilon}}^{\max}(t_0)$.
\end{proof}
Let $\{X_n\}$ be a sequence of real random variables. Let $g:\mathbb{R}^+\rightarrow\mathbb{R}^+$. Define the random variable  $T_{bound}\left(\{X_n\},g\right)$ to be the smallest positive integer such that for almost every $\omega$,
$$m>T_{bound}\left(\{X_n\},g\right)(\omega) \Rightarrow |X_m(\omega)| \leq g(m),~\quad\forall \omega\in\Omega.$$
Mattingly has proved the following Bounding Lemma in \cite{M99}.
\begin{lemma}\label{Bounding Lemma}\textbf{(Bounding Lemma)}
Assume that
$$\mathbb{P}(|X_n|\geq \varepsilon n^{\delta}) \leq \frac{\mathbb{E}|X_n|^p}{n^{p\delta}\varepsilon^p} \leq \frac{C}{n^{p\delta-r}\varepsilon^p},$$
for some $\varepsilon,\delta,p,C>0$ and $r \geq 0$, then \\
(i). if $p\delta > 1+r$ then $T_{bound}\left(\{X_n\},\varepsilon n^{\delta}\right) < +\infty$ a.s.$\omega$.\\
(ii). $\mathbb{E}\left[T_{bound}\left(\{X_n\},\varepsilon n^{\delta}\right)\right]^q$ is finite for  $q\in(0,p\delta -(1+r))$.
\end{lemma}
Let $X_{\varepsilon}(t,t_0;\omega,X_0)$ and $\tilde{X}_{\varepsilon}(t,t_0;\omega,\tilde{X}_0)$ be the solutions of Eq. \eqref{Msode} starting from different initial value $X_0,\tilde{X}_0$ respectively. Define $\rho_{\varepsilon}(t,t_0;X_0,\tilde{X}_0) = X_{\varepsilon}(t,t_0;\omega,X_0)-\tilde{X}_{\varepsilon}(t,t_0;\omega,\tilde{X}_0)$, we will consider the asymptotically stable property of Eq. \eqref{Msode} in the following lemmas.
\begin{lemma}\label{the lower bound}
Let $\Gamma(l,t_0;X_{\varepsilon}) =2\lambda C_1-\varepsilon\beta^2 -2C_0\Big(\frac{1}{l}\int_{t_0}^{t_0+l}\left\|X_{\varepsilon}(s)\right \|_{V}^2\mathrm{d}s\Big)$,  for any $\varepsilon_0>0$, there exists $\delta>1$ such that for any  $\varepsilon\in(0,\varepsilon_0)$ there exists a almost finite random variable $\tau_\varepsilon$ satisfies
\begin{equation*}
	\left \|\rho_{\varepsilon}(t_0+n,t_0;X_0,\tilde{X}_0)\right\|_{H}^2 \leq e^{-n\Gamma(n,t_0;X_{\varepsilon})} \Big(\left\|X_0 -\tilde{X}_0 \right\|_{H}^2+2\sqrt{\varepsilon}n^\delta\Big),
\end{equation*}
for every  $n>\tau_\varepsilon$.
\end{lemma}
\begin{proof}
It follows from  the It\^{o} formula and Hypothesis \ref{Msodehy} (i) and Remark \ref{condition remark} (i),(ii) that
\begin{equation}\label{q1}
	\begin{aligned}		\mathrm{d}\left\|\rho_{\varepsilon}(t)\right\|_{H}^2 = &2 {_{V^*}}\langle AX_{\varepsilon}-A\tilde{X}_{\varepsilon}+ F(X_{\varepsilon})-F(\tilde{X}_{\varepsilon}),\rho_{\varepsilon}\rangle_{V} \\&+2\sqrt{\varepsilon}\langle \rho_{\varepsilon},(B (X_{\varepsilon})- B (\tilde{X}_{\varepsilon}))\mathrm{d}W_t\rangle_H
		+\varepsilon| B (X_{\varepsilon})- B (\tilde{X}_{\varepsilon})|_{L_Q}^2\\
		\leq&-2\lambda\|\rho_{\varepsilon}\|_{V}^2+2C_0\left\|\rho_{\varepsilon}\right\|_{H}^2
		\left\|X_{\varepsilon}\right\|_{V}^2+\varepsilon \beta^2\|\rho_{\varepsilon}\|_{H}^2 +2\sqrt{\varepsilon}\langle \rho_{\varepsilon},( B (X_{\varepsilon})- B (\tilde{X}_{\varepsilon}))\mathrm{d}W_t\rangle_H\\
		\leq &-\Big(2\lambda C_1 -2C_0\|X_{\varepsilon}\|_{V}^2-\varepsilon\beta^2\Big)\|\rho_{\varepsilon}\|_{H}^2+2\sqrt{\varepsilon}\langle \rho_{\varepsilon},( B (X_{\varepsilon})- B (\tilde{X}_{\varepsilon}))\mathrm{d}W_t\rangle_H.
	\end{aligned}
\end{equation}

Set $M_{\varepsilon}(l,t_0;\rho):=\int_{t_0}^{t_0+l}\Big\langle \rho_{\varepsilon},\big[ B (X_{\varepsilon})- B (\tilde{X}_{\varepsilon})\big]\mathrm{d}W_t\Big\rangle_H $,
it follows from the Burkholder-Davis-Gundy inequality, Remark \ref{condition remark} (ii) and Lemma \ref{bounded estimate}, for any $\varepsilon_0>0$ there exists a constant $C$ such that
\begin{eqnarray*}
	\mathbb{E}\left(\sup_{t_0\leq s \leq l}M_{\varepsilon}(s,t_0;\rho)\right)^{2} \leq C \mathbb{E}\left(\left \langle M (\cdot,t_0;\rho)\right \rangle_l\right)
	\leq  CD^2 \int_{t_0}^{t_0 + l}\mathbb{E}\left\|\rho_{\varepsilon}(s)\right\|_{H}^{2}\mathrm{d}s
	\leq CD^2 l\mathbb{E}_{1,\varepsilon}^{\max},
\end{eqnarray*}
for every $\varepsilon\in(0,\varepsilon_0)$.

Let $M^n_{\varepsilon} = \sup_{t_0\leq s\leq n}M_{\varepsilon}(s,t_0;\rho)$, by the Chebyshev's inequality, we have
\begin{equation*}
	\mathbb{P}(|M^n_{\varepsilon}|\geq  n^{\delta}) \leq\frac{\mathbb{E}|M^n_{\varepsilon}|^{2}}{n^{2\delta}}\leq \frac{CD^2 \mathbb{E}_{1,\varepsilon}^{\max}}{n^{2\delta-1}},
\end{equation*}
for every $\varepsilon\in (0,\varepsilon_0)$ and $\delta>1$,	it follows from  Bounding Lemma \ref{Bounding Lemma},  $\tau_\varepsilon:=T_{bound}(\left\{M^n_{\varepsilon}\right\}, n^\delta)$ is almost finite.
Combining (\ref{q1}), when $l>\tau_\varepsilon$
\begin{equation*}
	\|\rho_{\varepsilon}(t_0+l)\|_{H}^2\leq \|\rho_{\varepsilon}(t_0)\|_{H}^2+2\sqrt{\varepsilon}l^\delta
	-\int_{t_0}^{t_0+l}\Big(\frac{2\lambda}{ C_1} -2C_0\|X_{\varepsilon}\|_{V}^2-\varepsilon\beta^2\Big)\left\|\rho_{\varepsilon}\right\|_{H}^2\mathrm{d}t.
\end{equation*}
By using the Gronwall inequality, we get the desired result.
\end{proof}
It follows from the It\^{o} formula and Hypothesis \ref{Msodehy} that
\begin{eqnarray}\label{uH1}
\frac{2 \lambda}{l}\int_{t_0}^{t_0+l}\left\|X_{\varepsilon}(s,t_0;X_0)\right \|_{V}^2\mathrm{d}s
&\leq& \frac{\left\|X_0\right\|_{H}^2 -\left\|X_{\varepsilon}(t_0+l,t_0;X_0)\right\|_{H}^2 }{l} +\frac{\varepsilon}{l}\int_{t_0}^{t_0+l}| B (X_{\varepsilon})|_{L_Q}^2\mathrm{d}s \nonumber\\	&&+\frac{2\sqrt{\varepsilon}}{l}\int_{t_0}^{t_0+l}\langle X_{\varepsilon}, B (X_{\varepsilon})\mathrm{d}W\rangle_H \nonumber \\
&\leq& \varepsilon D^2 + \frac{1}{l}\Big(\left\|X_0\right \|_{H}^2+2\sqrt{\varepsilon}\int_{t_0}^{t_0+l}\langle X_{\varepsilon}, B (X_{\varepsilon})\mathrm{d}W\rangle_H\Big). \nonumber \\
\end{eqnarray}
Set $\tilde{M}_{\varepsilon}(l,t_0;X_{\varepsilon}) = \int_{t_0}^{t_0+l}\langle X_{\varepsilon}, B (X_{\varepsilon})\mathrm{d}W\rangle_H$, for control $\Gamma(l,t_0;X_{\varepsilon})$ we need to control $\tilde{M}_{\varepsilon}(l,t_0;X_{\varepsilon})$.
By using the Burkholder-Davis-Gundy inequality and Lemma \ref{bounded estimate}, for any $\varepsilon_0>0$ there exists a constant $C$ such that
\begin{equation*}
\begin{aligned}
	\mathbb{E}\left(\sup_{t_0\leq s \leq l}\tilde{M}_{\varepsilon}(s,t_0;X_{\varepsilon})\right)^{2p} &\leq \mathbb{E}\left(\left \langle \tilde{M}_{\varepsilon}(l,t_0;X_{\varepsilon}) \right \rangle_l^{p}\right)\\
	& \leq  CD^{2p}l^{p-1}\int_{t_0}^{t_0 + l}\mathbb{E}\left\|X_{\varepsilon}(s)\right\|_{L^2}^{2p}\mathrm{d}s\\
	&\leq CD^{2p}l^{p}\mathbb{E}_{p,{\varepsilon}}^{\max}, \quad \forall\varepsilon\in(0,\varepsilon_0),
\end{aligned}
\end{equation*}
the second inequality comes from H\"older's inequality.

Let $\tilde{M}^n_{\varepsilon} = \sup_{t_0\leq s\leq n}\tilde{M}_{\varepsilon}(s,t_0;X_{\varepsilon})$. By the Chebyshev's inequality, we have
\begin{equation}\label{ T bound M inequality}
\mathbb{P}(|\tilde{M}^n_{\varepsilon}|\geq \eta n^{\delta}) \leq\frac{\mathbb{E}|\tilde{M}^n_{\varepsilon}|^{2p}}{n^{2p\delta}\eta^{2p}}\leq \frac{CD^{2p}\mathbb{E}_{p,{\varepsilon}}^{\max}}{n^{2p\delta-p}\eta^{2p}}.
\end{equation}
By using Bounding Lemma \ref{Bounding Lemma}, we can get the following Lemma.
\begin{lemma}\label{stopping time}
Let $X_0$ be a random variable, measurable with respect to $\mathcal{F}_{t_0}$, such that $\mathbb{E}\left\|X_0\right\|_{H}^{2p}$ is finite, then\\
(i). For any fixed $\varepsilon_0>0$, if $p>1$, then for any $\varepsilon\in(0,\varepsilon_0)$, $\eta_\varepsilon,\eta'_{\varepsilon}>0$,  $T_{bound}\left(\left\{\tilde{M}^n_{\varepsilon}\right\},\eta_\varepsilon n\right)$ and $T_{bound}\left(\left\{\left\|X_0\right\|_{H}^2\right\},\eta'_{\varepsilon} n\right)$ is finite almost surely.\\
(ii). $\mathbb{E}T_{bound}\left(\left\{\tilde{M}^n_{\varepsilon}\right\},
\eta_{\varepsilon} n\right)^q$ and $\mathbb{E}T_{bound}\left(\left\{\left\|X_0\right\|_{H}^2\right\},\eta'_{\varepsilon} n\right)^q$ is finite for $q\in(0,p-1)$.
\end{lemma}
\begin{proof}
Noticing that (\ref{ T bound M inequality}) and
$$\mathbb{P}(\left\|X_0\right\|_{H}^2\geq \eta'_{\varepsilon} n) \leq \frac{\mathbb{E}\left\|X_0\right\|_{H}^{2p}}{n^p(\eta'_{\varepsilon})^p}\leq \frac{C}{n^p (\eta'_{\varepsilon})^p},$$
Bounding Lemma \ref{Bounding Lemma} imply the conclusion.
\end{proof}
Fixed $\varepsilon_0>0$, for any $\varepsilon\in(0,\varepsilon_0)$, combining the definition of $\Gamma(l,t_0;X_{\varepsilon})$, (\ref{uH1}) and Lemma \ref{stopping time}, when  $$l>\max\left\{T_{bound}\left(\left\{\tilde{M}^n_{\varepsilon}\right\},\sqrt{\varepsilon} n\right), \quad T_{bound}\left(\left\{\left\|X_0\right\|_{H}^2\right\},
{\varepsilon} n\right)\right\},$$ we have
\begin{align}\label{q2}
\Gamma(l,t_0;X_{\varepsilon}) =&2\lambda C_1-\varepsilon\beta^2 -2C_0\Big(\frac{1}{l}\int_{t_0}^{t_0+l}\left\|X_{\varepsilon}(s)\right \|_{V}^2\mathrm{d}s\Big)\nonumber\\
\geq&2\lambda C_1-\varepsilon\beta^2  -\frac{C_0}{\lambda }\left\{\varepsilon D^2 + \frac{1}{l}\left(\left\|X_0\right \|_{H}^2+2\sqrt{\varepsilon}\int_{t_0}^{t_0+l}\langle X_{\varepsilon}, B (X_{\varepsilon})\mathrm{d}W\rangle_H\right)\right\}\nonumber\\
\geq&2\lambda C_1-\varepsilon\beta^2 -\frac{C_0}{\lambda }(\varepsilon D^2+3\varepsilon).
\end{align}

\begin{lemma}\label{long-time}
For any $\lambda_0\in(0,2\lambda C_1)$ and $t_0\in\mathbb{R}$. Let $X_0,\tilde{X}_0 \in H$ be initial condition, measurable with respect to $\mathcal{F}_{-\infty}^{t_0}$, such that $\mathbb{E}\left \|X_0\right \|_{H}^{2p} < \infty$, $\mathbb{E} \|\tilde{X}_0 \|_{H}^{2p} < \infty$ for some $p>1$. Let $X_{\varepsilon}(t,t_0;X_0)$ and $\tilde{X}_{\varepsilon}(t,t_0;\tilde{X}_0)$ be the solutions of Eq. \eqref{Msode} starting from different initial value $X_0,\tilde{X}_0$ respectively.  Then there exists $\varepsilon_0,\delta$ and a sequence almost finite random variable $\{l_{\varepsilon}\}_{\varepsilon\in(0,\varepsilon_0)}$ such that
\begin{equation*}
	\left \|X_{\varepsilon}(t,t_0;X_0) - \tilde{X}_{\varepsilon}(t,t_0;\tilde{X}_0) \right \|_{H}^2 \leq \left(\left \|X_0 - \tilde{X}_0 \right \|_{H}^2+\sqrt{\varepsilon}(t-t_0)^\delta\right)e^{-\lambda_0(t-t_0)},
\end{equation*}
for every $\varepsilon\in(0,\varepsilon_0)$, $t>t_0+l_{\varepsilon}$.
\end{lemma}
\begin{proof}
We can chose $\varepsilon_0$ small enough such that
\begin{equation}\label{q3}
	2\lambda C_1-\varepsilon\beta^2 -\frac{C_0}{\lambda }(\varepsilon D^2+3\varepsilon)>\lambda_0,
\end{equation}
for every $\varepsilon<\varepsilon_0$. We define  $l_{\varepsilon}$ by $$l_{\varepsilon}=\max\left\{\tau_{\varepsilon},\quad T_{bound}\left(\left\{\tilde{M}^n_{\varepsilon}\right\},\sqrt{\varepsilon} n\right),\quad T_{bound}\left(\left\{\left\|X_0\right\|_{H}^2\right\},
{\varepsilon} n\right)\right\},$$
where $\tau_{\varepsilon}$ is defined in Lemma \ref{the lower bound}, the conclusion follows from
Lemma \ref{the lower bound}, (\ref{q2}) and (\ref{q3}).
\end{proof}
\begin{lemma}\label{uniqueness proof different initial value estimate}
(cf. Theorem 2 in \cite{M99}) Fix  $\lambda_0 \in (0,2\lambda C_1)$ and  $t\in \mathbb{R}$. Let $\left \{X_0(n)\right \} $ be a sequence of random variable with $n \in \mathbb{Z}^{+}$. Assume that $X_0(n)$ is measurable with respect to $\mathcal{F}^{t-n}_{-\infty}$ and that $\mathbb{E}\left\|X_0(n)\right\|_{H}^{2p}$ is uniformly bounded in $n$ for some $p > 2$. Then there exists $\varepsilon_0,\delta$ such that for any $\varepsilon\in(0,\varepsilon_0)$ the following hold:\\
(i). With probability one, there exists a random  time $\vec{n}_{\varepsilon}>0$ such that for every $s > 0$ and all $n\in\mathbb{Z}$ with $n > \vec{n}_{\varepsilon}$, we have
$$\sup_{\tilde{X}_0 \in A_n} \left\|X_{\varepsilon}(t+s,t-n;X_0(n)) - X_{\varepsilon}(t+s,t-n;\tilde{X}_0)\right\|_{H}^2 \leq \left(n+\sqrt{\varepsilon}(s+n)^\delta\right)e^{-\lambda_0(s+n)}.$$
Here $A_n$ is the set $\left\{\tilde{X}_0:\left\|\tilde{X}_0\right\|_{H}^2\leq \frac{n}{4}\right\}. $ In addition $\mathbb{E}(\vec{n}_{\varepsilon}^q)<\infty$ for any $q\in(0,p-2)$.\\
(ii). Let $\left \{\bar{X}_0(n)\right \} $ be a second sequence of random variable with $n \in \mathbb{Z}^{+}$ measurable with respect to $\mathcal{F}^{t-n}_{-\infty}$ and that $\mathbb{E}\left\|\bar{X}_0(n)\right\|_{H}^{2p}$ is uniformly bounded in $n$ for some $p > 2$. Then with probability one, there exists a random time $\vec{n}'_{\varepsilon}(\varepsilon,\delta,t,\omega)>0$ such that for every $s > 0$ and all $n\in\mathbb{Z}$ with $n > \vec{n}'_{\varepsilon}$, we have
$$ \left\|X_{\varepsilon}(t+s,t-n;X_0(n)) - X_{\varepsilon}(t+s,t-n;\bar{X}_0(n))\right\|_{H}^2\leq \left(n+\sqrt{\varepsilon}(s+n)^\delta\right)e^{-\lambda_0(s+n)}.$$
And $\mathbb{E}\left[(\vec{n}'_{\varepsilon})^q\right]<\infty$ for any $q\in(0,p-2)$.
\end{lemma}
\begin{proof}
Only need to prove (i), it follows from Lemma \ref{long-time}, there exists $\varepsilon_0,\delta$ and a sequence almost finite random variable $\{l_{\varepsilon}\}_{\varepsilon\in(0,\varepsilon_0)}$ such that
\begin{equation*}
	\left\|X_{\varepsilon}(t+s,t-n;X_0(n)) - X_{\varepsilon}(t+s,t-n;\tilde{X}_0)\right\|_{H}^2  \leq \left(\left \|X_0(n) - \tilde{X}_0 \right \|_{H}^2+\sqrt{\varepsilon}(s+n)^\delta\right)e^{-\lambda_0(s+n)},
\end{equation*}
for every $\varepsilon\in(0,\varepsilon_0)$, $t>t_0+l_{\varepsilon}$.

Set $T_{X_0} = T_{bound}\left(\left\{\left\|X_0(n)\right\|_{H}^2\right\},\frac{n}{4}\right) $, then using Bounding Lemma \ref{Bounding Lemma} we have that $T_{X_0}$ is almost finite  and
$\mathbb{E}(T_{X_0})^q, q \in(0,p-2)$ is finite a.s.
Set ${\vec{n}}_{\varepsilon} = max\left\{T_{X_0},l_{\varepsilon}\right\}$, then for every $\varepsilon\in(0,\varepsilon_0)$ and $n>\vec{n}_{\varepsilon}$
\begin{eqnarray}
	&&\sup_{\tilde{X}_0 \in A_n} \left\|X_{\varepsilon}(t+s,t-n;X_0(n)) - X_{\varepsilon}(t+s,t-n;\tilde{X}_0)\right\|_{H}^2 \nonumber \\
	&\leq& \left(\left \|X_0(n) - \tilde{X}_0 \right \|_{H}^2+\sqrt{\varepsilon}(s+n)^\delta\right)e^{-\lambda_0(s+n)}\nonumber \\
	&\leq& \left(2\left \|X_0(n) \right \|_{H}^2+2\left \| \tilde{X}_0\right \|_{H}^2+\sqrt{\varepsilon}(s+n)^\delta\right)e^{-\lambda_0(s+n)}\nonumber \\
	&\leq& \left(n+\sqrt{\varepsilon}(s+n)^\delta\right)e^{-\lambda_0(s+n)}.\nonumber
\end{eqnarray}
\end{proof}
\begin{lemma}\label{cauchy estimate}
(cf. Corollary 1 in \cite{M99}) Fix $t_1 \in \mathbb{Z}$, $\lambda_0 \in (0,2\lambda C_1)$ and $u_0 =0$. There exist $\varepsilon_0,\delta$ such that for any $\varepsilon\in(0,\varepsilon_0)$,  with probability one, there is a positive random variable $n^{*}_{\varepsilon} $ such that for all $l \geq 0$ and all $n_1,n_2 \in \mathbb{Z}$, if $n_1,n_2 < n\leq t_1 - n^{*}_{\varepsilon}$, we have
$$\left \|X_{\varepsilon}(t_1+l,n_1;X_0)- X_{\varepsilon}(t_1+l,n_2;X_0)\right\|_{H}^2 \leq\left(n+\sqrt{\varepsilon}(t_1+l-n)^\delta\right)e^{-\lambda_0(t_1+l-n)}. $$			
\end{lemma}
\begin{proof}
Assume $n_1<n_2$, then according to Lemma \ref{bounded estimate} and Lemma \ref{uniqueness proof different initial value estimate} (ii) that
\begin{eqnarray*}
	&&\left \|X_{\varepsilon}(t_1+l,n_1;X_0)- X_{\varepsilon}(t_1+l,n_2;X_0)\right\|_{H}^2  \nonumber \\
	&=& \left \|X_{\varepsilon}(t_1+l,n;X_{\varepsilon}(n,n_1;X_0))- X_{\varepsilon}(t_1+l,n;X_{\varepsilon}(n,n_2;X_0))\right\|_{H}^2 \nonumber\\
	&\leq&\left(n+\sqrt{\varepsilon}(t_1+l-n)^\delta\right)e^{-\lambda_0(t_1+l-n)}.
\end{eqnarray*}

\end{proof}
We first consider the uniqueness of the solution of the  infinite horizon stochastic integral Eq. \eqref{MIHSIE}, which is  a crucial technical condition to obtain the stationarity solution $X_{\varepsilon}^{*}$ of Eq. \eqref{Msode}.
\begin{lemma}\label{uniqueness}
For every $\varepsilon\in(0,\varepsilon_0)$, where $\varepsilon_0$ is defined in Lemma \ref{long-time}. Assume that $X^{*}_{\varepsilon}(t,\omega) (t \in \mathbb{R})$ is a $\big(\mathcal{B}(\mathbb{R})\otimes \mathcal{F},\mathcal{B}(H)\big)$-measurable, $(\mathcal{F}_{-\infty}^{t})_{t\in \mathbb{R}}$-adapted process, and for any $N\in \mathbb{Z}^{+}, X^{*}_{\varepsilon}(t,\omega)|_{t\in[-N,N]}$ $\in C([-N,N];H)$.  Moreover, if $X_{\varepsilon}^{*}(t,\omega)$ satisfies the following equation in $H$ for any $t \in \mathbb{R}$
\begin{equation}\label{MIHSIE}
	X_{\varepsilon}^{*}(t)=\int_{-\infty}^{t} S_{A}(t-s)F\left( X_{\varepsilon}^{*}\right) \mathrm{d} s+\sqrt{\varepsilon}\int_{-\infty}^{t} S_{A}(t-s) B \left( X_{\varepsilon}^{*}\right)\mathrm{d} W_{s}.
\end{equation}
and
\begin{equation}\label{uniqueness proof condition}
	\sup_{t\in \mathbb{R}}\mathbb{E}\left\|X_{\varepsilon}^{*}(t)\right\|_{H}^{2} < \infty,
\end{equation}
then $X_{\varepsilon}^{*}(t)$ is unique.
\end{lemma}
\begin{proof}
For any $n\in\mathbb{N}^{+}$ and $-n<t$,
\begin{align*}
	X_{\varepsilon}^{*}(t)=&\int_{-\infty}^{t} S_{A}(t-s)F\left( X_{\varepsilon}^{*}\right) \mathrm{d} s+\sqrt{\varepsilon}\int_{-\infty}^{t} S_{A}(t-s) B \left( X_{\varepsilon}^{*}(s)\right)\mathrm{d} W_{s}\\
	=&\int_{-\infty}^{-n} S_{A}(t-s)F\left( X_{\varepsilon}^{*}(s)\right) \mathrm{d} s+\sqrt{\varepsilon}\int_{-\infty}^{-n} S_{A}(t-s) B \left( X_{\varepsilon}^{*}(s)\right)\mathrm{d} W_{s}\\
	&+\int_{-n}^{t}S_{A}(t-s)F\left(X_{\varepsilon}^{*}(s)\right) \mathrm{d} s+\sqrt{\varepsilon}\int_{-n}^{t}S_{A}(t-s) B \left( X_{\varepsilon}^{*}(s)\right) \mathrm{d} W_{s}\\
	= &X_{\varepsilon}^{*}(-n)+ \int_{-n}^{t}S_{A}(t-s)F\left(X_{\varepsilon}^{*}(s)\right) \mathrm{d} s+\sqrt{\varepsilon}\int_{-n}^{t}S_{A}(t-s) B \left( X_{\varepsilon}^{*}(s)\right) \mathrm{d} W_{s}.
\end{align*}
Assume that $Y_{\varepsilon}^{*}$ is another solution of Eq. \eqref{MIHSIE}, then
\begin{equation*}
	Y_{\varepsilon}^{*}(t)= Y_{\varepsilon}^{*}(-n)+ \int_{-n}^{t}  S_{A}(t-s)F\left(Y_{\varepsilon}^{*}(s)\right) \mathrm{d} s+\sqrt{\varepsilon} \int_{-n}^{t}  S_{A}(t-s) B \left(Y_{\varepsilon}^{*}(s)\right)\mathrm{d} W_{s},
\end{equation*}
which imply that $X_{\varepsilon}^{*}$ and $Y_{\varepsilon}^{*}$ are solutions of Eq. \eqref{Msode} with initial value $X_{\varepsilon}^{*}(-n)$ and $Y_{\varepsilon}^{*}(-n)$ at time $-n$.
It follows from Lemma \ref{uniqueness proof different initial value estimate} (ii) that there exists $\varepsilon_0,\delta,\lambda_0>0$ and a sequence almost surely finite stopping time $\{{\vec{n}}_{\varepsilon}\}_{\varepsilon\in(0,\varepsilon_0)}$  such that for every $\varepsilon\in(0,\varepsilon_0)$, $ s>0,t\in\mathbb{R}$ when $n>{\vec{n}}_{\varepsilon}$,
\begin{align*}
	&\|X_{\varepsilon}^{*}(t+s,t-n;X_{\varepsilon}^{*}(t-n))
	-Y_{\varepsilon}^{*}(t+s,t-n;Y_{\varepsilon}^{*}(t-n))\|_{H}^2\\
	& \leq  \left(n+\sqrt{\varepsilon}(s+n)^\delta\right)e^{-\lambda_0(s+n)},
\end{align*}
under the condition  \eqref{uniqueness proof condition}, let $n\rightarrow \infty$, we can get that for any $\eta>0$
\begin{equation*}
	\|X_{\varepsilon}^{*}(t+s)-Y_{\varepsilon}^{*}(t+s)\|_{H}^2\leq \eta ,
\end{equation*}
it implies $X_{\varepsilon}^{*}(\cdot)=Y_{\varepsilon}^{*}(\cdot),\quad a.s.$
\end{proof}
The following Lemma illustrate that if Eq. \eqref{MIHSIE} has unique solution $X_{\varepsilon}^{*}$, then $X_{\varepsilon}^{*}$ is also the stationary solution of Eq. \eqref{Msode}.
\begin{lemma}\label{uniqueness of infty intergal equation promise the periodic property}
If $X_{\varepsilon}^{*}$ satisfies all conditions and \eqref{MIHSIE}, \eqref{uniqueness proof condition} in Lemma \ref{uniqueness}, then for any $l\in\mathbb{R}$
\begin{equation*}
	X_{\varepsilon}^{*}(\cdot+l,\omega) = X_{\varepsilon}^{*}(\cdot,\theta(l,\omega)),\quad a.s.
\end{equation*}
\end{lemma}
\begin{proof}  Since $\theta$ is $\mathbb{P}$-preserving on the probability space $(\Omega, \mathcal{F}, \mathbb{P})$,
for any $l\in\mathbb{R}$, we have, almost surly,
\begin{eqnarray*}
	X_{\varepsilon}^{*}(t,\theta(l,\omega))
	&=& \int_{-\infty}^{t} S_{A}(t-s)F\left( X_{\varepsilon}^{*}\left(s,\theta(l,\omega)\right)\right) \mathrm{d} s+\sqrt{\varepsilon}\int_{-\infty}^{t}S_{A}(t-s) B \left( X_{\varepsilon}^{*}(s,\theta(l,\omega))\right)\mathrm{d} W_{s}(\theta(l,\omega)) \nonumber \\
	&=& \int_{-\infty}^{t+l} S_{A}(t+l-s)F\left( X_{\varepsilon}^{*}(s-l,\theta(l,\omega))\right) \mathrm{d} s\\
	&&+\sqrt{\varepsilon}\int_{-\infty}^{t+l}S_{A}(t+l-s) B \left( X_{\varepsilon}^{*}(s-l,\theta(l,\omega))\right)\mathrm{d} W_{s}( \omega), \nonumber
\end{eqnarray*}
set $\tilde{X}_{\varepsilon}^{*}(t,\omega) = X_{\varepsilon}^{*}(t-l,\theta(l,\omega))$, then
\begin{equation*}
	\tilde{X}_{\varepsilon}^{*}(t+l,\omega) = \int_{-\infty}^{t+l} S_{A}(t+l-s)F\left( \tilde{X}_{\varepsilon}^{*}(s,\omega)\right) \mathrm{d} s+\sqrt{\varepsilon}\int_{-\infty}^{t+l} S_{A}(t+l-s) B \left( \tilde{X}_{\varepsilon}^{*}(s,\omega)\right)\mathrm{d} W_{s}( \omega),
\end{equation*}
it follows from the uniqueness of Eq. \eqref{MIHSIE} that for any $l\in\mathbb{R}$
$$X_{\varepsilon}^{*}(\cdot+l,\omega) = X_{\varepsilon}^{*}(\cdot,\theta(l,\omega)), \quad a.s.$$
\end{proof}
Next, we will construct the solution of Eq. \eqref{MIHSIE}. For $n\in\mathbb{Z}^{+}$, $\varepsilon>0$, we
define $\left\{X^n_{\varepsilon}\right\}_{n=1}^{\infty}$ by
\begin{equation*}
X^n_{\varepsilon}(t,\omega)=\begin{cases}		X_{\varepsilon}(t,-n;\omega,0), &t\geq -n,\\
	0,&t< -n.\\	
\end{cases}
\end{equation*}
\begin{lemma} \label{Mstationary}
There exists $\varepsilon_0>0$ such that for any $\varepsilon\in(0,\varepsilon_0)$, $N\in \mathbb{Z}^{+}$,    $X^n_{\varepsilon}(\cdot,\omega)\rightarrow X_{\varepsilon}^{*}(\cdot,\omega)$ in $C([-N,N];H)$ as $n\rightarrow \infty$. Moreover, $X_{\varepsilon}^{*}$ satisfies the backward infinite horizon stochastic integral Eq. \eqref{MIHSIE} and \eqref{uniqueness proof condition}.
\end{lemma}
\begin{proof}
It follows from Lemma \ref{cauchy estimate} that  there exists $\varepsilon_0>0$, such that for any $\varepsilon\in(0,\varepsilon_0)$, $N\in \mathbb{Z}^{+}$, $X^n_{\varepsilon}$ is a Cauchy sequence in $C([-N,N];H)$. And since the space $C([-N,N];H)$ is complete, there exists a $X_{\varepsilon}^{*}$ such that $\lim_{n\rightarrow\infty}X^n_{\varepsilon} = X_{\varepsilon}^{*}$ in $C([-N,N];H)$. Since $N$ is arbitrary, $X_{\varepsilon}^{*}(t,\omega)$ is defined for all time, and from Lemma \ref{bounded estimate}, we have $\sup_{n}\sup_{t\in\mathbb{R}}\mathbb{E}\left\|X^n_{\varepsilon}(t)\right\|_{H}^{2} < \infty$. This implies that
\begin{equation}\label{u* bounded estimate}
	\sup_{t\in\mathbb{R}}\mathbb{E}\left\|X_{\varepsilon}^{*}(t)\right\|_{H}^{2} < \infty.
\end{equation}
We will through two steps to prove  $X_{\varepsilon}^{*}$ satisfies Eq. \eqref{MIHSIE}.

\textbf{Step 1}. Firstly, we will prove that for any $t\geq t_0$, $X^{*}$ satisfies
\begin{equation}\label{q4}
	X_{\varepsilon}^{*}(t)= S_{A}(t-t_0)X_{\varepsilon}^{*}(t_0)+ \int_{t_{0}}^{t}  S_{A}(t-s)F\left(X_{\varepsilon}^{*}(s)\right) \mathrm{d} s+\sqrt{\varepsilon}\int_{t_{0}}^{t}S_{A}(t-s) B \left(X_{\varepsilon}^{*}(s)\right) \mathrm{d} W_{s}.
\end{equation}
For any fixed $t_0$, there exists $N\in\mathbb{N}^+$ such that $t_0\geq -N$,
\begin{align*}
	X^n_{\varepsilon}(t,\omega)=&\ \int_{-n}^{t_0} S_{A}(t-s)F\left( X_{\varepsilon}(s,-n;\omega,0)\right) \mathrm{d} s+\sqrt{\varepsilon} \int_{-n}^{t_0}  S_{A}(t-s) B \left(X_{\varepsilon}(s,-n;\omega,0)\right)\mathrm{d} W_{s}\\
	&+ \int_{t_{0}}^{t}  S_{A}(t-s)F\left( X_{\varepsilon}(s,-n;\omega,0)\right) \mathrm{d} s+\sqrt{\varepsilon}\int_{t_{0}}^{t} S_{A}(t-s) B \left( X_{\varepsilon}(s,-n;\omega,0)\right) \mathrm{d} W_{s}\\
	=&S_{t-t_0}X_{\varepsilon}(t_0,-n;\omega,0)+ \int_{t_{0}}^{t} S_{A}(t-s)F\left( X_{\varepsilon}(s,-n;\omega,0)\right) \mathrm{d} s\\
	&+\sqrt{\varepsilon} \int_{t_{0}}^{t}  S_{A}(t-s) B \left(X_{\varepsilon}(s,-n;\omega,0)\right) \mathrm{d} W_{s}.
\end{align*}

We have proved that $X^n_{\varepsilon}(\cdot,\omega)\rightarrow X_{\varepsilon}^{*}(\cdot,\omega)$ as $n\rightarrow \infty$ in $C([-N,N];H)$  for any $\varepsilon\in(0,\varepsilon_0)$,
then
according to Hypothesis \ref{Msodehy} (ii), for any $t\in[-N,N]$,
\begin{equation*}
	\begin{aligned}
		&\lim_{n\rightarrow \infty}\left\|\int_{t_{0}}^{t}  S_{A}(t-s)F\left( X_{\varepsilon}(s,-n;\omega,0)\right) \mathrm{d} s- \int_{t_{0}}^{t}  S_{A}(t-s)F\left( X_{\varepsilon}^{*}(s,\omega)\right) \mathrm{d} s\right\|_{H}^2\\
		\lesssim&\lim_{n\rightarrow \infty}\int_{t_{0}}^{t}  (t-s)^{\alpha}\|( X_{\varepsilon}(s,-n;\omega,0)-  X_{\varepsilon}^{*}(s,\omega))\|_{H}^2 \mathrm{d} s\\
		\lesssim& \lim_{n\rightarrow \infty}\sup_{s\in[-N,N]}\|( X_{\varepsilon}(s,-n;\omega,0)-  X_{\varepsilon}^{*}(s,\omega))\|_{H}^2\int_{t_{0}}^{t}  (t-s)^{\alpha} \mathrm{d} s=0,
	\end{aligned}
\end{equation*}
where the sign  $"\lesssim"$ means that the left side is less than or equal to the right side of a constant multiple.

It follows from the Burkholder-Davis-Gundy inequality and Remark \ref{condition remark} (ii) and (iii) that
\begin{align*}
	&\lim_{n\rightarrow\infty}\mathbb{E}\left\{\sup_{t\in[-N,N]}\sqrt{\varepsilon} \int_{t_{0}}^{t} S_{A}(t-s) \left[ B \left( X_{\varepsilon}(s,-n;\omega,0)\right)-  B \left( X_{\varepsilon}^{*}(s,\omega)\right)\right] \mathrm{d} W_{s}\right\}^2\\
	\lesssim& \varepsilon\lim_{n\rightarrow\infty}\int_{t_0}^{N}\mathbb{E} |S_{A}(t-s) \left[B \left(X_{\varepsilon}(s,-n;\omega,0)\right)- B \left( X_{\varepsilon}^{*}(s,\omega)\right)\right]|_{L_{Q}}^2\mathrm{d}s\\
	\lesssim&  \varepsilon\lim_{n\rightarrow\infty}\int_{t_0}^{N}\mathbb{E} |S_A(t-s)|_{L}^2| B \left(X_{\varepsilon}(s,-n;\omega,0)\right)- B \left( X_{\varepsilon}^{*}(s,\omega))\right)|_{L_{Q}}^2\mathrm{d}s\\
	\lesssim &  \varepsilon\beta^2\lim_{n\rightarrow\infty} \int_{t_0}^{N}\mathbb{E} \left\| X_{\varepsilon}(s,-n;\omega,0)- X_{\varepsilon}^{*}(s,\omega)\right\|_{H}^2\mathrm{d}s=0.
\end{align*}
Thus there exists a subsequence still set as $\{X^n_{\varepsilon}\}_{n=1}^{\infty}$ such that
\begin{equation*}
	\lim_{n\rightarrow\infty}\sqrt{\varepsilon} \int_{t_{0}}^{t}S_{A}(t-s)  B \left(X_{\varepsilon}(s,-n;\omega,0)\right) \mathrm{d} W_{s}=\sqrt{\varepsilon} \int_{t_{0}}^{t} S_{A}(t-s)  B \left( X_{\varepsilon}^{*}(s,\omega)\right) \mathrm{d} W_{s},
\end{equation*}
in $C([-N,N];H)$ $a.s.$
At the same time, $X^n_{\varepsilon}(t),S_A(t-t_0)X^n_{\varepsilon}(t_0)$ converge strongly to $X^{*}_{\varepsilon}$ and $S_A(t-t_0)X^{*}_{\varepsilon}(t_0)$ in $H$
respectively, hence let $N\rightarrow \infty$, Eq. (\ref{q4}) holds for any $t\geq t_0$.

\textbf{Step 2}. Finally, we will prove that $X^{*}_{\varepsilon}$ satisfies Eq. \eqref{MIHSIE}.\\
It follows  from Eq. (\ref{q4}), for any $0<m<n$,
\begin{eqnarray*}
	&&\int_{-n}^{-m}S_A(-r)F(X^{*}_{\varepsilon}(r,\omega))\mathrm{d}r \nonumber \\
	&=& -\sqrt{\varepsilon}\int_{-n}^{-m}S_A(-r) B (X^{*}_{\varepsilon}(r,\omega)) \mathrm{d}W(r) -T(n)X^{*}_{\varepsilon}(-n)+T(m)X^{*}_{\varepsilon}(-m). \nonumber \\
\end{eqnarray*}
Thus
\begin{eqnarray*}
	&& \mathbb{E}\left\|\int_{-n}^{-m}S_A(-r)
	F(X^{*}_{\varepsilon}(r,\omega))\mathrm{d}r \right\|_H^2  \nonumber \\
	&\leq& 2 \left(\mathbb{E}\left\|\sqrt{\varepsilon}\int_{-n}^{-m}S_A(-r) B (X^{*}_{\varepsilon}(r,\omega)) \mathrm{d}W(r) \right\|_H^2+\mathbb{E}\left\|S_A(n)X^{*}_{\varepsilon}(-n) \right\|_H^2+\mathbb{E}\left\|S_A(m)X^{*}_{\varepsilon}(-m) \right\|_H^2\right)  \nonumber \\
	&=& 2(I + II + III).
\end{eqnarray*}

It follows from the It\^{o} equality and Remark \ref{condition remark} (ii), (iii) that
\begin{align*}
	I=&\mathbb{E}\left\|\sqrt{\varepsilon}\int_{-n}^{-m}S_A(-r) B (X^{*}_{\varepsilon}(r,\omega)) \mathrm{d}W(r) \right\|_H^2 =\varepsilon \mathbb{E} \int_{-n}^{-m}|S_A(-r) B (X^{*}_{\varepsilon}(r,\omega))
	Q^{\frac{1}{2}}|_{L_2}^2\mathrm{d}r\\
	\leq&\varepsilon\mathbb{E} \int_{-n}^{-m}|S_A(-r)|_{L}^2| B (X^{*}_{\varepsilon}(r,\omega))
	|_{L_Q}^2\mathrm{d}r\\
	\leq &\varepsilon D^2\int_{-n}^{-m}|S_A(-r)|_{L}^2\mathrm{d}r,
\end{align*}
is a Cauchy sequence, we obtain that $I\rightarrow 0$ as $m,n\rightarrow \infty$.

Moreover, because $II \leq |S_A(n)|_{L}^2\mathbb{E}\left\|X^{*}_{\varepsilon}(-n)\right\|_H^2$ and $III \leq |S_A(n)|_{L}^2\mathbb{E}\left\|X^{*}_{\varepsilon}(-m)\right\|_H^2$. Then form Remark
(\ref{condition remark}) (iii) we know $II$ and $III$ converge to 0 as $m,n \rightarrow\infty$. So
\begin{equation*}
	\int_{-n}^{t}S_A(-r)F(X^{*}_{\varepsilon}(r,\omega))\mathrm{d}r \quad \text{and}\quad\int_{-n}^{t}S_A(-r) B (X^{*}_{\varepsilon}(r,\omega)) \mathrm{d}W(r)
\end{equation*}
are Cauchy sequences in $L^2(\Omega,H)$ with respect to $n$ for any $t\in\mathbb{R}$, we  get that
\begin{eqnarray*}
	\int_{-n}^{t}S_A(-r)F(X^{*}_{\varepsilon}(r,\omega))\mathrm{d}r&\longrightarrow& \int_{-\infty}^{t}S_A(-r)F(X^{*}_{\varepsilon}(r,\omega))\mathrm{d}r ,\nonumber\\
	\int_{-n}^{t}S_A(-r) B (X^{*}_{\varepsilon}(r,\omega)) \mathrm{d}W(r) &\longrightarrow &\int_{-\infty}^{t}S_A(-r) B (X^{*}_{\varepsilon}(r,\omega)) \mathrm{d}W(r), \quad \text{as}~ n\rightarrow \infty ~\text{in} ~H. \nonumber\\
\end{eqnarray*}
Furthermore, $\mathbb{E}\left\|S_{A}(t+n)X^{*}_{\varepsilon}(-n) \right\|_H^2\leq |S_A(t+n)|_{L}^2 \mathbb{E}\left\|X^{*}_{\varepsilon}(-n)\right\|_H^2\rightarrow0,~as~ n\rightarrow\infty.$ Thus combining Eq. (\ref{q4}) we can get that $X^{*}_{\varepsilon}$ satisfies Eq. \eqref{MIHSIE}.
\end{proof}
The proof of Theorem \ref{Mstationary solution} in the following.
\begin{proof}
Combining Lemma \ref{uniqueness of infty intergal equation promise the periodic property} and Lemma \ref{Mstationary}, $X_{\varepsilon}^*$ be the stationary solution of Eq. \eqref{Msode} (in the sense of Definition \ref{Defition of SS}).  Moreover, $X_{\varepsilon}^{*}(t,\omega)$ satisfies Eq. \eqref{MIHSIE} in $H$ for any $t \in \mathbb{R}$ and \eqref{uniqueness proof condition}.
\end{proof}
\begin{lemma}
The solution $X_{\varepsilon}^*$ of the Eq. \eqref{MIHSIE} is the unique stationary
solution of Eq. \eqref{Msode}.
\end{lemma}
\begin{proof}
Let $\tilde{X}_{\varepsilon}^{*}(\cdot,\omega)$ be another stationary solution of Eq. \eqref{Msode}. Denote $\tilde{X}_{\varepsilon}^{*}(\omega)=\tilde{X}_{\varepsilon}^{*}(0, \omega)$ and $X_{\varepsilon}^*(\omega)=X_{\varepsilon}^*(0, \omega)$. It follows from the cocycle property and Lemma \ref{long-time} that there exists $\varepsilon_0,\delta,\lambda_0>0$ and a sequence almost surely finite stopping time $\{l_{\varepsilon}\}_{\varepsilon\in(0,\varepsilon_0)}$  such that for every $\varepsilon\in(0,\varepsilon_0)$, $ s>0,t\in\mathbb{R}$ when $t>l_{\varepsilon}$,
\begin{equation}\label{17}
	\begin{aligned}
		\left\|\tilde{X}_{\varepsilon}^{*}(\theta(t;\omega))-X_{\varepsilon}^*(\theta(t, \omega))\right\|_{H}^{2} &=\left\|\tilde{X}_{\varepsilon}(t, 0;\omega, \tilde{X}_{\varepsilon}^{*}(\omega) )-X_{\varepsilon}\left(t, 0; \omega, X_{\varepsilon}^*(\omega)\right)\right\|_{H}^{2} \\
		&\leq \left\|\tilde{X}_{\varepsilon}^{*}(\omega)-X_{\varepsilon}^{*}(\omega)
		+\sqrt{\varepsilon}t^\delta\right\|_{H}^{2} e^{-\lambda_0 t}.
	\end{aligned}
\end{equation}

As $\theta(t, \omega)$ is a $\mathbb{P}$-preserving ergodic Wiener shift, we have, for any $\eta>0, t>0$,
$$
\mathbb{P}\left(\left\|\tilde{X}_{\varepsilon}^{*}(\theta(t, \omega))-X_{\varepsilon}^{*}(\theta(t, \omega))\right\|_{H}^{2}>\eta\right)=
\mathbb{P}\left(\left\|\tilde{X}_{\varepsilon}^{*}( \omega)-X_{\varepsilon}^{*}(\omega)\right\|_{H}^{2}>\eta\right) .
$$
In fact, by the inequality (\ref{17}), we know that
$$
\left\{\omega:\left\|\tilde{X}_{\varepsilon}^{*}(\theta(t, \omega))-X_{\varepsilon}^{*}(\theta(t, \omega))\right\|_{H}^{2}>\eta\right\} \subset\left\{\omega: \left\|\tilde{X}_{\varepsilon}^{*}(\omega)-X_{\varepsilon}^{*}(\omega)+\sqrt{\varepsilon}t^\delta\right\|_{H}^{2} e^{-\lambda_0 t}>\eta\right\},
$$
and
$$
\mathbb{P}\left(\left\|\tilde{X}_{\varepsilon}^{*}(\theta(t, \omega))-X_{\varepsilon}^{*}(\theta(t, \omega))\right\|_{H}^{2}>\eta\right) \leq \mathbb{P}\left(\left\|\tilde{X}_{\varepsilon}^{*}(\omega)-X_{\varepsilon}^{*}(\omega)+\sqrt{\varepsilon}t^\delta\right\|_{H}^{2} e^{-\lambda_0 t}>\eta\right).
$$
However,
$$
\lim _{t \rightarrow \infty} \mathbb{P}\left(\left\|\tilde{X}_{\varepsilon}^{*}(\omega)-X_{\varepsilon}^{*}(\omega)+\sqrt{\varepsilon}t^\delta\right\|_{H}^{2} e^{-\lambda_0 t}>\eta\right)=0.
$$
This implies that $\mathbb{P}\left(\tilde{X}_{\varepsilon}^{*}(\omega)=X_{\varepsilon}^{*}(\omega)\right)=1$.
\end{proof}
\appendix

\section*{Appendix A: Weak convergence method in infinite intervals}\label{LPp}
\addcontentsline{toc}{section}{Appendix A: Weak convergence method in infinite intervals}
\renewcommand\thesection{A}
\setcounter{equation}{0}
\renewcommand\theequation{A.\arabic{equation}}
The Bou\'e-Dupuis formula in infinite intervals has been used directly without proof in \cite{BG20} by Barashkov and Gubinelli. Although we believe  that the experts in field of LDP and Gaussian measures are familiar with the Bou\'e-Dupuis formula in infinite intervals, we still present the proof of Bou\'e-Dupuis
formula and the weak convergence approach in infinite intervals for  the convenience of readers.

The main difference between the proof in infinite intervals and finite intervals appears in the lower bounded proof of Theorem \ref{Theorem7.2}, since bounded functions are integrable in finite intervals, but not in infinite intervals. However, the integrable function of infinite intervals can be approximated by simple function, and then the proof for infinite intervals can be transformed into finite intervals. Other routine proofs are shown for  the completeness of the present paper.
\subsection{Bou\'e-Dupuis formula in infinite interval}

\begin{lemma}{\label{lemmaB.1}}(cf. Problem 3.19 in \cite{KS88})
	The following conditions are equivalent for a continuous martingale $\{X_t,\mathcal{F}_t;0\leq t< \infty\}$.\\
	(i). It is a uniformly integrable family of random variables.\\
	(ii). It converges in $L^1$, as $t\longrightarrow\infty$.\\
	(iii). It converges $\mathbb{P}$-a.s. (as $t\longrightarrow\infty$) to an integrable random variable $X_{\infty}$, such that
	\begin{equation*} \{X_t,\mathcal{F}_t;0\leq t< \infty\}
	\end{equation*}
	is a martingale.\\
	(iv). There exists an integrable random variable $Y$, such that $X_t=\mathbb{E}(Y|\mathcal{F})$ $\mathbb{P}$-a.s., for every $t\geq0$.
\end{lemma}
Let $(\mathbb{W},\mathbb{H},\mu)$ be an abstract Wiener space. Namely, $(\mathbb{W},\|\cdot\|_{\mathbb{W}})$ is a separable Banach space, $(\mathbb{H},\|\cdot\|_{\mathbb{H}})$ is a separable Hilbert space densely and continuously embedded in $\mathbb{W}$, and $\mu$ is the Gaussian measure over $\mathbb{W}$. If we identify the dual space $\mathbb{H}^{\ast}$ with itself, then $\mathbb{W}^{\ast}$ may be viewed as a dense linear subspace of $\mathbb{H}$ so that $l(w)=\langle l,w\rangle_{\mathbb{H}}$ whenever $l\in\mathbb{W}^{\ast}$ and $w\in\mathbb{H}$ , where $\langle\cdot,\cdot\rangle_{\mathbb{H}}$ denotes the inner product in $\mathbb{H}$.	

We now recall some notations from \cite{Z09} about the filtration in abstract Wiener space. In what follows, we fix a continuous and strictly monotonic resolution $\pi=\{\pi_t,t\in[0,1]\}$ of the identity in $\mathbb{H}$, \\
(i). For each $t\in[0,1]$, $\pi_t$ is an orthogonal projection.\\
(ii). $\pi_0=0,\pi_1=I$.\\
(iii). For $0\leq s<t\leq 1$, $\pi_s\mathbb{H}\subsetneq \pi_t\mathbb{H}$.\\
(iv). For any $h\in\mathbb{H}$ and $t\in[0,1]$, $\lim_{s\rightarrow t}\pi_sh=\pi_th$.\\
For any $h\in\mathbb{H}$, there exists a  sequence $h_n\in \mathbb{W}^{\ast}$ such that $\lim\limits_{n\rightarrow\infty}\|h_n-h\|_{\mathbb{H}}=0$. Thus, there exists a $\delta(h)\in L^2(\mathbb{W},\mathcal{B}(\mathbb{W}),\mu)$ such that
\begin{equation*}
	\lim_{n\rightarrow\infty}\mathbb{E}|h_n(\cdot)-\delta(h)|^2=0.
\end{equation*}
The $\delta(h)(w)$ is also written as $\langle h,w\rangle$, called the Skorohod integral of $h$.

After taking $f(x)=\frac{x}{1-x}$, we can obtain a continuous and strictly monotonic resolution $\pi=\{\pi_t,0\leq t\leq\infty\}$ of the identity in $\mathbb{H}$, i.e.\\
(i). For each $t\in[0,\infty]$, $\pi_t$ is an orthogonal projection.\\
(ii). $\pi_0=0,\pi_\infty=I$.\\
(iii). For $0\leq s<t\leq \infty$, $\pi_s\mathbb{H}\subsetneq \pi_t\mathbb{H}$.\\
(iv). For any $h\in\mathbb{H}$ and $t\in[0,\infty]$, $\lim_{s\rightarrow t}\pi_sh=\pi_th$.

If we take another transform, we can also get the filtration on $\mathbb{R}$.
\begin{definition}
 \label{Definition7.1}
	The continuous filtation on $(\mathbb{W},\mu)$ is defined by
	\begin{equation*}
		\mathcal{F}_t:=\mathcal{ B} \{\delta(\pi_th),h\in\mathbb{H}\}\vee\mathcal{N},
	\end{equation*}
	where $\mathcal{N}$ is the collection of all the null sets in $\mathbb{W}$ with respect to $\mu$. We write $\mathcal{F}_{\infty}$ as $\mathcal{F}$, and remark that $\mathcal{B}(\mathbb{W})\subset\mathcal{F}$.
\end{definition}
Below, we shall consider the filtered probability space $(\mathbb{W},\mathcal{F},(\mathcal{F}_t)_{0\leq t\leq \infty}, \mu)$. If there is no special declaration, the expectation $\mathbb{E}$ and the term $``a.s"$ are always taken with respect to the Wiener measure $\mu$.
\begin{definition}\label{Definition7.2}
	For every $0\leq t\leq \infty$, let $\mathcal{C}_t$ be the collection of all cylindrical function with the form
	\begin{equation}\label{7.2}
		F(w)=g(\langle\pi_th_1,w\rangle,...,\langle\pi_th_n,w\rangle);\quad g\in C_b^{\infty}(\mathbb{R}^n),\quad h_1,...,h_n\in\mathbb{W}^{\ast}.
	\end{equation}
	In particular, the elements in $\mathcal{C}_t$ are measurable with respect to $\mathcal{F}_t$, We write $\mathcal{C}_1$ as $\mathcal{C}$.
\end{definition}
We have the following simple approximation result.
\begin{lemma}\label{Lemma7.1}
	For a fixed $0\leq t\leq \infty$, lat $F$ be an $\mathcal{F}_t$ measurable and bounded function on $\mathbb{W}$ with bound $N$. There exists a sequence $F_k\in\mathcal{C}_t$ such that
	\begin{equation*}
		\|F_k\|_{\infty}:=\sup_{w\in\mathbb{W}}|F_k(w)|\leq N \quad and \quad F_k\rightarrow F,\quad a.s.
	\end{equation*}
	In particular, $\mathcal{C}_t$ is dense in $L^2(\mathbb{W},\mathcal{F}_t,\mu)$.
\end{lemma}
\begin{definition}\label{Definition7.3}
	An $\mathbb{H}$ valued random variable $v$ is called adapted to $\mathcal{F}_t$ if for every $0\leq t\leq \infty$ and $h\in\mathbb{H},\langle\pi_th,w\rangle_{\mathbb{H}}\in\mathcal{F}_t$. All the adapted $\mathbb{H} $ valued random variables in $L^2(\mathbb{W},\mathcal{F},\mu;\mathbb{H})$ is denoted by $\mathcal{H}^a$. The set of all bounded elements in $\mathcal{H}^a$ is denoted by $\mathcal{H}_b^a$, i.e.
	\begin{equation*}
		\mathcal{H}_b^a:=\{v\in\mathcal{H}^a:\|v(w)\|_{\mathbb{H}}\leq N \ a.s.-w\ for\ some\ N>0\}.
	\end{equation*}
\end{definition}
A $v\in \mathcal{H}^a$ is called simple if it has the following form
\begin{equation*}
	v(w)=\sum_{i=0}^{n-1}\xi_i(w)(\pi_{t_{i+1}}-\pi_{t_i})h_i,\quad \xi\in\mathcal{C}_{t_i},\quad h_i\in\mathbb{H},
\end{equation*}
where $0=t_0<t_1<...<t_n<\infty$. The set of all simple elements in $\mathcal{H}^a$ is denoted  by $\mathcal{S}^a$. We write $\mathcal{S}_b^a:=\mathcal{S}^a\cap\mathcal{H}_b^a$.
\begin{proposition}\label{7.2}
	$\mathcal{H}^a$ is a closed subspace of $L^2(\mathbb{W},\mathcal{F},\mu;\mathbb{H})$, and $\mathcal{S}_b^a$ is dense in $\mathcal{H}^a$.
\end{proposition}
Basing on this Proposition, for any $v\in \mathcal{H}^a$, we can define It\^{o}'s integral $\delta(v)$ such that
\begin{equation*}
	\mathbb{E}|\delta(v)|^2=\mathbb{E}\|v\|_{\mathbb{H}}^2.
\end{equation*}
On the other hand, for any $\|\cdot\|_{L^2}$ norm finite and real Borel measurable function $f$ on $[0,\infty)$ and $h\in\mathbb{H}$, we may define the following integral with respect to the vector valued measure
\begin{equation*}
	\int_0^{\infty}f(s)\mathrm{d}\pi_sh
\end{equation*}
such that
\begin{equation*}
	\Big\|	\int_0^{\infty}f(s)\mathrm{d}\pi_sh\Big\|_{\mathbb{H}}^2=\int_0^{\infty}|f(s)|^2\mathrm{d}\langle\pi_sh,h\rangle_{\mathbb{H}}.
\end{equation*}
It is standard to prove the following result.
\begin{lemma}\label{Lemma7.2}
	Let $f$ be a left-continuous $\mathcal{F}_t$ adapted process the $\|\cdot\|_{L^2(\mathbb{R}^+)}$  and $\|\cdot\|_{L^{\infty}}$ bounded by $N$. Then for any $h\in\mathbb{H}$, there exists a sequence $v_k^h\in\mathcal{S}_b^a$ such that
	\begin{equation*}
		\|v_k^h(w)\|_{\mathbb{H}}\leq N\cdot\|h\|_{\mathbb{H}},\quad a.s.
	\end{equation*}
	and
	\begin{equation*}
		\lim_{k\rightarrow\infty}\mathbb{E}\|v_k^h-\int_0^{\infty}f(s)\mathrm{d}\pi_sh\|_{\mathbb{H}}^2=0.
	\end{equation*}
\end{lemma}
\begin{proof}
	First of all, we define for every $n\in\mathbb{N}$
	\begin{equation*}
		f_n(s):=\sum_{j=0}^{n2^n-1}f(j2^{-n})I_{[j2^{-n},(j+1)2^{-n})}(s).
	\end{equation*}
	Then, by the dominated convergence theorem we have
	\begin{equation*}
		\lim_{n\rightarrow\infty}\mathbb{E}\Big\|\int_0^{\infty}(f_n(s)-f(s))\mathrm{d}\pi_sh\Big\|_{\mathbb{H}}^2=	\lim_{n\rightarrow\infty}\mathbb{E}\int_0^{\infty}|f_n(s)-f(s))|^2\mathrm{d}\langle\pi_sh,h\rangle_{\mathbb{H}}=0.
	\end{equation*}
	For each $n\in\mathbb{N}$ and $j=0,...,n2^{-n}-1$, by Lemma \ref{Lemma7.1} one can find $F_{n,k}^j\in\mathcal{C}_{j2^{-n}}$ such that
	\begin{equation*}
		\|F_{n,k}^j\|_{\infty}\leq N,\quad F_{n,k}^j\rightarrow f(j2^{-n}),\ as\  k\rightarrow\infty.
	\end{equation*}
	Finally, we define
	\begin{equation*}
		v_{n,k}^h(w):=\sum_{j=0}^{n2^{-n}-1}F_{n,k}^j(w)\cdot(\pi_{(j+1)2^{-n}}-\pi_{j2^{-n}})h.
	\end{equation*}
	By the diagonalization method, we may find the desired sequence $v_k^h$. In fact, the condition $\|\cdot\|_{L^{\infty}}< N$ can be ignored.
\end{proof}
\begin{proposition}\label{Proposition7.3}
	Let $0<c\leq F\leq C$ be a Borel measurable function on $\mathbb{W}$. Then there exists a $v\in \mathcal{H}^a$ such that
	\begin{equation*}
		\mathbb{E}(F|\mathcal{F}_t)=\mathbb{F}\cdot\exp\Big\{\delta(\pi_tv)-\frac{1}{2}\|\pi_tv\|_{\mathbb{H}}^2\Big\},\quad 0\leq t\leq \infty.
	\end{equation*}
\end{proposition}
\begin{proof}
	Set $M_t:=\mathbb{E}(F|\mathcal{F}_t)$. Then $\{M_t,\mathcal{F}_t\}$ is a uniformly martingale bounded from above by $C$ and from below by $c$. By Lemma \ref{lemmaB.1} and the representation formula of martingales, there is a $u\in\mathcal{H}^a$ such that
	\begin{equation*}
		M_t=\mathbb{F}+\delta(\pi_tu).
	\end{equation*}
	Now define
	\begin{equation*}
		v:=\int_0^{\infty}\frac{\mathrm{d}\pi_tu}{M_t},\ m_t:=\delta(\pi_tv).
	\end{equation*}
	Then, clearly $v\in\mathcal{H}^a$ and $\{m_t,\mathcal{F}_t\}$ is a martingale with square variation process $t\rightarrow\|\pi_tv\|_{\mathbb{H}}$. Thus, we have
	\begin{equation*}
		M_t=\mathbb{F}+\int_0^tM_sdm_s.
	\end{equation*}
	The desired formula follows.
\end{proof}
We also need the following Clark-Ocone formula.
\begin{proposition}\label{Proposition7.4}
	For any $F\in\mathcal{C}$ with the form \eqref{7.2}, it then holds
	\begin{equation*}
		\mathbb{E}(F|\mathcal{F}_t)=\mathbb{E}F+\delta(\pi_tv),\quad \forall 0\leq t\leq \infty,
	\end{equation*}
	where
	\begin{equation*}
		v:=\sum_{i=1}^{n}\int_0^{\infty}\mathbb{E}\big[(\partial_ig)(\langle h_1,\cdot\rangle,...,\langle h_n,\cdot\rangle)|\mathcal{F}_t\big] \mathrm{d}\pi_th_i \in\mathcal{H}_b^a.
	\end{equation*}	
\end{proposition}
Let $\mathcal{H}_c^a$ be the set of all $v\in\mathcal{H}^a$ satisfying
\begin{equation*}
	\mathbb{E}\big[\exp\{\delta(v)-\frac{1}{2}\|v\|_{\mathbb{H}}^2\}\big]=1.
\end{equation*}
For $v\in\mathcal{H}_c^a$, we define
\begin{equation*}
	T_v(w):=w-v(w)
\end{equation*}
and
\begin{equation}\label{7.3}
	\mathrm{d}\mu_v=\exp\big\{\delta(v)-\frac{1}{2}\|v\|_{\mathbb{H}}^2\big\}\mathrm{d}\mu.
\end{equation}
Then by the Girsanov theorem, we have for any $A\in\mathcal{B}(\mathbb{W})$
\begin{equation*}
	\mu_v(w:T_vw\in A)=\mu(A).
\end{equation*}
\begin{lemma}\label{Lemma7.3}
	For $v\in\mathcal{H}_c^a$, let $\mu_v$  be defined by \eqref{7.3}. Then
	\begin{equation*}
		R(\mu_v||\mu)=\frac{1}{2}\mathbb{E}^{\mu_v}\|v\|_{\mathbb{H}}^2.
	\end{equation*}
\end{lemma}
\begin{theorem}\label{Theorem7.1}
	Let $F$ be any bounded Borel measurable function on $\mathbb{W}$. Then
	\begin{equation*}
		-\log\mathbb{E}(e^{-F})=\inf_{v\in\mathcal{H}_c^a}\mathbb{E}^{\mu_v}\big(F+\frac{1}{2}\|v\|_{\mathbb{H}}^2\big),
	\end{equation*}
	where  $\mu_v$ is defined by \eqref{7.3}. Moreover, the infimum is unique attained at some $v_0\in\mathcal{H}_c^a$.
\end{theorem}
\begin{proposition}\label{Proposition7.5}
	Let $F$ be any bounded Borel measurable function on $\mathbb{W}$. For any $v\in\mathcal{S}_b^a$, there are two $\tilde{v},\hat{v}\in\mathcal{S}_b^a$ such that
	\begin{align*}
		\mathbb{E}^{\mu_{\tilde{v}}}\big(F+\frac{1}{2}\|\tilde{v}\|_{\mathbb{H}}^2\big)&=\mathbb{E}\big(F(\cdot+v)+\frac{1}{2}\|v\|_{\mathbb{H}}^2\big).\\
		\mathbb{E}^{\mu_{v}}\big(F+\frac{1}{2}\|v\|_{\mathbb{H}}^2\big)&=\mathbb{E}\big(F(\cdot+\hat{v})+\frac{1}{2}\|\hat{v}\|_{\mathbb{H}}^2\big).
	\end{align*}
	Moreover
	\begin{equation}\label{7.4}
		R(\mathcal{L}_{\mu}(\cdot+v)||\mu)=\frac{1}{2}\mathbb{E}\|v\|_{\mathbb{H}}^2,
	\end{equation}
	where $\mathcal{L}_{\mu}(\cdot+v)$ denotes the law of $w\mapsto w+v(w)$ in $(\mathbb{W},\mathcal{F})$ under $\mu$.
\end{proposition}
Reader interested in proof can refer to \cite{Z09}.
\begin{theorem}\label{Theorem7.2}
	Let $F$ be a bounded Borel measurable function on $\mathbb{W}$. Then we have
	\begin{align}\label{B-D formula}
		-\log\mathbb{E}(e^{-F})&=\inf_{v\in\mathcal{H}^a}\mathbb{E}\big(F(\cdot+v)+\frac{1}{2}\|v\|_{\mathbb{H}}^2\big)\nonumber\\
		&=\inf_{v\in\mathcal{S}_b^a}\mathbb{E}\big(F(\cdot+v)+\frac{1}{2}\|v\|_{\mathbb{H}}^2\big).
	\end{align}
\end{theorem}
\begin{proof}
	\textbf	{(Upper bound)}Let $v\in\mathcal{H}^a$. By Proposition \ref{7.2} we may choose a sequence of $v_n\in\mathcal{S}_b^a$ such that
	\begin{equation*}
		\lim_{n\rightarrow\infty}\mathbb{E}\|v_n-v\|_{\mathbb{H}}^2.
	\end{equation*}
	So, $(w,v_n)$ converges in distribution to $(w,v)$ in $\mathbb{W}\times\mathbb{H}$, and $\mathcal{L}_{\mu}(\cdot+v_n)$ converges weakly to $\mathcal{L}_{\mu}(\cdot+v)$. Noting that by Eq. (\ref{7.4})
	\begin{equation*}
		\sup_nR(\mathcal{L}_{\mu}(\cdot+v_n)||\mu)=\frac{1}{2}\sup_n\mathbb{E}\|v_n\|_{\mathbb{H}}^2<\infty,
	\end{equation*}
	and we have (\cite{{Z09}} Lemma2.1(ii))
	\begin{equation*}
		\lim_{n\rightarrow\infty}\mathbb{E}(F(\cdot+v_n))=\mathbb{E}(F(\cdot+v)).
	\end{equation*}
	Therefore,by Theorem \ref{Theorem7.1}, we get the upper bound
	\begin{equation*}
		-\log\mathbb{E}(e^{-F})\leq \lim_{n\rightarrow\infty}\mathbb{E}\big(F(\cdot+v_n)+\frac{1}{2}\|v_n\|_{\mathbb{H}}^2\big)=\mathbb{E}\big(F(\cdot+v)+\frac{1}{2}\|v\|_{\mathbb{H}}^2\big).
	\end{equation*}
	\textbf{(Lower bound)} We divided the proof into two steps.
	
	\textbf{Step 1}: We first assume that $F\in\mathcal{C}$ with the form
	\begin{equation*}
		F(w):=g\big(\langle h_1,w\rangle,...,\langle h_n,w\rangle\big),\quad g\in C_b^{\infty}(\mathbb{R}^n),\quad h_1,...,h_n\in\mathbb{W}^{\ast}.
	\end{equation*}
	By Proposition \ref{Proposition7.4}, there is a $v\in\mathcal{H}^a$ such that for all $0\leq t\leq\infty$
	\begin{equation}\label{7.5}
		\mathbb{E}(e^{-F}|\mathcal{F}_t)=\mathbb{E}(e^{-F})\cdot \exp\{\delta(\pi_tv)-\frac{1}{2}\|\pi_tv\|_{\mathbb{H}}^2\}.
	\end{equation}
	By Proposition \ref{Proposition7.4}, we in fact find from the proof of Proposition \ref{Proposition7.3} that
	\begin{equation*}
		v:=\sum_{i=1}^n\int_0^{\infty}\frac{1}{\mathbb{E}(e^{-F}|\mathcal{F}_t)}\cdot\mathbb{E}\big[e^{-F}\cdot(\partial_ig)(\langle h_1,\cdot\rangle,...,\langle h_n,\cdot\rangle)|\mathcal{F}_t\big]\mathrm{d}\pi_th_i.
	\end{equation*}
	Thus, $v\in\mathcal{H}_b^a$. By Lemma \ref{Lemma7.2} there exists a sequence $v_{n}\in\mathcal{S}_b^a$ satisfying for some $C_F>0$
	\begin{equation*}
		\|v_{n}\|_{\mathbb{H}}\leq C_F\quad a.s.
	\end{equation*}
	such that
	\begin{equation*}
		\lim_{{k_n}\rightarrow\infty}\mathbb{E}\|v_{n}-v\|_{\mathbb{H}}^2=0.
	\end{equation*}
	By extracting a subsequence if necessary, we may further assume that
	\begin{equation*}
		\|v_{n}-v\|_{\mathbb{H}}^2\rightarrow0\quad a.s,\quad \delta(v_{n})\rightarrow\delta(v)\ a.s.
	\end{equation*}
	we have by the dominated convergence theorem
	\begin{equation*}
		\mathbb{E}^{\mu_{n}}\big(F+\frac{1}{2}\|v_{n}\|_{\mathbb{H}}^2\big)\rightarrow	\mathbb{E}^{\mu}\big(F+\frac{1}{2}\|v\|_{\mathbb{H}}^2\big)\quad as\ n\rightarrow\infty.
	\end{equation*}
	Here we have used the uniform integrability of $\{e^{\delta(v_{n})},n\in\mathbb{N}\}$.
	Moreover, by Eq. (\ref{7.5}) and Lemma \ref{Lemma7.3}
	\begin{equation*}
		-\log\mathbb{E}(e^{-F})=\mathbb{E}^{\mu_v}(F)+R(\mu_v||\mu)=\mathbb{E}^{\mu_v}(F+\frac{1}{2}\|v\|_{\mathbb{H}}^2).
	\end{equation*}
	
	\textbf{Step 2}: Let $F$ be a bounded measurable function on $(\mathbb{W},\mathcal{F})$. By Lemma \ref{Lemma7.1} we can choose a sequence $F_n\in\mathcal{C}$ such that $\|F_n\|_{\infty}\leq\|F\|_{\infty}<\infty$, and $\lim\limits_{n\rightarrow\infty}F_n=F \ \mu-$a.s.. For any $\epsilon>0$ and $F_n$,  by Step 1 there exists a $v_n\in\mathcal{S}_b^a$ such that
	\begin{equation}\label{7.6}
		-\log\mathbb{E}(e^{-F_n})\geq\mathbb{E}\big(F_n(\cdot+v_n)+\frac{1}{2}\|v_n\|_{\mathbb{H}}^2\big)-\epsilon.
	\end{equation}
	In view $\eqref{7.4}$ and $\eqref{7.6}$, we have
	\begin{equation*}
		\sup_nR(\mathcal{L}_{\mu}(\cdot+v_n)||\mu)\leq\frac{1}{2}\sup_n\mathbb{E}\|v_n\|_{\mathbb{H}}^2\leq 2\|F\|_{\infty}+\epsilon.
	\end{equation*}
	So there is a subsequence $n_k$ such that $\mathbb{E}\|v_{n_k}\|_{\mathbb{H}}^2$ convergence. We have
	\begin{equation*}
		\lim_{k\rightarrow\infty}\mathbb{E}|F_{n_k}(\cdot+v_{n_k})-F(\cdot+v_{n_k})|=0.
	\end{equation*}
	Dominated convergence theorem gives that for sufficiently large $k$,
	\begin{equation*}
		-\log\mathbb{E}(e^{-F})\geq\mathbb{E}\big(F(\cdot+v_{n_k})+\frac{1}{2}\|v_{n_k}\|_{\mathbb{H}}^2\big)-2\epsilon.
	\end{equation*}
	Since $v_{n_k}\in\mathcal{S}_b^a$, we thus complete the proof of the lower bound.
\end{proof}	
\subsection{The proof of Laplace principle for infinite intervals }\label{LP proof}
Since the large deviation principe is equivalent to Laplace principle  in Polish space, we only need to prove the Laplace principle for infinite intervals, i.e. Theorem \ref{LP}.
\begin{proof}
	In order to prove the Laplace principle, we must show that Eq. (\ref{LPeq}) holds for all real valued, bounded and continuous functions $h$ on space $E$.
	
	\textbf{(Lower bound)}
	Define
	\begin{align*}
		\mathbb{W}&=C(\mathbb{R};H),\quad\mathbb{H}=\big\{h=\int_{-\infty}^{\cdot}h^{\prime}\mathrm{d}s; \quad \int_{-\infty}^{\infty}\|h^{\prime}\|_{H_0}^2\mathrm{d}s<\infty\big\}, \\
		\delta(v)&=\int_{-\infty}^t B (v(t)) \mathrm{d} W(t),\quad\pi_{\theta}h=	\pi_{\theta}h\int_{-\infty}^{\cdot}h^{\prime}\mathrm{d}s=\int_{-\infty}^{\cdot} \mathbbm{1}_{(-\infty,\theta)}h^{\prime}\mathrm{d}s,\quad\forall h\in \mathbb{H}.
	\end{align*}
	It follows from the variational presentation (\ref{B-D formula}) that
	$$
	\begin{aligned}
		-& \log \mathbb{E}\left\{\exp \left[-\frac{1}{\varepsilon} h\left(X^{\varepsilon}(\cdot)\right)\right]\right\} \\
		&=\inf _{v \in \mathcal{A}} \mathbb{E}\left(\frac{\varepsilon}{2} \int_{-\infty}^{+\infty}\|v(s)\|_{H_0}^{2} \mathrm{d} s+h \circ \mathcal{G}^{\varepsilon}\left(W(\cdot)+\int_{-\infty}^{\cdot} v(s) \mathrm{d} s\right)\right).
	\end{aligned}
	$$
	
	Then for every $\delta>0$, there exists $\tilde{v}^{\varepsilon} \in \mathcal{A}$ such that
	\begin{equation}\label{v}
		\begin{aligned}
			&\inf _{v \in \mathcal{A}}\mathbb{E}\left(\frac{\varepsilon}{2} \int_{-\infty}^{+\infty}\|v(s)\|_{H_0}^2 \mathrm{d} s+h \circ \mathcal{G}^{\varepsilon}\left(W(\cdot)+ \int_{-\infty}^{\cdot} \tilde{v}(s) \mathrm{d} s\right)\right) \\
			&\geq \mathbb{E}\left(\frac{\varepsilon}{2} \int_{-\infty}^{+\infty}\|v^{\varepsilon}(s)\|_{H_0}^{2} \mathrm{d} s+h \circ \mathcal{G}^{\varepsilon}\left(W(\cdot)+ \int_{-\infty}^{.} v^{\varepsilon}(s) \mathrm{d} s\right)\right)-\delta.
		\end{aligned}
	\end{equation}
	We will prove that
	$$
	\begin{aligned}
		\liminf _{\varepsilon \longrightarrow 0}\mathbb{E} &\left(\frac{\varepsilon}{2} \int_{-\infty}^{+\infty}\|v^{\varepsilon}(s)
		\|_{H_0}^{2} \mathrm{d} s+h \circ \mathcal{G}^{\varepsilon}\left(W(\cdot)+ \int_{-\infty}^{.} v^{\varepsilon}(s) \mathrm{d} s\right)\right) \\
		& \geq \inf _{x \in E}\{I(x)+h(x)\}.
	\end{aligned}
	$$
	
	We claim that  we can assume without loss of generality that for all $\varepsilon>0$ and a.s.
	\begin{equation}
		\label{4.6}
		\varepsilon\int_{-\infty}^{+\infty}\|v^{\varepsilon}
		(s)\|_{H_0}^{2} \mathrm{d} s \leq N
	\end{equation}
	for some finite number $N$.
	
	To see this, observe that if $M :=\|h\|_{\infty}$ then $\sup _{\varepsilon>0} \mathbb{E}\left(\frac{\varepsilon}{2} \int_{-\infty}^{+\infty}\|v^{\varepsilon}(s)\|_{H_0}^{2} \mathrm{d} s \right) \leq 2 M+\delta<\infty$. Now define random variable  $$\tau_{N}^{\varepsilon} := \inf
	\left\{t \in\mathbb{R}: \frac{\varepsilon}{2}\int_{-\infty}^{t}\|v^{\varepsilon}(s)\|_{H_0}^{2} \mathrm{d} s \geq N\right\} \wedge \infty .$$

The processes $v^{\varepsilon, N}(s):= v^{\varepsilon}(s)  \mathbbm{1}_{(-\infty, \tau_{N}^{\varepsilon}]}(s)$ are
	in $\mathcal{A}$, $\mathbbm{1}$ being as before the indicator function, and furthermore
	$$
	\mathbb{P}\left\{v^{\varepsilon} \neq v^{\varepsilon, N}\right\} \leq \mathbb{P}\left\{\frac{\varepsilon}{2}\int_{-\infty}^{\infty}\|v^{\varepsilon}(s)\|_{H_0}^{2} \mathrm{d} s \geq N\right\} \leq \frac{2 M+\delta}{N}.
	$$
	This observation implies that the right side of inequality (\ref{v}) is at most
	$$\mathbb{E}\left(\frac{\varepsilon}{2} \int_{-\infty}^{\infty}\|v^{\varepsilon,N}(s)\|_{H_0}^{2} \mathrm{d} s+h \circ \mathcal{G}^{\varepsilon}\left(W(\cdot)+ \int_{-\infty}^{.} v^{\varepsilon, N}(s) \mathrm{d} s\right)\right)-\frac{2 M(2 M+\delta)}{N}-\delta.$$
	Hence it suffices to prove with $v^{\varepsilon}(s)$ replaced by $v^{\varepsilon, N}(s)$. This proves the claim.
	
	Henceforth we will assume that \eqref{4.6} holds. Pick a subsequence along which $\tilde{v^{\varepsilon}}:=\sqrt{\varepsilon}v^{\varepsilon}$ converges in distribution to $\tilde{v}$ as $S_{N}$ -valued random elements.
	
	We now have from Condition \ref{weakcondition} (i) that
	\begin{align*}
		&\liminf _{\varepsilon \rightarrow 0}\mathbb{E}\left(\frac{\varepsilon}{2} \int_{-\infty}^{\infty}\|v^{\varepsilon}(s)\|_{H_0}^{2} \mathrm{d} s+h \circ \mathcal{G}^{\varepsilon}\left(W(\cdot)+ \int_{-\infty}^{.} v^{\varepsilon}(s) \mathrm{d} s\right)\right)\\
		=&\liminf _{\varepsilon \rightarrow 0} \mathbb{E}\left(\frac{1}{2} \int_{-\infty}^{\infty}\|\tilde{v}^{\varepsilon}(s)\|_{H_0}^{2} \mathrm{d} s+h \circ \mathcal{G}^{\varepsilon}\left(W(\cdot)+\frac{1}{\sqrt{\varepsilon}} \int_{-\infty}^{.} \tilde{v}^{\varepsilon}(s) \mathrm{d} s\right)\right) \\
		\geq& \mathbb{E}\left(\frac{1}{2} \int_{-\infty}^{\infty}\|\tilde{v}(s)\|_{H_0}^{2} \mathrm{d} s+h\left(\mathcal{G}^{0}\left(\int_{-\infty} \tilde{v}(s) \mathrm{d} s\right)\right)\right) \\
		\geq& \inf _{\left\{(x, v) \in E \times L^2\left(\mathbb{R};H_0\right): x=\mathcal{G}^{0}\left(\int_{-\infty}^{.} \tilde{v}(s) \mathrm{d} s\right)\right\}}\left\{\frac{1}{2} \int_{-\infty}^{\infty}\|\tilde{v}(s)\|_{H_0}^{2} \mathrm{d} s+h(x)\right\} \\
		\geq& \inf _{x \in E}\{I(x)+h(x)\}.
	\end{align*}
	This completes the proof of the lower bound.
	
	\textbf{(Upper bound)}
	Since $h$ is bounded $\inf _{x \in E}\{I(x)+h(x)\}<\infty$. Let $\delta>0$ be arbitrary, and let $\tilde{x} \in E$ be such that
	$$
	I\left(\tilde{x}\right)+h\left(\tilde{x}\right) \leq \inf _{x \in E}\{I(x)+h(x)\}+\delta / 2.
	$$
	Choose $\tilde{v} \in L^2(\mathbb{R};H_0)$ such that
	$$
	\frac{1}{2} \int_{-\infty}^{\infty}\|\tilde{v}(t)\|_{H_0}^{2} \mathrm{d} t \leq I\left(\tilde{x}\right)+\delta / 2
	$$
	and
	$$
	\tilde{x}=\mathcal{G}^{0}\left(\int_{-\infty} \tilde{v}(s) \mathrm{d} s\right).
	$$
	By using the variational presentation, for bounded and continuous functions $h$, we could have
	$$
	\begin{aligned}
		&\limsup _{\varepsilon \rightarrow 0}-\varepsilon \log \mathbb{E}\left(\exp \left\{-h\left(X^{\varepsilon}(t,-\infty,\omega)x_0\right) / \varepsilon\right\}\right) \\
		=&\limsup _{\varepsilon \rightarrow 0} \inf _{v \in \mathcal{A}} \mathbb{E}\left(\frac{1}{2} \int_{-\infty}^{\infty}\|v(s)\|_{H_0}^{2} \mathrm{d} s+h \circ \mathcal{G}^{\varepsilon}\left(W(\cdot)+\frac{1}{\sqrt{\varepsilon}} \int_{-\infty} v(s) \mathrm{d} s\right)\right) \\
		\leq &\limsup _{\varepsilon \rightarrow 0} \mathbb{E}\left(\frac{1}{2} \int_{-\infty}^{\infty}\|\tilde{v}(s)\|_{H_0}^{2} \mathrm{d} s+h \circ \mathcal{G}^{\varepsilon}\left(W(\cdot)+\frac{1}{\sqrt{\varepsilon}} \int_{-\infty} \tilde{v}(s) \mathrm{d} s\right)\right) \\
		=&\frac{1}{2} \int_{-\infty}^{\infty}\|\tilde{v}(s)\|_{H_0}^{2} \mathrm{d} s+\limsup _{\varepsilon \rightarrow 0} \mathbb{E}\left(h \circ \mathcal{G}^{\varepsilon}\left(W(\cdot)+\frac{1}{\sqrt{\varepsilon}} \int_{-\infty} \tilde{v}(s) \mathrm{d} s\right)\right) \\
		\leq& I\left(\tilde{x}\right)+\delta / 2+\limsup _{\varepsilon \rightarrow 0} \mathbb{E}\left(h \circ \mathcal{G}^{\varepsilon}\left(W(\cdot)+\frac{1}{\sqrt{\varepsilon}} \int_{-\infty} \tilde{v}(s) \mathrm{d} s\right)\right) .
	\end{aligned}
	$$
	Now from Condition \ref{weakcondition} (i), as $\varepsilon \rightarrow 0$
	$$
	\mathbb{E}\left(h \circ \mathcal{G}^{\varepsilon}\left(W(\cdot)+\frac{1}{\sqrt{\varepsilon}} \int_{-\infty} \tilde{v}(s) \mathrm{d} s\right)\right),
	$$
	converges to $h\left(\mathcal{G}^{0}\left(\int_{-\infty}^{\cdot} \tilde{v}(s) \mathrm{d} s\right)\right)=h\left(\tilde{x}\right)$. Thus
	\begin{equation*}
		\limsup _{\varepsilon \rightarrow 0}-\varepsilon \log \mathbb{E}\left(\exp \left\{-h\left(X^{\varepsilon}(t,-\infty,\omega)x_0\right) / \varepsilon\right\}\right)\leq\inf _{x \in E}\{I(x)+h(x)\}+\delta.
	\end{equation*}
	Since $\delta$ is arbitrary, the proof is complete.
	The Condition \ref{weakcondition} (ii) guaranteed that $I$ is a good rate function.
\end{proof}
\section*{Acknowledgments}
\addcontentsline{toc}{section}{Acknowledgments}
We are very grateful to Professor Z. Dong for his help and suggestions, and also to Dr. W.L. Zhang for participating in the discussion.
Y. Liu appreciates  Professor H.Z. Zhao and Professor C.R.  Feng for their discussion about random stationary solutions, random periodic solutions, etc.  Y. Liu thanks to Mr. C. Bai for their discussion about the random Hopf model.

Y. Liu is supported by CNNSF (No. 11731009, No.11926327) and Center for Statistical Science, PKU. Z.H. Zheng is supported by CNNSF (No.12031020, No.12090014), the Key Lab. of Random Complex Structures and Data Sciences, CAS and National Center for Mathematics and Interdisplinary Sciences, CAS.


\section*{References}
\addcontentsline{toc}{section}{References}
\begin{thebibliography}{1}
	\bibitem[1]{LA98}L. Arnold, Random dynamical systems, Springer Monographs in Mathematics, Springer-Verlag, Berlin, 1998.
	
	\bibitem[2]{Bai19} C. Bai,  Some problems of ergodic properties of Hopf Model with perturbation, Master's thesis, Peking University, 2019.	
	\bibitem[3]{BD98}M. Bou\'{e}, P. Dupuis, A variational representation for certain functionals of Brownian motion, Ann. Probab. 26 (4)
	(1998) 1641-1659.
	\bibitem[4]{BC17} Z. Brzezniak, S. Cerrai, Large deviations principle for the invariant measures of the 2D stochastic Navier-Stokes
	equations on a torus, J. Funct. Anal. 273 (6) (2017) 1891-1930.
	\bibitem[5]{BD01}A. Budhiraja, P. Dupuis, A variational representation for positive functionals of infinite dimensional Brownian
	motion, Probab. Math. Statist. 20 (1) (2000) 39-61.
	
	\bibitem[6]{CR05}  S. Cerrai, M. R\"ockner, Large deviations for invariant measures of stochastic reaction-diffusion systems with multiplicative noise and non-Lipschitz reaction term, Ann. Inst. H. Poincar\'{e} Probab. Statist. 41 (1) (2005) 69-105.
	
	\bibitem[7]{DZ07} A. Dembo, O. Zeitouni, Large deviations techniques and applications, Vol. 38 of Stochastic Modelling and Applied
	Probability, Springer-Verlag, Berlin, 2010, corrected reprint of the second (1998) edition.
	\bibitem[8]{DS84}J.D. Deuschel, D.W. Stroock, Large deviation, Springer-Verlag, Universitext, New York, 1984.
	
	\bibitem[9]{DE97}P. Dupuis, R.S. Ellis, A weak convergence approach to the theory of large deviations, Wiley Series in Probability and
	Statistics: Probability and Statistics, John Wiley \& Sons, Inc., New York, 1997, a Wiley-Interscience Publication.
	
	
	
	\bibitem[10]{FZ15}C.R. Feng, Y. Liu, H.Z. Zhao, Numerical approximation of random periodic solutions of stochastic differential
	equations, Z. Angew. Math. Phys. 68 (5) (2017) 68-199.
	\bibitem[11]{FQZ21}C.R. Feng, B.Y. Qu, H.Z. Zhao, Random quasi-periodic paths and quasi-periodic measures of stochastic differential
	equations, J. Differential Equations 286 (2021) 119-163.
	\bibitem[12]{FZ12}C.R. Feng, H.Z. Zhao, Random periodic solutions of SPDEs via integral equations and Wiener-Sobolev compact
	embedding, J. Funct. Anal. 262 (10) (2012) 4377-4422.
	\bibitem[13]{FZB11}C.R. Feng, H.Z. Zhao, B. Zhou, Pathwise random periodic solutions of stochastic differential equations, J.
	Differential Equations 251 (1) (2011) 119-149.
	
	
	\bibitem[14]{FW12} M.I. Freidlin, A.D. Wentzell, Random perturbations of dynamical systems, 3rd Edition, Vol. 260 of Grundlehren
	der mathematischen Wissenschaften [Fundamental Principles of Mathematical Sciences], Springer, Heidelberg,
	2012, translated from the 1979 Russian original by Joseph Sz\"ucs.
	
	\bibitem[15]{H56}E. Hopf, Repeated branching through loss of stability, an example. Proceedings of the conference on differential equations (dedicated to A. Weinstein), University of Maryland Book Store, College Park, Md., 1956, 49056.     Also in
	Selected works of Eberhard Hopf with commentaries. Edited by Cathleen S. Morawetz, James B. Serrin and Yakov G. Sinai. American Mathematical Society, Providence, RI, 2002.
	\bibitem[16]{KS88} I. Karatzas, S.E. Shreve, Brownian motion and stochastic calculus, 2nd Edition, Vol. 113 of Graduate Texts in
	Mathematics, Springer-Verlag, New York, 1991.
	
	\bibitem[17]{KF57} A.N. Kolmogorov, S.V. Fomin, Elements of the theory of functions and functional analysis. Vol. 1. Metric and
	normed spaces, Graylock Press, Rochester, N.Y., 1957, translated from the first Russian edition by Leo F. Boron.
	
	\bibitem[18]{LW15} W. Liu, M. R\"ockner, Stochastic partial differential equations: an introduction, Universitext, Springer, Cham, 2015.
	\bibitem[19]{LZ09} Y. Liu, H. Z. Zhao, Representation of pathwise stationary solutions of stochastic burgers’ equations, Stoch. Dynam.
	9 (04) (2009) 613-634.
	\bibitem[20]{DM15}D. Martirosyan, Large deviations for stationary measures of stochastic nonlinear wave equations with smooth white
	noise, Comm. Pure Appl. Math. 70 (9) (2017) 1754-1797.
	
	\bibitem[21]{M99}J. C. Mattingly, Ergodicity of 2d navier stokes equations withrandom forcing and large viscosity, Commun. Math.
	Phys. 206 (2) (1999) 273-288.
	\bibitem[22]{SZZ08} S.-E.A. Mohammed, T.S. Zhang, H.Z. Zhao, The stable manifold theorem for semilinear stochastic evolution
	equations and stochastic partial differential equations, Mem. Amer. Math. Soc. 196 (917) (2008) 1-105.
	\bibitem[23]{BG20} B. Nikolay, G. Massimiliano, A variational method for {$\Phi^4_3$}
	, Duke Math. J. 169 (17) (2020) 3339-3415.
	\bibitem[24]{PR07}C.  Pr\'{e}v\^{o}t,  M.  R\"{o}ckner, A concise course on stochastic partial differential equations, Vol. 1905 of Lecture Notes in
	Mathematics, Springer, Berlin, 2007.
	
	\bibitem[25]{SP06}S. Sritharan, P. Sundar, Large deviations for the two-dimensional Navier-Stokes equations with multiplicative noise,
	Stochastic Process. Appl. 116 (11) (2006) 1636-1659.
	
	
	
	\bibitem[26]{Z09}X.C. Zhang, A variational representation for random functionals on abstract wiener spaces, J. Math. Kyoto. U
	49 (3) (2009) 475-490.
	\bibitem[27]{ZZ17}H.Z. Zhao, Z.H. Zheng, Random periodic solutions of random dynamical systems, J. Differential Equations 246 (5)
	(2009) 2020-2038.
\end{thebibliography}
\end{document}